\documentclass{amsart}
\usepackage{mathrsfs}

\usepackage{amsfonts,amssymb}
 \usepackage{xspace,xcolor}
 \usepackage[all,cmtip]{xy}
 \usepackage{graphicx,enumerate}
 \usepackage{tikz}
 \usetikzlibrary{cd}
 \usepackage{ifthen}

\usepackage{lscape}

 \usepackage{color}
\usepackage[colorlinks,citecolor=blue]{hyperref}

\theoremstyle{plain}
\newtheorem{theorem}{Theorem}[section]

\newtheorem{lemma}[theorem]{Lemma}
\newtheorem{corollary}[theorem]{Corollary}
\newtheorem{conjecture}[theorem]{Conjecture}
\newtheorem{proposition}[theorem]{Proposition}
\newtheorem{prop}[theorem]{Proposition}
\newtheorem{prop-defn}[theorem]{Proposition-Definition}
\newtheorem{claim}[theorem]{Claim}
\newtheorem*{theorem:main}{Main Theorem} 
\theoremstyle{definition}
\newtheorem{defn}[theorem]{Definition}
\newtheorem{definition}[theorem]{Definition}
\newtheorem{remark}[theorem]{Remark}

\newtheorem{question}[theorem]{Question}
\newtheorem*{remark*}{Remark}
\newtheorem*{remarks*}{Remarks}

\newtheorem{example}[theorem]{Example}

\newcommand{\Mod}{\mathrm{Mod}}

\newcommand{\cl}[1]{\marginpar{\color{red}\tiny #1 --cl}}

\newcommand{\diam}{\mathrm{diam}}
\newcommand{\Teich}{\mathcal T}

\newcommand{\C}{\mathcal C}
\newcommand{\Aff}{\rm Aff}

\newcommand{\PGL}{\rm PGL}
\newcommand{\PML}{\mathcal PML}
\newcommand{\CX}{{\mathcal X}}
\newcommand{\CW}{{\mathcal W}}

\newcommand{\CB}{{\mathcal B}}
\newcommand{\CK}{{\mathcal K}}
\newcommand{\tree}{\mathrm{tree}}
\newcommand{\bTheta}{{\bf\Theta}}
\newcommand{\CP}{{\mathcal P}}
\newcommand{\Sat}{\mathrm{Sat}}

\newcommand{\bXi}{{\bf\Xi}}
\newcommand{\vtx}{\mathcal{V}}

\newcommand{\SL}{\rm SL}

\newcommand{\kvtx}{{\mathcal K}}
\DeclareMathOperator{\QI}{QI}
\DeclareMathOperator{\Isom}{Isom}
\newcommand{\fib}{{\mathrm{fib}}}
\newcommand{\Isomfib}{\mathrm{Isom_{fib}}}

\newcommand{\pow}{{\bf P}}
\newcommand{\duaug}[2]{{#1}^{+{#2}}}
\newcommand{\link}{\mathrm{Lk}}
\newcommand{\sat}{\mathrm{Sat}}
\newcommand{\nest}{\sqsubseteq}
\newcommand{\orth}{\bot}
\newcommand{\transverse}{\pitchfork}
\usepackage{mathabx}
\newcommand{\propnest}{\sqsubsetneq}

\newcommand{\sdl}{\sigma}

\theoremstyle{plain}
\newtheorem*{qi-theorem}{Theorem \ref{thm:qirigid intro}}

\title[Veech extensions II: Hierarchical hyperbolicity and QI-rigidity]{Extensions of Veech groups II:  Hierarchical hyperbolicity and quasi-isometric rigidity}
\author[Dowdall]{Spencer Dowdall}
	\address{Department of Mathematics, Vanderbilt University, Nashville, TN}
	\email{spencer.dowdall@vanderbilt.edu}
\author[Durham]{Matthew G. Durham}
	\address{Department of Mathematics, University of California, Riverside, CA}
	\email{mdurham@ucr.edu}
\author[Leininger]{Christopher J. Leininger}
	\address{Department of Mathematics, Rice University, Houston, TX}
	\email{cjl12@rice.edu}
\author[Sisto]{Alessandro Sisto}
	\address{Department of Mathematics, Heriot-Watt University, Edinburgh, UK}
	\email{a.sisto@hw.ac.uk}

\begin{document}

\begin{abstract} 
We show that for any lattice Veech group in the mapping class group $\Mod(S)$ of a closed surface $S$, the associated $\pi_1S$--extension group is a hierarchically hyperbolic group.  As a consequence, we prove that any such extension group is quasi-isometrically rigid.
\end{abstract}

\maketitle


\section{Introduction}

This paper studies geometric properties of surface group extensions and how these relate to their defining subgroups of mapping class groups.
Let $S$ be a closed, connected, oriented surface of genus at least $2$.
Recall that a $\pi_1 S$--extension of a group $G$ is a short exact sequence of the form
\[1 \to \pi_1 S \to \Gamma \to G\to 1. \]
Such  extensions are in bijective correspondence with monodromy homomorphisms from $G$ to the extended mapping class group $\Mod^\pm(S) \cong \mathrm{Out}(\pi_1 S)$ of the surface.  Alternatively, 
these groups $\Gamma$ are precisely the fundamental groups of $S$--bundles.

Many advances in the study of mapping class groups have been motivated by a longstanding but incomplete analogy between hyperbolic space $\mathbb{H}^n$ and the Teichm\"uller space $\Teich(S)$ of a surface.
In the theory of Kleinian groups, a discrete group of isometries of $\mathbb{H}^n$ is {convex cocompact} if it acts cocompactly on an invariant, convex subset. 
Farb and Mosher \cite{FM:CC} adapted this notion to mapping class groups by defining a subgroup $G\le \Mod^\pm(S)$ to be \emph{convex cocompact} if it acts cocompactly on a quasi-convex subset of $\Teich(S)$.
This has proven to be a fruitful concept with many interesting connections to, for example, the intrinsic geometry of the mapping class group \cite{DT15,BBKL-undistored_purely_pA}, and its actions on the curve complex and the boundary of Teichm\"uller space \cite{KL:CC}.
Most importantly, the work of Farb--Mosher \cite{FM:CC} and Hamenst\"adt \cite{hamenstadt:WHSG} remarkably shows that an extension $\Gamma$ as above is word hyperbolic if and only if the associated monodromy $G\to \Mod^\pm(S)$ has finite kernel and convex cocompact image (see also \cite{MjSardar}).

For Kleinian groups, convex cocompactness is a special case of a more prevalent phenomenon called \emph{geometric finiteness}, which roughly amounts to acting cocompactly on a convex subset minus horoballs invariant by parabolic subgroups. 
In \cite{Mosher:Problems}, Mosher suggested
this notion should have an analogous framework in mapping class groups that would extend the geometric connection with surface bundles to a larger class of examples.
The prototypical candidates for geometric finiteness are the \emph{lattice Veech subgroups}; these are special punctured-surface subgroups of $\Mod(S)$ that arise naturally in the context of Teichm\"uller dynamics and whose corresponding $S$--bundles are amenable to study via techniques from flat geometry.  

Our prequel paper \cite{DDLSI} initiated an analysis of the $\pi_1 S$--extensions associated to lattice Veech subgroups, with the main result being that each such extension $\Gamma$ admits an action on a hyperbolic space $\hat E$ that captures much of the geometry of $\Gamma$.
Building on that work, the first main result of this paper is the following, which provides a concrete answer to \cite[Problem 6.2]{Mosher:Problems} for lattice Veech groups.

\begin{theorem} \label{T:main simple} For any lattice Veech subgroup $G < \Mod(S)$, the associated  $\pi_1S$--extension group $\Gamma$ of $G$ is a hierarchically hyperbolic group.
\end{theorem}

Hierarchical hyperbolicity means that in fact all the geometry of $\Gamma$ is robustly encoded by hyperbolic spaces. This is exactly the sort of relaxed hyperbolicity for $\pi_1 S$--extensions that one hopes should follow from a good definition of geometric finiteness in $\Mod(S)$. Thus Theorem~\ref{T:main simple} suggests a possible general theory of geometric finiteness, which we expound upon in \S\ref{section:musings-on-geom-finte} below.

Hierarchical hyperbolicity has many strong consequences, some of which are detailed in \S\ref{subsec:HHSery} below. 
It also enables, via tools from \cite{BHS:quasiflats}, the proof of our second main result, 
which answers \cite[Problem 5.4]{Mosher:Problems}:

\begin{theorem} \label{T:qir simple}
For any lattice Veech group $G < \Mod(S)$, the associated $\pi_1S$--extension group $\Gamma$ of $G$ is quasi-isometrically rigid.
\end{theorem}

In \cite{Mosher:Problems}, Mosher in fact suggests an alternate approach to proving quasi-isometric rigidity that culminates in the formulation of Problem 5.4 of \cite{Mosher:Problems} as an equivalent condition in this case. Both our proof and this alternate approach share a common key step of showing that quasi-isometries are coarsely fiber-preserving; for this we use tools from hierarchical hyperbolicity (see Proposition~\ref{prop:strip_bijection} below), whereas Mosher's approach uses ideas from coarse algebraic topology appealing to the fact that $G$ is virtually free (see \cite{FarbMosher-abelian}). In \S\ref{S:Mosher proof} we give a rough sketch that carries out Mosher's approach, drawing partly from his unpublished results in \cite{Mosher-fiber}, and leading to an alternate proof of Theorem~\ref{T:qir simple}.


It is our hope that the framework we develop for proving quasi-isometric rigidity for extension groups via hierarchical hyperbolicity will be applicable to a wider class of geometrically finite subgroups (\S\ref{section:musings-on-geom-finte}), including those which may not be virtually free.

The rest of this introduction gives a more in-depth treatment of these results while elaborating on the concepts of, and connections between, hierarchical hyperbolicity, extensions of Veech groups, quasi-isometric rigidity, and geometric finiteness.

\subsection{Hierarchical hyperbolicity}\label{subsec:HHSery}

The notion of hierarchical hyperbolicity was defined by Behrstock, Hagen, and Sisto \cite{BHS:HHS1} and motivated by the seminal work of Masur and Minsky \cite{MM:CC2}.
In short, it provides a framework and toolkit for understanding the coarse geometry of a space/group in terms of interrelated  hyperbolic pieces.
More precisely, a {\em hierarchically hyperbolic space (HHS)} structure on a metric space $X$ is a collection of hyperbolic spaces $\{\C(W)\}_{W \in \mathfrak S}$,
arranged in a hierarchical fashion, in which any pair are nested $\nest$, orthogonal $\orth$, or transverse $\transverse$,
along with Lipschitz projections to and between these spaces that together capture the coarse geometry of $X$.
A {\em hierarchically hyperbolic group (HHG)} is then an HHS structure on a group that is equivariant with respect to an appropriate action  on the union of hyperbolic spaces $\C(W)$.
See \S\ref{S:combinatorial HHS} for details or \cite{BHS:HHS1,BHS:HHS2,HHS_survey} for many examples and further discussion.

Showing that a space/group is a hierarchically hyperbolic gives access to several results regarding, for example, a coarse median structure and quadratic isoperimetric function \cite{Bow:rigidity_MCG,Bow:coarse_median}, asymptotic dimension \cite{BHS:asdim}, stable and quasiconvex subsets and subgroups \cite{ABD:stable, RST18}, quasiflats \cite{BHS:quasiflats}, bordifications and automorphisms \cite{DHS}, and quasi-isometric embeddings of nilpotent groups \cite{BHS:HHS1}.  In particular, the following is an immediate consequence of Theorem \ref{T:main simple}.

\begin{corollary}\label{cor:HHS consequences}
Let $G < \Mod(S)$ be any lattice Veech group and $\Gamma$ the associated $\pi_1S$--extension group.  Then:
\begin{enumerate}
\item $\Gamma$ has quadratic Dehn function \cite{Bow:coarse_median}.
\item\label{item:univ_acyl_action} $\Gamma$ is acylindrically hyperbolic and, moreover, its action on the $\nest$--maximal hyperbolic space in the hierarchy is a universal acylindrical action \cite{ABD:stable}.
\item $\Gamma$ is semihyperbolic and thus has solvable conjugacy problem \cite{DMS:stabcube, HHP:helly}.
\end{enumerate}
\end{corollary}

As discussed in \S\ref{sec:intro-HHG_on_Gamma} below, further information about $\Gamma$ can be gleaned from the specific HHG structure constructed in proving  Theorem~\ref{T:main simple}. We note that the {$\nest$--maximal hyperbolic space} of this structure, and thus the universal acylindrical action indicated in Corollary~\ref{cor:HHS consequences}(\ref{item:univ_acyl_action}), is simply 
the space $\hat E$ from \cite{DDLSI}.

\subsection{The HHG structure on $\Gamma$}
\label{sec:intro-HHG_on_Gamma}
In order to describe the HHG structure more precisely and explain its connection to quasi-isometric rigidity in Theorem \ref{T:qir simple}, we must first recall some of the structure 
of Veech groups and their extensions.
Let $G < \Mod(S)$ be a lattice Veech group and $\Gamma = \Gamma_G$ the associated extension group. 
First note that (up to finite index) $\Gamma$ is naturally the fundamental group of an $S$--bundle $\bar E/\Gamma$ over a compact surface with boundary (see \S\ref{sec:setup} for details and notation).  Each boundary component of $\bar E / \Gamma$ is virtually the mapping torus of a multi-twist on $S$, and is thus a {\em graph manifold}:  the tori in the JSJ decomposition are suspensions of the multi-twist curves.

Graph manifolds admit HHS structures \cite{BHS:HHS2} where the maximal hyperbolic space is the Bass--Serre tree dual to the JSJ decomposition, and all other hyperbolic spaces are either quasi-lines or quasi-trees (obtained by coning off the boundaries of the universal covers of the base orbifolds of the Seifert pieces).
The stabilizers of the vertices of the Bass--Serre trees are called \emph{vertex subgroups}, and are precisely the fundamental groups of the Seifert pieces of the JSJ decomposition.
We let $\vtx$ denote the disjoint union of the vertices of all Bass--Serre trees associated to the boundary components of the universal cover $\bar E$ of this $S$--bundle.  Given $v,w \in \vtx$, we say that these vertices are {\em adjacent} if they are connected by an edge in the same Bass-Serre tree.
 
The HHG structure on the extension group $\Gamma$ may now be described as follows:

\begin{theorem} \label{T:main full} Suppose $G<\Mod(S)$ is a lattice Veech group with extension group $\Gamma$ and let $\Upsilon_1,\ldots,\Upsilon_k < \Gamma$ be representatives of the conjugacy classes of vertex subgroups. 
Then $\Gamma$ admits an HHG structure with the following set of hyperbolic spaces and relations among them (ignoring those of diameter $\leq 2$):
\begin{enumerate}
\item The maximal hyperbolic space $\hat E$ is quasi-isometric to the Cayley graph of $\Gamma$ coned off along the cosets of $\Upsilon_1,\ldots,\Upsilon_k$ \cite{DDLSI}.
\item There is a quasi-tree $v^{qt}$ and a quasi-line $v^{ql}$, for each $v \in \vtx$, and:
\begin{enumerate}
\item For all $v \in \vtx$, $v^{qt} \orth v^{ql}$.
\item For all $v,w \in \vtx$, if $v$ and $w$ are adjacent, then $w^{ql}\orth v^{ql}$ and $w^{ql}\nest v^{qt}$.
\item All other pairs are transverse.
\end{enumerate}
\end{enumerate}
\end{theorem}

This description of the HHG structure readily leads to further consequences for $\Gamma$.
For example, the maximal number of infinite-diameter pairwise orthogonal hyperbolic spaces is evidently $2$. In view of \cite{BHS:HHS1,BHS:quasiflats}, 
we thus see that $\Gamma$ is as ``close to hyperbolic'' as possible in that its quasi-flats are at worst $2$--dimensional:

\begin{corollary} 
\label{cor:top-dim-flats}
Each top-dimensional quasi-flat in $\Gamma$ has dimension $2$ and is contained in a finite-radius neighborhood of finitely many cosets of vertex subgroups.
\end{corollary}

We note that quasiflats will be crucial for our proof of quasi-isometric rigidity, and we remark that the analogous statement for graph manifolds is due to Kapovich--Leeb \cite{KL:qi_decomposition}.

Recall that an element of a group is a \emph{generalized loxodromic} if it acts loxodromically under some acylindrical action on a hyperbolic space, and that a \emph{universal acylindrical action} on a hyperbolic space is one in which every generalized loxodromic acts loxodromically \cite{ABD:stable}.
It is shown in \cite{S-Morse-hypemb} that a generalized loxodromic element $g$ of a finitely generated group is necessarily \emph{Morse}, meaning that in any finite-valence Cayley graph for the group, any $(K,C)$--quasi-geodesic with endpoints in the cyclic subgroup $\langle g \rangle$ stays within controlled distance $M = M(K,C)$ of $\langle g \rangle$.  While being Morse is, in general, strictly weaker than being generalized loxodromic, these conditions are in fact equivalent in HHGs  \cite[Theorem B]{ABD:stable}. 

In the case of our extension group $\Gamma$, it follows from  Corollary~\ref{cor:HHS consequences}(\ref{item:univ_acyl_action}) that the generalized loxodromics and Morse elements are precisely those elements acting loxodromically on $\hat E$. In \cite[Theorem 1.1]{DDLSI} we characterized these elements in terms of the vertex subgroups of $\Gamma$, thus yielding the following:

\begin{corollary}
Let $\Gamma$ be a lattice Veech group extension with vertex subgroups $\Upsilon_1,\dots,\Upsilon_k$ as in Theorem~\ref{T:main full}. The following are equivalent for an infinite order element $\gamma\in \Gamma$:
\begin{itemize}
\item $\gamma$ is not conjugate into any of the vertex subgroups $\Upsilon_i$
\item $\gamma$ is a generalized loxodromic element of $\Gamma$
\item $\gamma$ is a Morse element of $\Gamma$.
\end{itemize}
\end{corollary}

\subsection{Quasi-isometric rigidity}

To state our rigidity theorem, first recall that $\Gamma$ is (up to finite index) the fundamental group of an $S$-bundle $\widebar{E}/\Gamma$ over a compact surface with boundary. Here $\bar E$ is a $\Gamma$--invariant truncation of the universal $\tilde S$--bundle over the Teichm\"uller disk stabilized by the Veech group $G$. In particular, $\bar E$ is quasi-isometric to $\Gamma$. Let $\Isom(\widebar{E})$ and $\QI(\widebar{E})$ denote the isometry and quasi-isometry groups of $\widebar{E}$, respectively, and let $\Isomfib(\widebar{E}) \leq \Isom(\widebar{E})$ denote the subgroup of isometries that map fibers to fibers.  

\begin{theorem}\label{thm:qirigid intro}
There is an allowable truncation $\bar E$ of $E$ such that the natural homomorphisms $\Isomfib(\bar E) \to \Isom(\bar E) \to \QI(\bar E) \cong \QI(\Gamma)$ are all isomorphisms, and $\Gamma  \le \Isom(\bar E) \cong \QI(\Gamma)$ has finite index.
\end{theorem}

This is an analog, and indeed was motivated by, Farb and Mosher's \cite{FM:SbF} theorem 
that in the case of a surface group extension $\Gamma_H$ associated to a Schottky subgroup $H$ of $\Mod(S)$, the natural homomorphism $\Gamma_H\to \QI(\Gamma_H)$ is injective with finite cokernel.
This rigidity also leads to the following strong \emph{algebraic} consequence:

\begin{corollary}
\label{cor:alg_qi-rigid_intro}
If $H$ is any finitely generated group quasi-isometric to $\Gamma$, then $H$ and $\Gamma$ are weakly commensurable.
\end{corollary}

In the statement, recall that two groups $H_1,H_2$ are \emph{weakly commensurable} if there are finite normal subgroups $N_i\lhd H_i$ so that the quotients $H_i/N_i$ have a pair of finite-index subgroups that are isomorphic to each other.

\subsection{Motivation and Geometric Finiteness}
\label{section:musings-on-geom-finte}

Before outlining the paper and providing some ideas about the proofs, we provide some speculative discussion.
For Kleinian groups---that is, discrete groups of isometries of hyperbolic $3$--space---the notion of geometric finiteness is important in the deformation theory of hyperbolic $3$--manifolds by the work of Ahlfors \cite{Ahlfors} and Greenberg \cite{Greenberg}.
While the definition has many formulations (see \cite{Marden,Maskit:bdy,thurston:GT,Bow:GF}), roughly speaking a group is geometrically finite if it acts cocompactly on a convex subset of hyperbolic $3$--space minus a collection of horoballs that are invariant by parabolic subgroups.  When there are no parabolic subgroups, geometric finiteness reduces to convex cocompactness: a cocompact action on a convex subset of hyperbolic $3$--space.

While there is no deformation theory for subgroups of mapping class groups,
Farb and Mosher \cite{FM:CC} introduced a notion of convex cocompactness for $G < \Mod(S)$ in terms of the action on Teichm\"uller space $\Teich(S)$. 
Their definition requires that $G$ acts cocompactly on a quasi-convex subset for the Teichm\"uller metric, while Kent and Leininger later proved a variety of equivalent formulations analogous to the Kleinian setting \cite{KL:survey,KL:CC,KL:uniform}. Farb and Mosher proved that convex cocompactness is equivalent to hyperbolicity of the associated extension group $\Gamma_G$ (with monodromy given by inclusion) when $G$ is virtually free. This equivalence was later proven in general by Hamenst\"adt \cite{hamenstadt:WHSG} (see also Mj--Sardar \cite{MjSardar}), though at the moment the only known examples are virtually free.

The coarse nature of Farb and Mosher's formulation reflects the fact that the Teichm\"uller metric is far less well-behaved than that of hyperbolic $3$--space. Quasi-convexity in the definition is meant to help with the lack of nice {\em local} behavior of the Teichm\"uller metric. It also helps with the {\em global} lack of Gromov hyperbolicity (see Masur--Wolf \cite{masur:NH}), as cocompactness of the action ensures that the quasi-convex subset in the definition is Gromov hyperbolic (see Kent--Leininger \cite{KL:CC}, Minsky \cite{Minsky:quasiproj}, and Rafi \cite{rafi:HT}).  

The inclusion of reducible/parabolic mapping classes in a subgroup $G< \Mod(S)$ brings the {\em thin parts} of $\Teich(S)$ into consideration; these subspaces contain higher rank quasi-flats and even exhibit aspects of positive curvature (see Minsky~\cite{Minsky:product}).  This is a main reason why extending the notion of convex cocompactness to geometric finiteness is complicated. These complications are somewhat mitigated in the case of lattice Veech groups. Such subgroups are stabilizers of isometrically and totally geodesically embedded hyperbolic planes, called {\em Teichm\"uller disks}, that have finite area quotients.  Thus, the intrinsic hyperbolic geometry agrees with the extrinsic Teichm\"uller geometry, and as a group of isometries of the hyperbolic plane, a lattice Veech group {\em is} geometrically finite. This is why these subgroups serve as a test case for geometric finiteness in the mapping class group.   This is also why a {\em subgroup} of a Veech group is convex cocompact in $\Mod(S)$ if and only if it is convex cocompact as a group of isometries of the hyperbolic plane (which also happens if and only if it is finitely generated and contains no parabolic elements).

The action of $\Mod(S)$ on the {\em curve graph}, which is Gromov hyperbolic by work of Masur--Minsky \cite{MM:CC1}, provides an additional model for these considerations. Specifically, convex cocompactness is equivalent to the orbit map to the curve graph $\C(S)$ being a quasi-isometric embedding with respect to the word metric from a finite generating set (see Kent--Leininger \cite{KL:CC} and Hamenst\"adt \cite{hamenstadt:WHSG}). 
Viewing geometric finiteness as a kind of ``relative convex cocompactness'' for Kleinian groups suggests an interesting connection with the curve complex formulation.
The connection is best illustrated by the following theorem of Tang \cite{tang2019affine}. 

\begin{theorem}[Tang] \label{T:Tang} For any lattice Veech group $G < \Mod(S)$ stabilizing a Teichm\"uller disk $D \subset \Teich(S)$, there is a $G$--equivariant quasi-isometric embedding $D^{el} \to \C(S)$,
where $D^{el}$ is the path metric space obtained from $D$ by coning off the $G$--invariant family of horoballs in which $D$ ventures into the thin parts of $\Teich(S)$.
\end{theorem}

Farb \cite{Farb:rel} showed that non-cocompact lattices in the group of isometries of hyperbolic space are {\em relatively hyperbolic} relative to the parabolic subgroups.  For Veech groups, the space $D^{el}$ is quasi-isometric to the (hyperbolic) coned off Cayley graph, illustrating (part of) the relative hyperbolicity of $G$.  We thus propose a kind of ``qualified'' notion of geometric finiteness with this in mind:

\begin{definition}[Parabolic geometric finiteness]
 \label{D:par geom finite} A finitely generated subgroup $G< \Mod(S)$ is {\em parabolically geometrically finite} if $G$ is relatively hyperbolic, relative to a (possibly trivial) collection of subgroups $\mathcal H = \{H_1,\ldots,H_k\}$, and
\begin{enumerate}
\item $H_i$ contains a finite index, abelian subgroup consisting entirely of multitwists, for each $1 \leq i \leq k$; and
\item the coned off Cayley graph $G$--equivariantly and quasi-isometrically embeds into $\C(S)$.
\end{enumerate}
\end{definition} 

When $\mathcal H = \{ \{id\} \}$, we note that the condition is equivalent to $G$ being convex cocompact.  By Theorem~\ref{T:Tang}, lattice Veech groups are parabolically geometrically finite.  In fact, Tang's result is more general and implies that any finitely generated Veech group satisfies this definition.  These examples are all virtually free, but other examples include the combination subgroups of Leininger--Reid \cite{Leininger-Reid}, which are isomorphic to fundamental groups of closed surfaces of higher genus, and free products of higher rank abelian groups constructed by Loa \cite{Loa}.

In view of Theorem \ref{T:main simple}, one might formulate the following.

\begin{conjecture}
 Let $G< \Mod(S)$ be parabolically geometrically finite. Then the $\pi_1S$--extension group $\Gamma$ of $G$ is a hierarchically hyperbolic group.
\end{conjecture}

We view Definition~\ref{D:par geom finite} as only a qualified formulation 
because there are many subgroups of $\Mod(S)$ that are not  relatively hyperbolic but are nevertheless candidates for being geometrically finite in some sense. It is possible that there are different types of geometric finiteness for subgroups of mapping class groups, with Definition~\ref{D:par geom finite} being among the most restrictive. Other notions might include an HHS structure on the subgroup which is compatible with the ambient one on $\Mod(S)$ (e.g., hierarchical quasiconvexity \cite{BHS:HHS2}).
From this perspective, some candidate subgroups that may be considered geometrically finite include:

\begin{itemize}
\item the whole group $\Mod(S)$;
\item multi-curve stabilizers;
\item the right-angled Artin subgroups of mapping class groups constructed in \cite{CLM:raag_in_MCG, K:raag_in_MCG, runnels2020effective};
\item free and amalgamated products of other examples.
\end{itemize}

\begin{question} For each example group $G\le \Mod(S)$ above, is the associated extension $\Gamma_G$ a hierarchically hyperbolic group?
\end{question}

We note that the answer is `yes'  for the first example, since the extension group is the mapping class group of the surface $S$ with a puncture.  Moreover, since our work on this subject first appeared, Russell \cite{russell2021extensions} addressed the second example by proving  extensions of multicurve stabilizers are hierarchically hyperbolic groups.

\bigskip

\subsection{Outline and proofs} \label{subsec:outline}
Let us briefly outline the paper and comment on the main structure of the proofs. 
In \S\ref{sec:setup}  we review necessary background material and introduce the objects and notation that will be used throughout the paper.  In particular, we define the spaces $E$ and $\bar E$, the latter being a quasi-isometric model for the Veech group extension $\Gamma$, as well as the hyperbolic collapsed space $\hat E$.  All of these were constructed in \cite{DDLSI}. 

In \S\S\ref{S:projections}--\ref{S:combinatorial HHS} we prove that the extension group $\Gamma$ is hierarchically hyperbolic by utilizing a combinatorial criterion from \cite{comb_HHS}. Besides hyperbolicity of $\hat E$, the other hard part of the criterion is an analogue of Bowditch's fineness condition from the context of relative hyperbolicity. Its geometric interpretation is roughly that two cosets of vertex subgroups as above have bounded coarse intersection, aside from the ``obvious'' exception when the cosets correspond to vertices of the same Bass--Serre tree within distance 2 of each other. 
To this end, in \S\ref{S:projections} we associate to each vertex $v\in \vtx$ a {spine bundle} $\bTheta^v \subset \bar E$, which corresponds to a Seifert piece of the JSJ decomposition of the peripheral graph manifold, along with a pair of hyperbolic spaces $\mathcal K^v$ and $\bXi^v$ that will figure into the HHS structure on $\Gamma$. The space $\mathcal K^v$ is obtained via a quasimorphism constructed using the Seifert fibered structure following ideas in forthcoming work of the fourth author with Hagen, Russell, and Spriano \cite{SquidGames}, while $\bXi^v$ is coarsely obtained by coning off boundary components of the universal covers of the base $2$--orbifold of this Seifert fibered manifold. We then appeal to the flat geometry of the fibers of $E$ to construct and study certain projection maps
\begin{center}
\begin{tikzpicture}[scale = 1]
\node (E) at (0,1.1) {$\bar E$};
\node (theta) at (0,0) {$\bTheta^v$};
\node (K) at (-2.5,0) {$\mathcal{K}^v$};
\node (Xi) at (2.5,0) {$\bXi^v$};
\draw[->] (E) -- node[right] {$\Pi^v$} (theta);
\draw[->] (theta) -- node[midway,fill=white] {$\lambda^v$} (K);
\draw[->] (theta) -- node[midway,fill=white] {$i^v$} (Xi);
\path (E) edge[bend right=10,->] node[midway,fill=white] {$\Lambda^v$} (K);
\path (E) edge[bend left=10,->] node[midway,fill=white] {$\xi^v$} (Xi);
\end{tikzpicture}
\end{center}
and prove that various pairs of subspaces of $\bar E$ have bounded projection onto each other (Proposition \ref{prop:R-projections}).

In \S\ref{S:combinatorial HHS}, we begin assembling the combinatorial objects necessary to apply the HHG criterion from \cite{comb_HHS}, which involves both combinatorial and geometric aspects.  The first step involves the construction of a natural flag complex $\mathcal X$ containing the union of the Bass-Serre trees, together with appropriate ``subjoins''  with the union of all $\mathcal K^v$, over $v \in \vtx$.  Next, we use the geometry of $\bar E$ to construct a certain graph $\mathcal W$ whose vertices are maximal simplices of $\mathcal X$ and on which $\Gamma$ acts metrically properly and coboundedly.  The remainder of this section is devoted to verifying the necessary combinatorial conditions as well as translating the facts about $\mathcal K^v$ and $\bXi^v$ and the projections described above into proofs of the necessary geometric conditions.  We note that in the combinatorial HHG setup, the complex $\mathcal X$ comes with its own hierarchy projections between the induced hyperbolic spaces (Definitions~\ref{defn:orth}--\ref{defn:projections}), which may be different than the projections to $\mathcal K^v$ and $\bXi^v$.

In \S\ref{S:QI rigidity} we prove our QI-rigidity result Theorem \ref{thm:qirigid intro}.
The starting point is the hierarchical hyperbolicity of $\Gamma$ provided by Theorem \ref{T:main full}, as it gives access to the results and arguments in \cite{BHS:quasiflats} about the preservation of quasi-isometrically embedded flats.
Every collection of pairwise orthogonal hyperbolic spaces in an HHG determines a natural product subspace, with the maximal \emph{standard} quasi-isometrically embedded flats (or orthants) 
arising inside such subspaces as products of quasi-lines in a maximal collection of pairwise orthogonal hyperbolic spaces of the HHG.  
Theorem A of \cite{BHS:quasiflats} states that a quasi-isometry of an HHS preserves the structure of its quasi-flats and takes any maximal quasi-flat within bounded Hausdorff distance of the union of standard maximal orthants. 
The maximal quasi-flats in the HHG structure on $\widebar{E}$, namely the $2$--dimensional flats indicated in Corollary~\ref{cor:top-dim-flats}, are encoded by certain \emph{strip bundles} that, roughly, correspond to flats in the peripheral graph manifolds.  We use the preservation of the maximal quasi-flats to derive coarse preservation of these strip bundles, which we then upgrade to coarse preservation of the fibers (\S\ref{sec:HHS-quasi-flats}). By using tools of flat geometry from \cite{BankLein,DELS}, we then show any quasi-isometry induces an \emph{affine} homeomorphism of any fiber to itself (\S\S\ref{sec:QI_of_barE_to_QI_of_E_0}--\ref{sec:qis_to_affine}) and moreover that this assignment is injective (\S\ref{sec:injectivity_of_A}). Finally, we show this association is an isomorphism by proving (\S\ref{sec:affine_to_isometry}) that every affine homeomorphism of a fiber induces an isometry and hence quasi-isometry of $\bar E$. Quasi-isometric rigidity and its algebraic consequence Corollary~\ref{cor:alg_qi-rigid_intro} are then easily obtained in \S\ref{sec:finish_proof_of_rigidity}.

\subsection*{Acknowledgments} The authors would like to thank MSRI and its Fall 2016 program on {\em Geometric Group Theory}, where this work began. We also gratefully acknowledge NSF grants DMS 1107452, 1107263, 1107367 (the GEAR Network) for supporting travel related to this project.
Dowdall was partially supported by NSF grants DMS-1711089 and DMS-2005368. Durham was partially supported by NSF grant DMS-1906487.  Leininger was partially supported by NSF grants DMS-1510034, DMS-1811518, and DMS-2106419.
Sisto  was partially supported by the Swiss National  Science Foundation (grant \#182186).  The authors would like to thank the anonymous referee for their very helpful comments on the first version of this paper.

\section{Setup: The groups and spaces}
\label{sec:setup}
 
Here we briefly recall the basic set up from \cite{DDLSI} which we will use throughout the remainder of the paper.  We refer the reader to Sections 2 and 3 of that paper for details and precise references.

\subsection{Flat metrics and Veech groups} \label{S:flat metrics Veech groups}
Fix a closed surface of genus at least $2$, a complex structures $X_0$ (viewed as a point in the Teichm\"uller space $\Teich(S)$), and a nonzero holomorphic quadratic differential $q$ on $(S,X_0)$.  Integrating a square root of $q$ determines preferred coordinates on $(S,X_0)$ for $q$ which defines a translation structure (in the complement of the isolated zeros of $q$).  We also write $q$ for the associated flat metric defined by the half-translation structure (though the metric only determines the half-translation structure or quadratic differential up to a complex scalar multiple).  This metric is a non-positively curved Euclidean cone metric, with cone singularities at the zeros of $q$. 
The orbit of $(X_0,q)$ under the natural $\SL_2(\mathbb R)$ action on quadratic differentials projects to a {\em Teichm\"uller disk}, $D = D_q \subset \Teich(S)$, which we equip with its Poincar\'e metric $\rho$. The circle at infinity of $D$ is naturally identified with the projective space of directions, $\mathbb P^1(q)$, in the tangent space of any nonsingular point of $q$.  For $\alpha \in \mathbb P^1(q)$, we write $\mathcal F(\alpha)$ for the singular foliation by geodesics in direction $\alpha$.

We assume that the associated Veech group $G = G_q$ is a lattice---recall that $G$ can be viewed as the stabilizer in the mapping class group of $S$ of $D$ as well as the affine group of $q$, and the lattice assumption is equivalent to requiring the quotient orbifold $D/G$ to have finite $\rho$--area.  The parabolic fixed points in the circle at infinity form a subset we denote $\CP \subset \mathbb P^1(q)$.  This subset corresponds precisely to the completely periodic directions for the flat metric $q$; that is, the directions $\alpha$ for which the foliation $\mathcal F(\alpha)$ decomposes $S$ into cylinders foliated by $q$--geodesic core circles.   The boundaries of these cylinders are $q$--saddle connections ($q$--geodesic segments connecting pairs of cone points, with no cone points in their interior), and by the Veech Dichotomy, every saddle connection is in a direction in $\CP$.  We let $\{B_\alpha\}_{\alpha \in \CP}$ denote any $G$--invariant, $1$--separated set of horoballs in $D$ and let
\[ \bar D  = D \smallsetminus \bigcup_{\alpha \in \CP} {\rm{int}}(B_\alpha)\]
be the $G$--invariant subspace obtained by removing these horoballs.  We write $\bar \rho$ for the induced path metric on $\bar D$.
Finally, we let
\[ p \colon D \to \hat D\]
be the $G$--equivariant quotient obtained by collapsing each horoball $\CB_\alpha$ to a point, for $\alpha \in \CP$.  There is a natural path metric $\hat \rho$ on $\hat D$ so that $p$ is $1$--Lipschitz and is a local isometry at every point not in one of the horoballs.

We will also make use of the closest point projection to the horoball
\[ c_\alpha \colon D \to B_\alpha\]
for each $\alpha \in \CP$.

\subsection{The bundles $E$ and $\bar E$.} \label{S:E and bar E}
For each point $X \in D$, we let $q_X$ denote the associated flat metric or quadratic differential (defined up to scalar multiplication) on $S$.  The space of interest $E$ is a bundle over $D$,
\[ \pi \colon E \to D, \]
for which the fiber $E_X$ over $X \in D$ is naturally identified with the universal cover $\widetilde S$ of $S$, equipped with the pull-back complex structure $X$ and quadratic differential/flat metric $q_X$.  We write $\CB_\alpha = \pi^{-1}(B_\alpha)$ for $\alpha \in \CP$.

For any $X,Y \in D$, the Teichm\"uller map between these complex structures has initial and terminal quadratic differentials $q_X$ and $q_Y$ (up to scalar multiple) and this map lifts to a canonical affine map between the fibers $f_{Y,X} \colon E_X \to E_Y$.  These maps satisfy $f_{Z,X} = f_{Z,Y} f_{Y,X}$  for all $X,Y, Z \in D$, and for any $X \in D$, assemble to a map $f_X \colon E \to E_X$ defined by $f_X(y) = f_{X,\pi(y)}(y)$.  Moreover, for any $X,Y \in D$, $f_{Y,X}$ is $e^{\rho(X,Y)}$--bi-Lipschitz.  We use the maps $f_{X,X_0}$ to identify $\mathbb P^1(q) \cong \mathbb P^1(q_X)$ for all $X \in D$.

The fiber over $X_0$ is denoted $E_0 = E_{X_0}$ and the maps $f_0 = f_{X_0} \colon E \to E_0$ and $\pi \colon E \to D$ are projections on the factors in a product structure $E \cong D \times E_0 \cong D \times \widetilde S$.  For $x \in E$, we write $D_x = f_{\pi(x)}^{-1}(x)$, which is just the slice $D \times \{f_0(x)\}$ in the product structure.  The affine maps $f_{Y,X}$ sends the cone points $\Sigma_X$ of $E_X$ to the cone points $\Sigma_Y$ of $E_Y$.  Consequently, the union of all singular points
\[ \Sigma = \bigcup_{X \in D} \Sigma_X \]
is a locally finite union of disks $D_x$, one for each $x \in\Sigma_0 = \Sigma_{X_0}$.

We give the space $E$ a singular Riemannian metric $d$ which is the flat metric on each fiber $E_X$ and the Poincar\'e metric on each disk$D_x$ so that at each smooth point of intersection, the tangent planes are orthogonal.  The singular locus of this metric is precisely $\Sigma$.  Each disk $D_x$ is isometrically embedded since $\pi$ is a $1$--Lipschitz map, and hence restricts to an isometry $\pi|_{D_x} \colon D_x \to D$.  The metric on $E \smallsetminus \Sigma$ is in fact a locally homogeneous metric, modeled on a four-dimensional, Thurston-type geometry; see \cite[\S5]{DDLSI}.

The extension group $\Gamma$ acts on $E$ by bundle maps with the kernel $\pi_1 S  < \Gamma$ of the projection to $G$ acting trivially on $D$ and by covering transformation on each fiber $E_X$.  We set $\bar E = \pi^{-1}(\bar D) \subset E$, and write $\bar \pi \colon \bar E \to \bar D$.  When convenient to do so, we put ``bars" over objects associated to $\bar D$ or $\bar E$, e.g.~$\bar D_x = D_x \cap \bar E$, $\bar p \colon \bar D \to \hat D$, etc.  In particular, we write $\bar d$ for the induced path metric on $\bar E \subset E$, induced from the metric on $E$ described above.

For any $\alpha \in \CP$, the closest point projection $c_\alpha \colon D \to B_\alpha$ has a useful ``lift" $f_\alpha \colon E \to \CB_\alpha$, defined by
\[ f_\alpha(x) = f_{c_\alpha(\pi(x))}(x),\]
for any $x \in \bar E$.  That is, $f_\alpha$ maps each fiber $E_X$ via the map $f_{Y,X}$ to $E_Y$, where $Y = c_\alpha(X)$ is the image of the closest point projection to $B_\alpha$ of $X$ in $D$.

\subsection{The hyperbolic space $\hat E$}\label{subsec:hatE}
The quotient $p \colon D \to \hat D$ is the descent of a quotient $P \colon E \to \hat E$ which we now describe.  First, for each $\alpha \in \CP$, the foliation $\mathcal F(\alpha)$ lifts to a foliation on $E_0$ in direction $\alpha$, and hence on any fiber $E_X$ by push-forward via the map $f_{X,X_0}$, also in direction $\alpha$ (via the identification $\mathbb P^1(q) \cong \mathbb P^1(q_X)$).  There is a natural transverse measure coming from the flat metric on $X$.  Given $\alpha \in \CP$, we fix some $X_\alpha \in \partial B_\alpha$ and let $T_\alpha$ be the dual simplicial $\mathbb R$--tree to this measured foliation in direction $\alpha$ on $E_{X_\alpha}$, and we let $t_\alpha \colon E \to T_\alpha$ be the composition of the leafspace projection $E_{X_\alpha} \to T_\alpha$ with the map $f_{X_\alpha} \colon E \to E_{X_\alpha}$.

Now we define $P \colon E \to \hat E$ to be the quotient space obtained by collapsing the subset $\CB_\alpha$ to $T_\alpha$ via $t_\alpha|_{\CB_\alpha}$ for each $\alpha \in \CP$.  We also write $\bar P = P|_{\bar E} \colon \bar E \to \hat E$.  The maps $P$ and $\bar P$ descend to the maps $p$ and $\bar p$, and the map $\pi$ determines maps $\hat \pi$ and $\bar \pi$, which all fit into the following commutative diagram.
\[
\begin{tikzcd}[row sep=6]
                 & E \ar[dr,  "P"] \ar[dd, near start, "\pi"] \\
\bar E \arrow[ur, hook] \ar[dd,  "\bar \pi"] \ar[rr,  near start,  "\bar P" below, crossing over] &&  \hat E \ar[dd,  "\hat \pi"]\\
                 & D \ar[dr,  "p"] \\
\bar D \arrow[ur, hook] \ar[rr,  "\bar p"] &&  \hat D.\\
\end{tikzcd}
\]
A metric $\hat d$ on $\hat E$ is determined by $\bar d$ on $\bar E$ and the map $\bar P$.  The main facts about this metric are summarized in the following theorem; see \cite[Theorem 1.1, Lemma 3.2]{DDLSI}.
\begin{theorem} \label{T:quotient metric on E hat}
There is a Gromov hyperbolic path metric $\hat d$ on $\hat E$ so that $\bar P \colon \bar E \to \hat E$ is $1$--Lipschitz and is a local isometry at every point $x \in \bar E - \partial \bar E$.  Furthermore, for every $\alpha\in \CP$, 
\begin{itemize}
\item The induced path metric on $P(\partial \CB_\alpha) = T_\alpha$ is the $\mathbb{R}$--tree metric determined by the transverse measure on the foliation of $E_{X_\alpha}$ in direction $\alpha$.
\item The subspace topology on $T_\alpha \subset \hat E$ agrees with the $\mathbb{R}$--tree topology on $T_\alpha$.
\end{itemize}
\end{theorem}

\begin{remark} The underlying simplicial tree $T_\alpha$ is precisely the Bass-Serre tree dual to the splitting of $\pi_1S$ defined by the cores of the cylinders of $\mathcal F(\alpha)$ on $S$.
\end{remark}

For each $x \in E$, we denote the image of $D_x$ in $\hat E$ by $\hat D_x$, which is obtained by collapsing $\CB_\alpha \cap D_x$ to a point, for each $\alpha \in \CP$.  Consequently, $\hat \pi|_{\hat D_x} \colon \hat D_x \to \hat D$ is a bijection, and so each $\hat D_x$, with its path metric, is isometric to $\hat D$ and isometrically embedded in $\hat E$.
We call objects in $E$, $\bar E$, and $\hat E$ {\em vertical} if they are contained in a fiber of $\pi$, $\bar \pi$, or $\hat \pi$, respectively, and {\em horizontal} if they are contained in $D_x$, $\bar D_x$, or $\hat D_{x}$, for some $x \in E,\bar E$.

\subsection{Vertices, spines, and spine bundles}
We will write $\vtx \subset \hat E$ for the union over all $\alpha \in \CP$ of all vertices of $T_\alpha$.  We will simultaneously view $\vtx$ as both a subset of $\hat E$ and abstractly as an indexing set that will be used in sections \S\S\ref{S:projections}--\ref{S:combinatorial HHS} to develop an HHS structure on $\bar E$.  Since each vertex belongs to a unique tree, and since the trees are indexed by $\alpha \in \CP$, we obtain a map $\alpha \colon \vtx \to \CP$ so that $v$ is a vertex of $T_{\alpha(v)}$.
For convenience, we also write $B_v = B_{\alpha(v)}$, $\partial B_v = \partial B_{\alpha(v)}$, etc for each $v \in \vtx$, and write $c_v = c_{\alpha(v)}$ for the $\rho$--closest point projection $D \to B_v$. 

For $v,w \in \vtx$, we write $v \parallel w$ if $\alpha(v) = \alpha(w)$. Then define  $d_{\tree}(v,w) \in \mathbb Z_{\geq 0} \cup \{\infty\}$ to be the combinatorial (integer valued) distance in the simplicial tree $T_{\alpha(v)}= T_{\alpha(w)}$ when $v \parallel w$  (as opposed to the distance from the $\mathbb R$--tree metric) and to equal $\infty$ when $v\not\parallel w$.

Given $\alpha \in \CP$, $X \in D$, and $v \in T_\alpha$, the {\em $v$--spine} in $E_X$ is the subspace
\[  \theta^v_X = (P \circ f_{X_\alpha,X})^{-1}(v) = t_\alpha^{-1}(v) \cap E_X.\]
The $v$--spine $\theta^v_X$ is the union of the saddle connections on the fiber $E_X$ in direction $\alpha$ that project to $v$ by $t_\alpha$.
When $d_\tree(v,w) = 1$ (and hence $v,w$ are adjacent in the same tree $T_\alpha$) there is a unique component of $E_X \smallsetminus (\theta^v_X \cup \theta^w_X)$ whose closure is an infinite strip, $\mathbb R \times [a,b]$, that covers a maximal cylinder in the quotient $E_X/\pi_1S = (S,X,q_X)$ in the direction $\alpha$.  We let $\bTheta^v_X$ be the union of $\theta^v_X$ and all such strips defined by $w \in T_\alpha$ with $d_\tree(v,w) = 1$.    We call $\bTheta^v_X$ the {\em thickened $v$--spine in $E_X$}.  In the special case $X=X_0$, we write $\theta^v_0 = \theta^v_{X_0}$ and $\bTheta^v_0=\bTheta^v_{X_0}$.  Observe that the affine map $f_{Y,X}$ maps $\theta^v_X$ and $\bTheta^v_X$ to $\theta^v_Y$ and $\bTheta^v_Y$, respectively, for all $X,Y \in D$.
Finally, we write
\[ \theta^v = \bigcup_{X \in \partial B_v} \theta^v_X \quad \quad \bTheta^v = \bigcup_{X \in \partial B_v} \bTheta^v_X. \]
These spaces are bundles over $\partial B_v$ which we call, respectively, the {\em $v$--spine bundle} and the {\em thickened $v$--spine bundle}.

\subsection{Schematic of the space $\bar E$ and its important pieces.}

\begin{figure}[htb]
\includegraphics[width=\linewidth]{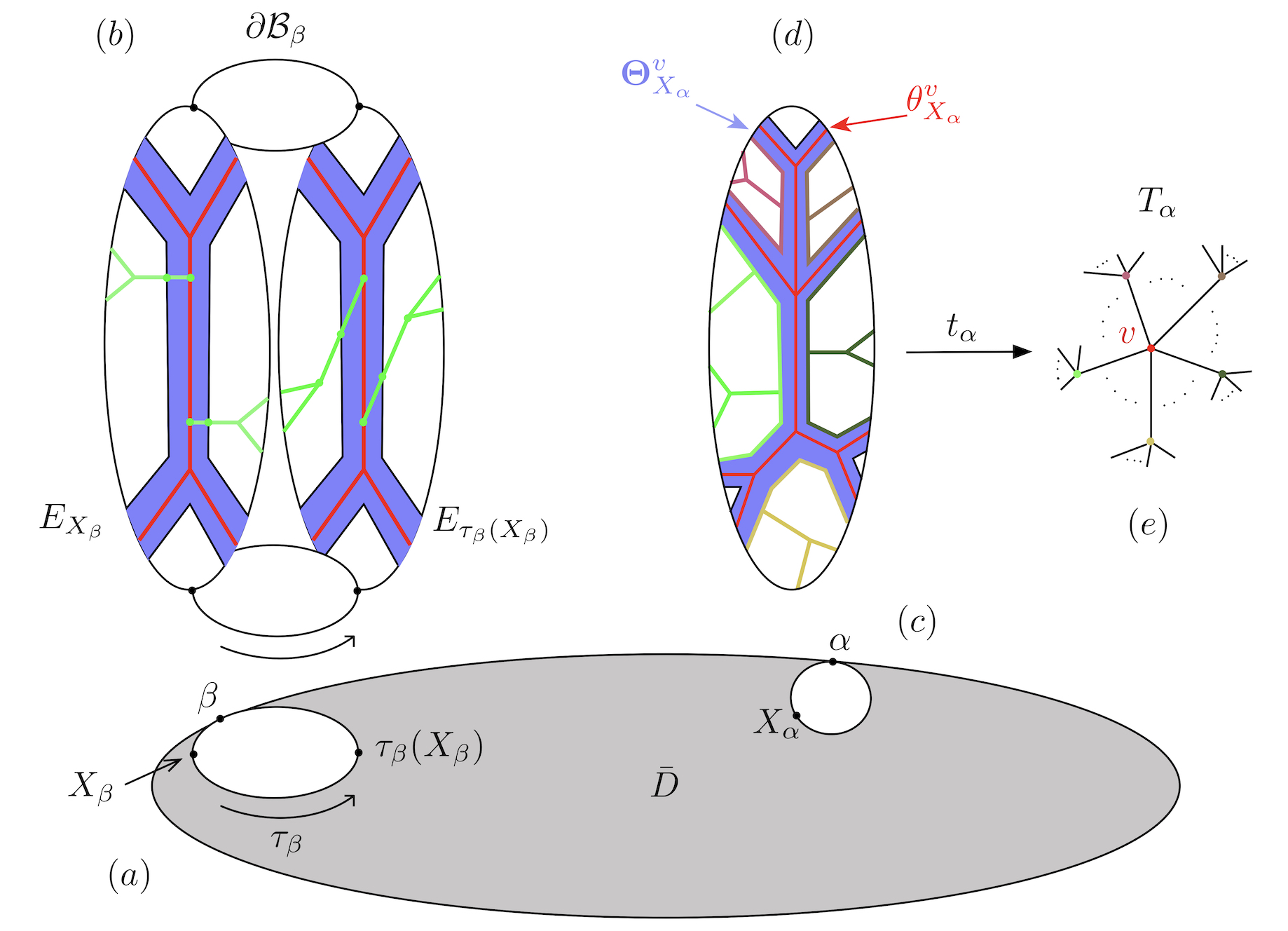}
\caption{A schematic of $\bar E$ and various key features of it.}
\label{fig:Veech bundle}
\end{figure}

Figure \ref{fig:Veech bundle} is a cartoon of the bundle $\bar E$ over the truncated Teichm\"uller disk $\bar D$.  We have tried to highlight some of the key features of $\bar E$ which are relevant to this paper.

\begin{itemize}
\item[(a)] The stabilizer of a horoball based at a point $\beta \in \mathcal P$ is virtually cyclic, generated by a multitwist $\tau_\beta$ acting as a parabolic on $D$.  The base point $X_\beta$ on the horocycle based at $\beta$ and its image are shown. 

\item[(b)] The bundle over the boundary horocycle based at $\beta$ is shown.  This is the universal cover, $\partial \mathcal B_{\beta}$, of a graph manifold which is the mapping torus of $\tau_\beta$.  Two fibers $E_{X_\beta}$ and $E_{\tau_\beta(X_\beta)}$ are shown with the effect on a part of a spine (in green) in some other direction illustrating the sheering in strips after applying $\tau_\beta$.

\item[(c)] This is another horoball in some direction $\alpha$, with the chosen basepoint $X_{\alpha}$ and its horocycle $\partial B_{\alpha}$.

\item[(d)] The spine $\theta^v_{X_{\alpha}}$ in direction $\alpha$ is shown in red, corresponding to a vertex $v \in T_\alpha$.  The thickened spine $\bTheta^v_{X_{\alpha}}$ is indicated in lavender.  Spines for vertices of $T_\alpha$ adjacent to $v$ meet $\bTheta^v_{X_{\alpha}}$ along lines in $\partial \bTheta^v_{X_\alpha}$ and are shown in various other colors.

\item[(e)] The restriction of $t_\alpha\colon E\to T_\alpha$ to $E_{X_\alpha}$  collapses each spine $\theta^w_{X_\alpha}$ or strip in direction $\alpha$ to the corresponding vertex $w$ or edge the Bass-Serre tree $T_\alpha$.  
The space  $\hat E$ is formed by collapsing $\mathcal B_\alpha$ to $T_\alpha$ via $t_\alpha$.
\end{itemize}

\subsection{Some technical lemmas and coarse geometry}

Here we briefly recall some basic facts about the setup above proved in \cite{DDLSI} as well as some useful coarse geometric facts.  The first fact is the following; see \cite[Lemma~3.4]{DDLSI}.

\begin{lemma}
\label{L:strip-and-saddle-bound}
There exists a constant $M > 0$ such that for each $v\in \vtx$ and $X\in \partial B_v$, every saddle connection in $\theta^v_X$ has length at most $M$ and every strip in $\bTheta^v_X$ has width at most $M$. In particular, for points $X\in \partial B_\alpha$, the saddle connections and strips of $E_X$ in direction $\alpha\in \CP$ have, respectively, uniformly bounded lengths and widths.
\end{lemma}

Every connected graph can be made into a geodesic metric space by locally isometrically identifying each edge with a unit interval.  We will need the following well-known result (for a proof of this version, see \cite[Proposition~2.1]{DDLSI}).

\begin{prop}\label{P:graph approximation}
Let $\Omega$ be a path metric space and $\Upsilon \subset \Omega$ an $R$--dense subset for some $R > 0$.  For any $R'> 3R$, consider a graph $\mathcal G$ with vertex set $\Upsilon$ such that:
  \begin{itemize}
   \item all pairs of elements of $\Upsilon$ within distance $3R$ are joined by an edge in $\mathcal G$,
   \item if an edge in $\mathcal G$ joins points $w,w'\in \Upsilon$, then $d_\Omega(w,w') \le R'$.
  \end{itemize}
 Then the inclusion of $\Upsilon$ into $\Omega$ extends to a quasi-isometry $\mathcal G \to \Omega$.
\end{prop}

The following criterion for a graph to be a quasi-tree is well-known, and an easy consequence of Manning's bottleneck criterion \cite{Man:bottleneck}. We include a proof for completeness.

\begin{prop} \label{prop:bottleneck}
 Let $X$ be a graph, and suppose that there exists a constant $B$ with the following property:  For each pair of vertices $w,w'$ there exists an edge path $\gamma(w,w')$ from $w$ to $w'$ so that for any vertex $v$ on $\gamma(w,w')$, any path from $w$ to $w'$ intersects the ball of radius $B$ around $v$. Then $X$ is quasi-isometric to a tree, with quasi-isometry constants depending on $B$ only.
\end{prop}

\begin{proof}
 We check that \cite[Theorem 4.6]{Man:bottleneck} applies; that is, we check the following property.  For any two vertices $w,w' \in X$, there is a midpoint $m(w,w')$ between $w$ and $w'$ so that any path from $w$ to $w'$ passes within distance $B'=B'(B)$ of $m(w,w')$.  (The uniformity in the quasi-isometry comes from the proof of Manning's theorem, see \cite[page 1170]{Man:bottleneck}.)
 
 Consider any geodesic $\alpha$ from $w$ to $w'$, and let $m=m(w,w')$ be its midpoint. We will show that $m$ lies within distance $2B+1$ of a vertex of $\gamma=\gamma(w,w')$, so that we can take $B'=3B+1$.
 
Indeed, suppose by contradiction that this is not the case. Let $w=w_0,\dots,w_n=w'$ be the vertices of $\gamma$ (in the order in which they appear along $\gamma$), and let $d_i=d(w,w_i)$, so that $|d_{i+1}-d_i|\leq 1$. Each $w_i$ lies within distance $B$ of some point $p_i$ on $\alpha$ which must satisfy $d(p_i,m)\geq B+1$. In particular, we have that every $d_i$ satisfies either $d_i \leq d(w,w')/2 -1$ or $d_i\geq d(w,w')/2 +1$. Since $d_0=0$ and $d_n=d(w,w')$, we cannot have $|d_{i+1}-d_i|\leq 1$ for all $0\leq i\leq n-1$, a contradiction.
\end{proof}

We end with a few definitions from coarse geometry which may not be completely standard, but will appear in the next two sections.
Given two metrics $d$ and $d'$ on a set $X$, we say that $d$ is {\em coarsely bounded} by $d'$ if there exists a monotone function $N \colon [0,\infty) \to [0,\infty)$ so that $d(x,y) \leq N(d'(x,y))$, for all $x,y \in X$.  If $d$ is coarsely bounded by $d'$ and $d'$ is coarsely bounded by $d$, we say that $d$ and $d'$ are {\em coarsely equivalent}.
An isometric action of a group $H$ on a metric space $Y$ is {\em metrically proper} if for any $R >0$ and any point $y \in Y$, there are at most finitely many elements $h \in H$ for which $h \cdot B(y,R) \cap B(y,R) \neq \emptyset$.  For proper geodesic spaces, this is equivalent to acting properly discontinuously.  If there exists $y,R$ so that $H \cdot B(y,R) = Y$, then we say that the action is cobounded, and for proper geodesic metric spaces this is equivalent to acting cocompactly.

\section{Projections and vertex spaces} \label{S:projections}

An HHS structure on a metric space consists of certain additional data, most importantly a collection of hyperbolic spaces together with projection maps to each space. For the HHS structure that we will build on (Cayley graphs of) $\Gamma$, the hyperbolic spaces will (up to quasi-isometry) be the space $\hat E$ from \cite{DDLSI} (see \S\ref{subsec:hatE}) and the spaces $\mathcal K^v$ and $\Xi^v$ introduced in this section, where $v$ varies over all vertices of the trees $T_\alpha$.
Morally, the projections will be given by the maps $\Lambda^v$ and $\xi^v$ that we study below. However, to prove hierarchical hyperbolicity we will use a criterion from \cite{comb_HHS} which does not require actually defining projections, but nevertheless provides them. Still, the maps $\Lambda^v$ and $\xi^v$ will play a crucial role in proving this criterion applies.

We will establish properties of $\Lambda^v$ and $\xi^v$ that are reminiscent of subsurface projections or of closest-point projections to peripheral sets in relatively hyperbolic spaces/groups; these are summarized in Proposition \ref{prop:R-projections}.
Essentially, these same properties would be needed if we wanted to construct an HHS structure on $\Gamma$ directly without using \cite{comb_HHS}.

From a technical point of view, we would like to draw attention to Lemma \ref{lem:window}, which is the crucial lemma that ensures that the projections behave as desired and that various subspaces have bounded projections. Roughly, the lemma says that closest-point projections to a spine do not vary much under affine deformations.

In what follows, we will write $d_{\bTheta^v}$ and $d_{\partial {\mathcal B}_\alpha}$ for the path metrics on $\bTheta^v$ and $\partial {\mathcal B}_\alpha$ induced from $\bar d$.  Using the map $f_\alpha \colon \bar E \to \partial {\mathcal B}_\alpha$, it is straightforward to see that $d_{\partial {\mathcal B}_\alpha}$ is uniformly coarsely equivalent to the subspace metric from $\bar d$: in fact, $\bar d \leq d_{\partial {\mathcal B}_\alpha} \leq e^{\bar d}\bar d$.  The same is true for $d_{\bTheta^v}$, which follows from the fact that the inclusion of $\bTheta^v$ into $\partial {\mathcal B}_\alpha$ is a quasi-isometric embedding with respect to the path metrics (see below).

Associated to each $v \in \vtx$ we will be considering two types of projections.  These projections have a single projection $\Pi^v \colon \bar{\Sigma} \to \bTheta^v$ as a common ingredient.  It is convenient to analyze $\Pi^v$  via an auxiliary map which serves as a kind of fiberwise closest point projection that survives affine deformations, and which we call the {\em window map}.  We describe the two types of projections restricted to $\bTheta^v$, as well as the target spaces of said projections, in \S\ref{S:Kapovich-Leeb distances} and \S\ref{subsec:coned-off}, where we also explain some of their basic features.  Next we define the window map and prove what is needed from it. Finally, we define $\Pi^v$ and prove the key properties of the associated projections.

\subsection{Quasimorphism distances} \label{S:Kapovich-Leeb distances}

For each $v \in \vtx$, we will use ideas from work-in-progress of the fourth author with Hagen, Russell, and Spriano \cite{SquidGames}  to define a map
\[\lambda^v \colon \bTheta^v \rightarrow \mathcal K^v,\]
where $\mathcal K^v$ is a discrete set quasi-isometric to $\mathbb R$.  The key properties of this map are given by the next proposition.
We note that the proposition and Lemma \ref{lem:retract_on_the_middle} can be used as black-boxes (in particular, the definitions of $\lambda^v$ and $\mathcal K^v$ are never used after we prove those results).

\begin{proposition}\label{prop:KL}
There exists $K_1>0$ such that, for each $v \in T_\alpha^{(0)} \subset \vtx$, there exist a space $\mathcal K^v$ that is $(K_1,K_1)$--quasi-isometric to $\mathbb R$ and a map $\lambda^v\colon \bTheta^v \rightarrow \mathcal K^v$ satisfying the following properties:
\begin{enumerate}
\item \label{item:lambda_lipschitz} $\lambda^v$ is $K_1$--coarsely Lipschitz with respect to the path metric on $\bTheta^v$.
\item \label{item:lambda_flowlines_bounded} For any $x \in \partial \bTheta^v$, if $\ell_{x,\alpha} = D_x \cap \partial \mathcal B_\alpha$ then $\lambda^v(\ell_{x,\alpha})$ is a set of diameter bounded by $K_1$.
\item\label{item:KL_qi} For any $v,w\in \vtx$ with $d_{\mathrm{tree}}(v,w)=1$, $\lambda^v\times\lambda^w\colon \bTheta^v\cap \bTheta^w \to \mathcal K^v\times \mathcal K^w$ is a $K_1$--coarsely surjective $(K_1,K_1)$--quasi-isometry with respect to the induced path metric on the domain.
\item\label{item:equivariance} (Equivariance) For any $g\in \Gamma$ and $v \in T_\alpha^{(0)}$ there is an isometry $g\colon\mathcal K^v\to\mathcal K^{gv}$ and for all $x\in \Theta^v$ we have $\lambda^{gv}(gx)=g\lambda^v(x)$.
\end{enumerate}
\end{proposition}

The sets $\ell_{x,\alpha}$ in item (2) are certain lines whose significance is explained below.

\begin{remark} An earlier version of this paper used work of Kapovich and Leeb to construct the spaces $\mathcal K^v$ and maps $\lambda^v$, resulting in a weaker version of this proposition which did not include the last, equivariance condition.  Consequently $\Gamma$ could only be shown to be an HHS, rather than an HHG.  The ideas from \cite{SquidGames} were crucial in this extension.
\end{remark}

To explain the proof of the proposition, it is useful to review some background on graph manifolds, which we do now.

\subsection*{Graph manifolds and trees}
Recall that a graph manifold is a $3$--manifold that contains a canonical finite union of tori (up to isotopy), so that cutting along the tori produces a disjoint union of Seifert fibered $3$--manifolds, called the {\em Seifert pieces}.  Seifert fibered $3$--manifolds are compact $3$--manifolds foliated by {\em circle leaves}; see \cite{JacoShalen}.

The universal cover of a graph manifold decomposes into a union of universal covers of the Seifert pieces glued together along $2$--planes (covering the tori).  The decomposition is dual to a tree, and the universal covers of the Seifert pieces are {\em the vertex spaces}.  For any Seifert fibered space, its universal cover is foliated by lines, the lift of the foliation by circles, and we refer to the leaves simply as {\em lines} in the universal cover.

\subsection*{Horocycles and bundles}
Next we describe the specific graph manifolds that are relevant for our purposes.

Let $G_\alpha < G$ denote the stabilizer of $B_\alpha$, for each $\alpha \in \mathcal P$.  This has a finite index cyclic  subgroup $G_\alpha^0$ generated by a multitwist, $\langle \tau_\alpha \rangle = G_\alpha^0 < G_\alpha$; see e.g.~\cite[\S2.9]{DDLSI}.  The preimage of $G_\alpha$ in $\Gamma$ is the $\pi_1S$--extension group $\Gamma_\alpha$ of $G_\alpha$, and we likewise denote by $\Gamma_\alpha^0 < \Gamma_\alpha$ the extension group of $G_\alpha^0$.   The action of $\Gamma_\alpha$ on $\partial \mathcal B_\alpha$ is cocompact, and $\partial \mathcal B_\alpha/\Gamma_\alpha$ has a finite sheeted (orbifold) covering by $\partial \mathcal B_\alpha/\Gamma_\alpha^0$, which is the graph manifold mentioned in the introduction.  

\begin{center}
\begin{figure}[htb]
\begin{tikzpicture}[scale = .7]
\draw[fill, opacity = .2, color=green] (0,0) -- (4,0) -- (4,2) -- (0,2) -- (0,0);
\draw[fill, opacity = .2, color=blue] (2,2) -- (2,4) -- (0,4) -- (0,2) -- (2,2);
\draw[ultra thick] (0,0) -- (4,0) (0,2) -- (4,2) (0,4) -- (2,4);
\draw[thin] (0,0) -- (0,4) (4,0) -- (4,2) (2,2) -- (2,4);
\draw[line width = 1pt,dashed] (0,1) -- (4,1);
\draw (0,2) -- (2,2);
\draw[line width = 1pt,dashed] (0,3) -- (2,3);
\draw[fill] (0,0) circle (.08);
\draw[fill] (4,0) circle (.08);
\draw[fill] (4,2) circle (.08);
\draw[fill] (2,2) circle (.08);
\draw[fill] (2,4) circle (.08);
\draw[fill] (0,4) circle (.08);
\draw[fill] (0,2) circle (.08);
\draw[fill] (2,0) circle (.08);
\node[left] at (0,1) {$1$};
\node[left] at (0,3) {$2$};
\node[right] at (4,1) {$1$};
\node[right] at (2,3) {$2$};
\node[above] at (1,4) {$3$};
\node[below] at (1,0) {$3$};
\node[below] at (3,0) {$4$};
\node[above] at (3,2) {$4$};
\end{tikzpicture}
\caption{The surface obtained by gluing sides of the ``L-shaped" polygon in pairs by translation according to the numbering has a decomposition into two cylinders (shaded blue and green) in the horizontal direction, $\alpha$; $\tau_\alpha$ is a twist in the bottom cylinder and a square of a twist in top cylinder.  The boundaries of these cylinders (drawn in bold) form spines for the complement of core curves (drawn as dotted lines).} \label{F:twisting and retraction}
\end{figure}
\end{center}

Consider the surface $S$ with the flat metric $q_{X_\alpha}$, so that $(S,X_\alpha,q_{X_\alpha}) = E_{X_\alpha}/\pi_1S$.  The multitwist $\tau_\alpha$ is an affine map that preserves the cylinders in direction $\alpha$, acting as a power of a Dehn twist in each cylinder and as the identity on their boundaries.  The union of the boundaries of the cylinders are spines (deformation retracts) for the subsurfaces that are the complements of the twisting curves (core curves of the cylinders).  Consequently, $\tau_\alpha$ is the identity on these spines.  The homeomorphism $\tau_\alpha$ induces a homeomorphism on the subsurface obtained by cutting open $S$ along a core curve of each cylinder.  Each such induced homeomorphism is the identity on the corresponding spine, and is thus isotopic to the identity relative to the spine; see Figure~\ref{F:twisting and retraction}.  The mapping torus of each subsurface is a product of the subsurface times a circle, and embeds in the the mapping torus $\partial \mathcal B_\alpha/\Gamma_\alpha^0$ of $\tau_\alpha$.  These sub-mapping tori are the Seifert pieces for the graph manifold structure on $\partial \mathcal B_\alpha/\Gamma_\alpha^0$.

The lifted graph manifold decomposition of $\partial \mathcal B_\alpha$ corresponds to $T_\alpha$.  That is, for each $v \in T_\alpha^{(0)}$, there is a vertex space contained in $\bTheta^v$ and containing $\theta^v$.  In fact, with respect to the covering group, $\bTheta^v$ is an invariant, bounded neighborhood of the vertex space and $\theta^v$ is an equivariant deformation retraction of that space.  We let $\Gamma^v < \Gamma_\alpha$ denote the stabilizer of $\bTheta^v$ in $\Gamma_\alpha$ and $\Gamma^{v0} < \Gamma_\alpha^0$ the stabilizer in $\Gamma_\alpha^0$.
The suspension flow on the mapping torus $\partial \mathcal B_\alpha/\Gamma_\alpha^0$ restricted to each quotient of the spine bundle, $\theta^v/\Gamma^{v0}$, defines circle leaves of the corresponding Seifert piece; that is, flow lines through any point on the $\theta^v/\Gamma^{v0}$ are precisely the circle leaves.  In the universal covering $\partial \mathcal B_\alpha$, the lifted flowline through a point $x \in \partial \mathcal B_\alpha$ is a lifted horocycle, $\ell_{x,\alpha} = D_x \cap \partial \mathcal B_\alpha$.  Thus, for any vertex $v$ and any $x \in \theta^v$, $\ell_{x,\alpha}$ is a line for the vertex space corresponding to $v$.  We note that not only does $\Gamma^{v0}$ preserve this set of lines, but so does $\Gamma^v$.

For any $x \in \theta^v$, the stabilizer in $\Gamma^{v0}$ of $\ell_{x,\alpha}$ is generated by a lift $g_v$ of $\tau_\alpha$.  Therefore, the quotient $\bTheta^v/\Gamma^{v0}$ is homeomorphic to a product, $\bTheta^v_X/\pi_1S^v \times S^1$, where $\pi_1S^v$ is the stabilizer of $v$ in $\pi_1S < \Gamma$ and $X \subset \partial B_\alpha$ is any point.  Indeed, there is a deformation retraction to $\theta^v/\Gamma^{v0} = \theta^v_X/\pi_1S^v \times S^1$.   If we do not care about the particular point $X$ over which we take the fiber, we simply write $S^v$ for the surface $\bTheta^v_X/\pi_1S^v$, so that $\bTheta^v/\Gamma^{v0} \cong S^v \times S^1$.  Since $E_X$ is a copy of the universal cover of $S$, we can consider $S^v$ as a subsurface of $S$ (embedded on the interior) and $\pi_1S^v$ is its fundamental group inside $\pi_1S$ (up to conjugacy).

The product structure $S^v \times S^1 \cong \bTheta^v/\Gamma^{v0}$ can be chosen so that $\bTheta^v/\Gamma^{v0} \to \bTheta^v/\Gamma$ is an orbifold cover sending circles to circles making $\bTheta^v/\Gamma^v$ into a Seifert fibered {\em orbifold} (some of the Seifert fibers may be part of the orbifold locus) that also (orbifold)-fibers over the circle (with finite order monodromy).  Write $\bTheta^v/\Gamma^v \to \mathcal O^v$ for the Seifert fibration to the quotient $2$--orbifold.  Further write
$$\nu^v \colon \Gamma^v \to \pi_1^{orb}(\mathcal O^v)$$
for the induced homomorphism of the Seifert fibration and
$$\phi^v \colon \Gamma^v \to \mathbb Z$$
for the induced homomorphism from the fibration over the circle.  Because $g_v$ acts as translation on the line $\ell_{x,\alpha}$ for $x \in \theta^v$, it represents a loop that traverses a circle in the Seifert fibration, which is thus also a suspension flowline for the fibration over the circle.  Thus we have $\nu^v(g_v) = 0$ and $\phi^v(g_v) \neq 0$.  To complete the picture, we note that restricting $S^v \times S^1 \cong \bTheta^v/\Gamma^{v0} \to \bTheta^v/\Gamma^v \to \mathcal O^v$ to $S^v$ defines an orbifold covering $S^v \to \mathcal O^v$.

Finally, note that for any $w$ adjacent to $v$ in $T_\alpha$, $\langle g_v \rangle \times \langle g_w \rangle \cong \mathbb Z^2$ has finite index in $\Gamma^v \cap \Gamma^w$.  Viewing $\langle g_w \rangle < \Gamma^v$, we note that $\nu^v|_{\langle g_w \rangle}$ is an isomorphism onto an infinite cyclic subgroup of $\pi_1^{orb} \mathcal O^v$.  In fact, the image $\nu^v(\langle g_w \rangle)$ is (a conjugate of a power of) the fundamental group of a boundary component of $\mathcal O^v$.

\begin{remark}
One caveat about the lines for the vertex spaces: flowlines through points {\em not} on a spine are {\em not} lines of any vertex space.  In fact, they are not even uniformly close to lines for any vertex space.
\end{remark}

\subsection*{Constructing the map}
Here we define $\mathcal K^v$ and $\lambda^v$ and prove the main properties we will need about them.  We require a little more set up first.  We choose representatives of the $\Gamma$--orbits of vertices, $\vtx_0 = \{v_1,\ldots,v_k\} \subset \vtx$.  For each $v \in \vtx_0$, choose a fundamental domain $\Delta^v$ for the action of $\Gamma^v$ on $\bTheta^v$.  We assume that $\Delta^v$ has compact, connected closure, that $g \Delta^v \cap \Delta^v = \emptyset$ for all $g \in \Gamma^v \setminus \{1\}$, and that $\bigcup_{g \in \Gamma^v} g \cdot \Delta^v= \bTheta^v$.  The set
\begin{equation} \label{E:finite gen set} 
\{g \in \Gamma^v \mid g \overline{\Delta}^v \cap \overline{\Delta}^v \neq \emptyset\}
\end{equation}
is a finite generating set for $\Gamma^v$.
The $\Gamma^v$--translates of $\Delta^v$ define a tiling of $\bTheta^v$, and the map sending every point of $g \Delta^v$ to $g \in \Gamma^v$  is a quasi-isometry by the Milnor-Schwarz Lemma.  We denote this map as $\widetilde \lambda^v \colon \bTheta^v \to \Gamma^v$.

We note that any word metric on $\Gamma^v$ defines a ``word metric" on each coset $g \Gamma^v$, for $g \in \Gamma$ (elements are distance $1$ if they differ by right multiplication by an element of the generating set).  We can push the tiling forward by $g$ to a $\Gamma^{gv} = g \Gamma^v g^{-1}$--invariant tiling of $\bTheta^{gv}$ (if $g \in \Gamma^v$, this is precisely the given tiling of $\bTheta^v$).  For any element $g' \in g\Gamma^v$, the map that sends every point in $g' \Delta^v$ to $g'$ defines a quasi-isometry $\widetilde \lambda^{gv} \colon \bTheta^{gv} \to g \Gamma^v$ which is $\Gamma^{gv}$--equivariant, with the same quasi-isometry constants.  
If $g' \in \Gamma$ and $x' \in \Delta^v$, then for all $g \in \Gamma$ 
\[ \widetilde \lambda^{gg'v}(gg'x') = gg' = g \widetilde \lambda^{g'v}(g'x').\]
On the other hand, any $w \in \Gamma \cdot v$ and $x \in \bTheta^w$ have the form $w= g'v$ and $x = g'x'$ for some $g' \in \Gamma$ and $x' \in \Delta^v$.
Thus, for any $g \in \Gamma$, the equation above becomes
\begin{equation} \label{E:equivariance unraveled} \widetilde \lambda^{gw} (gx) = \widetilde g\lambda^{w} (x) \end{equation}

Having carried out the construction above for each $v \in \vtx_0$ and each vertex in its orbit, we have maps $\widetilde \lambda^w$  from $\bTheta^w$ to a coset of a vertex stabilizer from $\vtx_0$ for every $w \in \vtx$, so that equation \eqref{E:equivariance unraveled} holds for every $x \in \bTheta^w$, and $g \in \Gamma$.

Next, recall that a {\em homogeneous quasimorphism} ({\em with deficiency $D$}) from a group $H$ to $\mathbb R$ is a map
\[ \psi \colon H \to \mathbb R \]
such that for all $h,h_1,h_2 \in H$ and $n \in \mathbb Z$ we have $\psi(h^n) = n \psi(h)$ and
\[ |\psi(h_1h_2) - \psi(h_1) - \psi(h_2)| \leq D. \]

\begin{lemma} \label{L:quasimorphism} For any $v \in \vtx$, there is a homogeneous quasimorphism $\psi^v \colon \Gamma^v \to \mathbb R$ such that $\psi^v(\langle g_v \rangle)$ is unbounded, and $\psi^v (g_w) = 0$ for any adjacent vertex $w \in \vtx$.
\end{lemma}

\begin{proof} Let $w_1,\ldots,w_r$ be $\Gamma^v$--orbit representatives of the vertices adjacent to $v$.  Here $r$ is the number of boundary components of $\mathcal O^v$, so that $\nu^v(g_{w_1}),\ldots,\nu^v(g_{w_r})$ are peripheral loops around the $r$ distinct boundary components of $\mathcal O^v$.   Since $\pi_1^{orb}\mathcal O^v$ is the fundamental group of a hyperbolic $2$--orbifold with non-empty boundary, appealing to \cite[Theorem 4.2]{HullOsin}, which applies to $\pi^{orb}_1\mathcal O^v$ and its subgroups $\langle \nu^v(g_{w_i}) \rangle$ in view of \cite[Corollary 6.6, Theorem 6.8]{DGO},
one can find a homogeneous quasimorphism $\eta_i \colon \pi^{orb}_1\mathcal O^v \to \mathbb R$, for $i=1,\ldots,r$, such that $\eta_i(\nu^v(g_{w_i})) = 1$ and $\eta_i(\nu^v(g_{w_j})) = 0$ for $j \neq i$. (The construction of Epstein--Fujiwara \cite{EpsteinFujiwara} should also be applicable to construct such quasimorphisms).
  Set $s_0 = 1/\phi^v(g_v)$, and for each $i=1,\ldots,r$, set $s_i = s_0\phi^v(g_{w_i})$, and then define
\[ \psi^v = s_0\phi^v - \sum_{i=1}^r s_i \eta_i \circ \nu^v.\]
As a linear combination of homogeneous quasimorphisms, $\psi^v$ is a homogeneous quasimorphism.  Since $g_v \in \ker(\nu^v)$, it follows that $\eta_i \circ \nu^v(g_v) = 0$ for all $i$, hence $\psi^v(g_v) = s_0\phi^v(g_v) = \phi^v(g_v)/\phi^v(g_v) =1$.  On the other hand, for any $j = 1,\ldots,r$ we have
\[ \psi^v(g_{w_j}) = s_0\phi^v(g_{w_j}) - \sum_{i=1}^r s_i \eta_i (\nu^v(g_j)) = s_0\phi^v(g_{w_j}) - \sum_{i=1}^r s_0\phi^v(g_{w_i}) \delta_{ij} = 0,\]
proving the lemma. \end{proof}

According to \cite[Lemma~4.15]{ABO}, there is an (infinite) generating set for $\Gamma^v$ so that with respect to the resulting word metric, the quasimorphism $\psi^v \colon \Gamma^v \to \mathbb R$ from Lemma~\ref{L:quasimorphism} is a quasi-isometry.  For $v \in \vtx_0$, define $\mathcal K^v = \Gamma^v$ with this choice of word metric and let
\[ \lambda^v \colon \bTheta^v \to \mathcal K^v \]
simply be the map $\widetilde \lambda^v$ (followed by the identification of $\Gamma^v$ with $\mathcal K^v$).
For any $g \in \Gamma$, define $\mathcal K^{gv}$ to be the coset $g \Gamma^v$ with this generating set so that $\widetilde \lambda^{gv}$ defines a map
\[ \lambda^{gv} \colon \bTheta^{gv} \to \mathcal K^{gv}.\]
Carrying this out for every $v \in \vtx_0$,  \eqref{E:equivariance unraveled} implies
\begin{equation} \label{E:equivariance unraveled 2} \lambda^{gw} (g x) = g \lambda^w(x) \end{equation}
for all $w \in \vtx$ and $x \in \bTheta^w$, and $g \in \Gamma$. 

Before we proceed to the proof of Proposition~\ref{prop:KL}, observe that $\Gamma^{gv} = g \Gamma^v g^{-1}$ acts isometrically on $g \Gamma^v$ with respect to any generating set, and thus we can use this to define a generating set for the conjugate so that (any) orbit map is an isometry; in fact, this will just be a conjugate of the generating set for $\Gamma^v$.  In particular, when convenient we will identify $\mathcal K^{gv}$ isometrically with the conjugate $g \Gamma^v g^{-1}$ via such an orbit map.  Conjugating the quasimorphisms $\psi^v$ from the lemma, for $v \in \vtx_0$, we obtain uniform quasi-isometries
\begin{equation} \label{E:quasimorphism general} \psi^w \colon \mathcal K^w \to \mathbb R
\end{equation}
for all $w \in \vtx$, which for an appropriate choice of identification of $\mathcal K^w$ with a conjugate of some $\Gamma^v$, $v \in \vtx_0$, is a quasimorphism (with uniformly bounded deficiency).

\begin{proof}[Proof of Proposition~\ref{prop:KL}]  From the discussion above and Equation \eqref{E:equivariance unraveled 2}, we immediately see that item \eqref{item:equivariance} of the proposition holds.

Next, observe that by adding finitely many generators to the infinite generating set of $\Gamma^{v_0}$ for any $v_0 \in \vtx_0$, changes $\mathcal K^{v_0}$ by quasi-isometry.  On the other hand, the finite generating set described in Equation \eqref{E:finite gen set} for $v_0 \in \vtx_0$ makes $\widetilde \lambda^{v_0}$ a quasi-isometry.  Thus, adding these generators to the infinite generating set does not change the quasi-isometry type of $\mathcal K^{v_0}$, but clearly makes $\lambda^{v_0}$ coarsely Lipschitz.
Therefore, $\lambda^v$ is uniformly coarsely Lipschitz for all $v \in \vtx$, and hence item \eqref{item:lambda_lipschitz} holds for all $v \in \vtx$.

To prove item \eqref{item:lambda_flowlines_bounded}, let $v \in \vtx$ and $x \in \partial \bTheta^v$.  Then $x \in \theta^w$, for some $w \in \vtx$ adjacent to $v$.  As discussed above, we view $\mathcal K^v$ and $\mathcal K^w$ as conjugates $\Gamma^v$ and $\Gamma^w$ of groups $\Gamma^{v_0}$ and $\Gamma^{w_0}$, respectively, for $v_0,w_0 \in \vtx_0$, equipped with their conjugated infinite generating sets.  Let $\psi^v \colon \mathcal K^v \to \mathbb R$ and $\psi^w \colon \mathcal K^w \to \mathbb R$ be the associated uniform quasi-isometric homogeneous quasimorphisms.
The element $g_w \in \Gamma^v$ stabilizes $\ell_{x,\alpha}$ acting by translation on it, and by construction, $\psi^v(g_w) = \psi^v(g_w^n) = 0$ for all $n \in \mathbb Z$.  It follows that every orbit of $\langle g_w \rangle$ acting on $\mathcal K^v$ is uniformly bounded.  Indeed, if $D$ is the deficiency of $\psi^v$, then for any $g \in \mathcal K^v$, we have
\[ |\psi^v(g_w^ng) - \psi^v(g)|  =  |\psi^v(g_w^ng) - \psi^v(g) - \psi^v(g_w^n)| \leq D \]
and therefore $g_w^ng$ and $g$ are uniformly bounded distance apart in $\mathcal K^v$ (since $\psi^v$ is a uniform quasiisometry).

Now, since $g_w^nv = v$, by item \eqref{item:equivariance} of the proposition we have
\[ \lambda^v(g_w^nx) = \lambda^{g_w^nv}(g_w^n x) = g_w^n \lambda^v(x), \]
and since $g_w^n \lambda^v(x)$ is uniformly close to $\lambda^v(x)$, it follows that $\lambda^v$ sends the $\langle g_w \rangle$--orbit of $x$ to a uniformly bounded set.  Since this orbit is $R$--dense in $\ell_{x,\alpha}$ for some uniform $R >0$, and since $\lambda^v$ is uniformly coarsely Lipschitz (by item \eqref{item:lambda_lipschitz}) we see that $\lambda^v(\ell_{x,\alpha})$ has uniformly bounded diameter.  This proves item \eqref{item:lambda_flowlines_bounded}.

For item \eqref{item:KL_qi}, we continue with the assumptions on $v,w$ as above.  Note that since $\psi^v(g_v^n) = n$, using again the fact that $\psi^v$ is a uniform quasi-isometric homogeneous quasimorphism to $\mathbb R$, it follows that for any $x \in \bTheta^v$, the map $n \mapsto \lambda^v(g_v^nx)$ is a uniformly coarsely surjective, uniform quasiisometry $\mathbb Z \to \mathcal K^v$.  Since every orbit of $\langle g_w \rangle$ on $\mathcal K^v$ is uniformly bounded, it follows that for all $n,m \in \mathbb Z$, the two points
$\lambda^v(g_v^ng_w^mx) = g_w^m \lambda^v(g_v^nx)$ and $\lambda^v(g_v^n x)$
are uniformly close to each other.
Likewise, $\lambda^w(g_v^ng_w^mx)$ and $\lambda^w(g_w^mx)$ are also uniformly close to each other.
But this means that
\[ \lambda^v \times \lambda^w(g_v^ng_w^mx) \mbox{ and } (\lambda^v(g_v^nx),\lambda^w(g_w^mx)) \]
are uniformly close, and thus
\[ (n,m) \mapsto \lambda^v \times \lambda^w(g_v^ng_w^mx) \]
is a uniformly coarsely surjective, uniform quasiisometry $\mathbb Z^2 \to \mathcal K^v \times \mathcal K^w$.

On the other hand, the assignment $(n,m) \mapsto g_v^ng_w^m x$ defines a uniform quasiisometry $\mathbb Z^2 \to \bTheta^v \cap \bTheta^w$ since  $\langle g_w \rangle \times \langle g_v \rangle \cong \mathbb Z^2$ acts cocompactly on $\bTheta^v \cap \bTheta^w$ (with uniformity coming from the fact that there are only finitely many $\Gamma$--orbits of pairs $(v,w)$ of adjacent vertices).  Combining these two facts, together with the fact that $\lambda^v$ and $\lambda^w$ are uniformly coarsely Lipschitz, it follows that
\[ \lambda^v \times \lambda^w \colon \bTheta^v \cap \bTheta^w \to \mathcal K^v \times \mathcal K^w \]
is a uniformly coarsely surjective, uniform quasiisometry.  This proves  item \eqref{item:KL_qi}, and completes the proof of the proposition.
\end{proof}

\subsection{A technical lemma}

The goal of this subsection is to prove Lemma \ref{lem:retract_on_the_middle}, whose relevance will only be clear in \S\ref{S:combinatorial HHS}. We prove it here since we have now established the setup for its proof.

We recall that for each $v$, since $\bTheta^v/\Gamma^v$ is a Seifert fibered orbifold, we have have a $\Gamma^v$--equivariant, uniformly biLipschitz homeomorphism $\mu^v \times \rho^v \colon \bTheta^v \to \widetilde S^v \times \mathbb R$, where $\widetilde S^v$ is the (simply connected) surface-with-boundary $\bTheta^v_X \subset E_X$ for some $X \in \partial B_\alpha$ (and $\alpha = \alpha(v)$) and the slices $\{ x \} \times \mathbb R$ (more precisely, the level sets of $(\mu^v)^{-1}(x) \subset \bTheta^v$) are lines for $\bTheta^v$.  These lines project to circle fibers in $\bTheta^v/\Gamma^v$ and we may assume they contain all the lines $\ell_{x,v}$ for all $x \in \theta^v$.

\begin{lemma} \label{L:KL-like} The map $\mu^v \times \lambda^v \colon \bTheta^v \to \widetilde S^v \times \mathcal K^v$ is a uniform, $\Gamma^v$--equivariant quasi-isometry with uniformly dense image.  Moreover, the constant $K_1$ from Proposition~\ref{prop:KL} can be chosen so that for any $v \in \vtx$ and $s \in \mathcal K^v$, the subspace
\[ M(s)=(\lambda^v)^{-1}(N_{K_1}(s)) \subset \bTheta^v,\]
has the property that $\mu^v \times \lambda^v(M(s))$ has uniformly bounded Hausdorff distance to the slice $\widetilde S^v \times \{s\}$, and furthermore $M(s)$ nontrivially intersects every line of $\bTheta^v$.
\end{lemma}  We note that the intersection of $M(s)$ with each line of $\bTheta^v$ is necessarily a uniformly bounded diameter set by the uniform bounded Hausdorff distance condition.
 \begin{proof} All constants will be independent of the specific vertex $v$, so we drop it from the notation.  We write $d$ for all path-metric distances in what follows (the location of points will determine which metric is being used).  Products are given the $L^1$ metric for convenience.  We further let $K$ be the maximum of the coarse Lipschitz constants of $\mu,\rho,\lambda$ and the biLipschitz constant of $\mu \times \rho$, and assume, as we may, that $K \geq 2$.   From the proof of Proposition~\ref{prop:KL}\eqref{item:KL_qi}, if $x$ is any point of a line of $\bTheta$, then $n \mapsto \lambda^v(g_v^nx)$ is a uniformly coarsely surjective, uniform quasi-isometry from $\mathbb Z$ to $\mathcal K$.  Therefore, $\lambda = \lambda^v$ is a uniformly coarsely surjective, uniform quasi-isometry from any line of $\bTheta$ to $\mathcal K$.  We further assume that the coarse surjectivity constants and quasi-isometry constants are also all taken to be $K$.

Let $x,y \in \bTheta$ be any two points.
Since $\mu$ and $\lambda$ are  $K$--coarsely Lipschitz, $\mu \times \lambda$ is $(2K,2K)$--coarsely Lipschitz.  To prove the required uniform lower bound on $\mu \times \lambda$--distances, we note that since $\mu \times \rho$ is a $K$--biLipschitz homeomorphism, it suffices to uniformly coarsely bound $d(\mu \times \rho(x),\mu\times \rho(y))$ from above by $d(\mu \times \lambda(x),\mu \times \lambda(y))$.
For reasons that will become clear shortly, we observe that
\begin{eqnarray} \label{E:funny K4th}  \notag
d(\mu \times \rho(x),\mu \times \rho(y)) & = & d(\mu(x),\mu(y)) + d(\rho(x),\rho(y))\\ 
& \leq & \displaystyle{2 \max\left\{ K^4 d(\mu(x),\mu(y)), d(\rho(x),\rho(y))\right\}}.
\end{eqnarray}
If the maximum is realized by $K^4 d(\mu(x),\mu(y))$, then note that
\begin{eqnarray*} d(\mu \times \rho(x),\mu \times \rho(y)) & \leq & 2K^4 d (\mu(x),\mu(y)) \\
& \leq & 2K^4 d (\mu(x),\mu(y)) + 2K^4 d(\lambda(x),\lambda(y)) \\
&= &2K^4 d(\mu \times \lambda(x),\mu \times \lambda(y)),
\end{eqnarray*}
as required.  

We are left to consider the case that the maximum in \eqref{E:funny K4th} is realized by $d(\rho(x),\rho(y))$, which thus satisfies
\[ d(\rho(x),\rho(y)) \geq K^4 d(\mu(x),\mu(y)).\]
Let $z \in \bTheta$ be such that $\rho(z) = \rho(x)$ and $\mu(z) = \mu(y)$.  Since $\mu(z) = \mu(y)$, $z$ and $y$ lie on a line, and since the restriction of $\lambda$ to this line is a $(K,K)$--quasi-isometry, we have
\[ d(\lambda(z),\lambda(y)) \geq \tfrac1{K} d(\rho(z),\rho(y)) - K = \tfrac{1}K d(\rho(x),\rho(y)) - K.\]
Since $\lambda$ is $K$--coarsely Lipschitz and $\mu \times \rho$ is $K$--biLipschitz, we have
\begin{eqnarray*} 
d(\lambda(x),\lambda(z))& \leq & Kd(x,z) + K\\
& \leq & K^2 (d(\mu(x),\mu(z))+d(\rho(x),\rho(z))) + K \\
&=& K^2 d(\mu(x),\mu(y)) + K \\
&\leq &K^2 \left(\tfrac{1}{K^4} d (\rho(x),\rho(y))\right) + K\\
& = & \tfrac1{K^2} d(\rho(x),\rho(y)) + K
\end{eqnarray*}
Combining the previous two sets of inequalities and the triangle inequality, we have
\begin{eqnarray*}
d(\lambda(x),\lambda(y)) & \geq & d(\lambda(z),\lambda(y)) - d(\lambda(z),\lambda(x))\\
& \geq & \tfrac{1}K d(\rho(x),\rho(y)) - K - \left( \tfrac1{K^2} d(\rho(x),\rho(y)) + K \right)\\
& \geq & \tfrac{K-1}{K^2} d(\rho(x),\rho(y)) - 2K.
\end{eqnarray*}
Combining this inequality with \eqref{E:funny K4th} where we have assumed the maximum is realized by $d(\rho(x),\rho(y))$, we obtain
\begin{eqnarray*}
d(\mu \times \rho(x),\mu \times \rho(y)) & \leq & 2 d(\rho(x),\rho(y)) \\
& \leq & \tfrac{2K^2}{K-1} d(\lambda(x),\lambda(y)) + \tfrac{4K^3}{K-1}\\
& \leq & \tfrac{2K^2}{K-1} d(\mu \times \lambda(x),\mu \times \lambda(y)) +\tfrac{4K^3}{K-1}.
\end{eqnarray*}
which provides the required upper bound.  This completes the proof of the first claim of the lemma.  

For the second claim of the lemma, we now increase $K_1$ from Proposition~\ref{prop:KL} if necessary, so that $K_1 \geq K$.  Observe that
\begin{equation}\label{E:slice preserving-ish} (\mu \times \lambda)\circ (\mu \times \rho)^{-1}(x,t) = (x,\lambda((\mu \times \rho)^{-1}(x,t))).
\end{equation}
That is, $(\mu \times \lambda)\circ(\mu \times \rho)^{-1}$ sends the line $\{x \} \times \mathbb R$ to $\{x\} \times \mathcal K$, for any $x \in \widetilde S$.  As already noted at the start of the proof, restricting to this line, $\lambda$ is $K$--coarsely Lipschitz and $K$--coarsely onto.  Therefore, for any $s \in \mathcal K$ and $x \in \widetilde S$, there exists $t$ so that $\lambda((\mu \times \rho)^{-1}(x,t))$ is within $K_1 \geq K$ of $s$.
Thus, for any line of $\bTheta$, the $\lambda$--image nontrivially intersects $N_{K_1}(s)$, and hence this line nontrivially intersects $M(s)$.  By definition, $\mu \times \lambda$ maps $M(s)$ into $\widetilde S^v \times N_{K_1}(s)$, and by the previous sentence, every point of $\widetilde S \times \{s\}$ is within $K_1$ of some point of $\mu \times \lambda(M(s))$.  Thus, $\lambda \times \mu(M(s))$ has Hausdorff distance at most $K_1$ from $\widetilde S \times \{s\}$, as required.
\end{proof}

As mentioned above, the following technical lemma will be needed in \S\ref{S:combinatorial HHS}. In the statement $M(s)$ is as defined by Lemma~\ref{L:KL-like}.

\begin{lemma}\label{lem:retract_on_the_middle} 
   There is a function $N = N_{\mathcal K} \colon [0,\infty)\to[0,\infty)$ with the following property. Suppose that $v,v_1,v_2\in\vtx$ are so that $d_\tree(v,v_i)=1$. Then for each $s\in\mathcal K^v$ and $t_i\in\mathcal K^{v_i}$ we have
   $$\bar d(M(s)\cap M(t_1),M(s)\cap M(t_2))\leq N(\bar d(M(t_1),M(t_2))).$$
\end{lemma}

 \begin{proof}
It suffices to prove the lemma with $\bar d$ replaced by the path metric $d_{\partial \mathcal B_\alpha}$ on $\partial \mathcal B_\alpha$, since they are uniformly coarsely equivalent.  In fact, it will be convenient to consider the path metric $d_0$ on the union of the three vertex subspaces
\[ \Omega = \bTheta_v \cup \bTheta_{v_1} \cup \bTheta_{v_2},\]
which is also uniformly coarsely equivalent since each vertex space uniformly quasi-isometrically embeds in $\partial \mathcal B_\alpha$.  In this subspace, we will actually prove that the two distances are uniformly comparable.

Now, for each $i=1,2$ the uniform quasi-isometry  $\mu^{v_i} \times \lambda^{v_i} \colon \bTheta^{v_i} \to \widetilde S^{v_i} \times \mathcal K^{v_i}$ from Lemma~\ref{L:KL-like} maps the space $\bTheta^{v} \cap \bTheta^{v_i}$ within bounded Hausdorff distance of a subspace $\partial^v \widetilde S^{v_i} \times \mathcal K^{v_i}$, for a boundary component $\partial^v \widetilde S^{v_i}$  of $\widetilde S^{v_i}$.
Let $\eta \colon \widetilde S^{v_i} \to \partial^v \widetilde S^{v_i}$ be the closest point projection, and then set
\[ (\mu^{v_i} \times \lambda^{v_i})^{-1}  \circ (\eta \times id) \circ (\mu^{v_i} \times \lambda^{v_i} ) \colon \bTheta^{v_i} \to \bTheta^{v_i} \]
where $(\mu^{v_i} \times \lambda^{v_i} )^{-1}$ is a coarse inverse of $\mu^{v_i} \times \lambda^{v_i} $ with $\mu^{v_i} \circ (\mu^{v_i} \times \lambda^{v_i} )^{-1}(x,s) = x$ (c.f.~Equation~\eqref{E:slice preserving-ish}). This map is a 
uniformly coarsely Lipschitz, coarse retraction of $\bTheta^{v_i}$ onto $\bTheta^{v_i} \cap \bTheta^v$.  Moreover, this sends $M(t_i)$, which is uniformly close to the $(\mu^{v_i} \times \lambda^{v_i})^{-1}$--image of $\widetilde S^{v_i} \times \{t_i\}$, to a uniformly bounded neighborhood of $M(t_i) \cap \bTheta^v$.  Consequently, 
\begin{equation} \label{E:neighbor distances 1} d_0(M(t_1),M(t_2)) \asymp d_0( M(t_1) \cap \bTheta^v, M(t_2)  \cap \bTheta^v)
\end{equation}
with uniform constants.  

Next, observe that $M(t_i) \cap \bTheta^v \subset \bTheta^v \cap \bTheta^{v_i} \subset \bTheta^v$. 
\begin{claim} The quasi-isometry $\mu^v \times \lambda^v$ maps $M(t_i) \cap \bTheta^v$ within a uniformly bounded Hausdorff distance of the  slice $\{ z_i \} \times \mathcal K^v \subset \widetilde S^v \times \mathcal K^v$, for each $i=1,2$, where $(z_i,t_i')$ is a point in the $\mu^v \times \lambda^v$--image of $M(s) \cap M(t_i)$.
\end{claim}
Assuming the claim, we note then that
\begin{eqnarray*} 
d_0( M(t_1) \cap \bTheta^v, M(t_2)  \cap \bTheta^v) & \asymp & d_{\widetilde S^v \times \mathcal K^v} (\{z_1\} \times \mathcal K^v, \{z_2\} \times \mathcal K^v)\\
& = & d_{\widetilde S^v}(z_1,z_2)\\
 & \asymp & d_0(M(s) \cap M(t_1),M(s) \cap M(t_2))
\end{eqnarray*}
again with uniform constants.  Combining this coarse equation with \eqref{E:neighbor distances 1} we get the required uniform estimate
\[ d_0(M(s) \cap M(t_1),M(s) \cap M(t_2)) \asymp d_0(M(t_1),M(t_2)). \]

Fix $i=1$ or $2$ and we prove the claim.  Since $\lambda^v \times \lambda^{v_i}$ is $K_1$--coarsely surjective (Proposition~\ref{prop:KL}(\ref{item:KL_qi}), there exists some point $y_i \in \bTheta^v \cap \bTheta^{v_i}$ with $(\lambda^v(y_i),\lambda^{v_i}(y_i))$ within distance $K_1$ of $(s,t_i) \in \mathcal K^v \times \mathcal K^{v_i}$.  Therefore, $y_i \in M(s) \cap M(t_i)$ and we set $\mu^v\times \lambda^v(y_i) = (z_i,t_i')$.  

Next, we observe that $\lambda^{v_i}$ is uniformly coarsely constant on any line of $\bTheta^v$ contained in $\bTheta^v \cap \bTheta^{v_i}$ by Proposition~\ref{prop:KL}(\ref{item:lambda_flowlines_bounded}) and uniformly coarsely Lipschitz by Proposition~\ref{prop:KL}(\ref{item:lambda_lipschitz}). 
Hence, the line
\[ (\mu^v)^{-1}(z_i) = (\mu^v \times \rho^v)^{-1}(\{z_i\} \times \mathbb R)\]
of $\bTheta^v$ maps under $\lambda^v \times \lambda^{v_i}: \bTheta^v \cap \bTheta^{v_i} \to \mathcal K^v \times \mathcal K^{v_i}$ into a neighborhood of uniformly bounded radius of $\mathcal K^v\times \{t_i\}$. Therefore, any point in the image of the line in $\mathcal K^v \times \mathcal K^{v_i}$ lies uniformly close to a point in $\lambda^v \times \lambda^{v_i}(M(t_i)\cap \bTheta^v)$ by Proposition~\ref{prop:KL}(\ref{item:KL_qi}) (which guarantees that any point in $\mathcal K^v\times \{t_i\}$ is $K_1$--close to a point in the image of the subspace $\bTheta^v \cap \bTheta^{v_i}$). Therefore any point in the line lies uniformly close to some point in $M(t_i)\cap \bTheta^v$ since $\lambda^v \times \lambda^{v_i}$ is a uniform quasi-isometry again by Proposition~\ref{prop:KL}(\ref{item:KL_qi}). 

On the other hand, Lemma~\ref{L:KL-like} implies $\mu^{v_i} \times \lambda^{v_i}(M(t_i))$ is uniformly bounded Hausdorff distance from the slice $\widetilde S^{v_i} \times \{t_i\} \subset \widetilde S^{v_i} \times \mathcal K^{v_i}$.  Moreover, since $M(t_i)$ meets every line of $\bTheta^{v_i}$ (Lemma~\ref{L:KL-like} again), it follows that
\[ \mu^{v_i} \times \lambda^{v_i}(M(t_i) \cap \bTheta^v) \]
is uniformly bounded Hausdorff distance to the quasi-line $\partial^v \widetilde S^{v_i} \times \{t_i\}$ (see the proof of Lemma~\ref{L:KL-like}). In particular, $M(t_i) \cap \bTheta^v$ is itself a uniform quasi-line and consequently lies within a uniformly bounded neighborhood of the line $(\mu^v)^{-1}(z_i)$.  Since this line maps within  within a uniformly bounded Hausdorff distance of the slice $\{z_i\} \times \mathcal K^v$ in $\widetilde S^v \times \mathcal K^v$ by $\mu^v \times \lambda^v$, we see that $M(t_i) \cap \bTheta^v$ does as well
\end{proof}

\subsection{Coned-off surfaces} 
\label{subsec:coned-off}

For $v\in \vtx$, we define $\bXi^v$ to be the graph whose vertices are all $w\in\vtx$ so that $d_\tree(v,w)=1$, and with edges connecting the pairs $w,w'$ whenever $\bTheta^w\cap \bTheta^{w'}\neq \emptyset$.
As such, vertices $w\in \vtx$ are in bijective correspondence with the boundary components of $\bTheta^v$ and there is an ``inclusion'' map
\[i^v\colon \partial \bTheta^v \to \bXi^v\]
that sends any point $x \in \partial \bTheta^v$ to the vertex $w$ for which $x \in \bTheta^w$.
In light of the following lemma, we note that we could alternately define the edges of $\bXi^v$ in terms of subspaces lying within bounded distance of each other, and produce a space quasi-isometric to $\bXi^v$.

\begin{lemma}\label{lem:parameter_for_Xi}
There is a function $N = N_{\bXi} \colon[0,\infty)\to [0,\infty)$ so that whenever $v,w_1,w_2\in \vtx$ satisfy $\bar d(\bTheta^{w_1}\cap \bTheta^v,\bTheta^{w_2}\cap \bTheta^v)\leq r$,
the sets $\bTheta^{w_1}\cap \bTheta^v$ and $\bTheta^v\cap\bTheta^{w_2}$ may be connected via a concatenation of at most $N(r)$ paths, each of which is contained in a set of the form $\bTheta^v\cap \bTheta^{w}$.
\end{lemma}

\begin{proof}
If $\bar d(\bTheta^{w_1}\cap \bTheta^v,\bTheta^{w_2}\cap \bTheta^v)\leq r$, then there are cone points $p_i\in\bTheta^{w_i}\cap \theta^v$ within distance $r + 3M$, where $M$ is the bound on the width of a strip and length of a saddle connection from Lemma~\ref{L:strip-and-saddle-bound}. Since the path metric $d_{\theta^v}$
on $\theta^v$ is coarsely equivalent to the subspace metric, $d_{\theta^v}(p_1,p_2)$ is bounded in terms of $r$. The path metric on $\theta^v$ is biLipschitz equivalent to the $\ell^1$--metric on the product $\theta^v_{X_v}\times \mathbb R$. Since each edge of $\theta^v_{X_v}$ has definite length, there is a path from $p_1$ to $p_2$ in $\theta^v$ obtained by concatenating boundedly many (in terms of $r$) paths $\alpha_i$ with $f_{X_v}(\alpha_i)$ a saddle connection in $\theta^v_{X_v}$. Since each $\alpha_i$ is contained in some $\bTheta^{w}$, we are done.
\end{proof}

\begin{corollary}
\label{cor:inclusion_Lipschitz}
For any $x,y\in \partial \bTheta^v$, we have $d_{\bXi^v}(i^v(x), i^v(y))\le N(d_{\bTheta^v}(x,y))+1$.
\end{corollary}
\begin{proof}
Let $w = i^v(x)$ and $w' = i^v(y)$. Since $d_{\bTheta^v}$ and $\bar d$ are path metrics, we have $\bar d(x,y) \le d_{\bTheta^v}(x,y)$. By Lemma~\ref{lem:parameter_for_Xi}, $x$ and $y$ may be joined by a concatenation $\alpha_1\dotsb\alpha_k$ of $k \le N(d_{\bTheta^v}(x,y))+2$ paths $\alpha_j$ each of which lies in some $\bTheta^v\cap \bTheta^{w_j}$, and where $w = w_1$ and $w' =w_k$. For successive paths $\alpha_j,\alpha_{j+1}$, the vertices $w_j,w_{j+1}$ are adjacent in $\bXi^v$ by definition. Therefore $d_{\bXi^v}(w,w') \le k-1 \le N(d_{\bTheta^v}(x,y)) +1$.
\end{proof}

\begin{lemma}\label{lem:Xi_hyp}
Each $\bXi^v$ is uniformly quasi-isometric to a tree.  In particular, there exists $\delta > 0$ so that each $\bXi^v$ is $\delta$--hyperbolic. Moreover, $\bXi^v$ has at least two points at infinity.
\end{lemma}

\begin{proof}
We appeal to Proposition \ref{prop:bottleneck} and show that for any vertices $w,w'$ of $\bXi^v$ there exists a path $\gamma(w,w')$ so that any path from $w$ to $w'$ passes within distance $3$ of each vertex of $\gamma(w,w')$.

First, note that $\bXi^v$ is isomorphic to the intersection graph of the collection of strips in $\bTheta_0^v$.  For each strip we have a vertex, and for each saddle connection of the spine $\theta_0^v$, there is an edge of $\bXi^v$ that connects the vertices corresponding to the strips that contain the saddle connection.  For each cone point in the spine $\theta_0^v$, there is also a complete graph on the vertices corresponding to strips that contain this cone point.  This accounts for all edges (because intersections of strips either arise along saddle connections or single cone points), and we note that the closure of each edge of the first type separates $\bXi^v$ into two components.  

Suppose $w,w' \in \bXi^v$ are two vertices, and let $x,x' \in \theta_0^v$ be points in the (boundaries of the) strips corresponding to $w$ and $w'$, respectively, that are closest in $\bTheta_0^v$, and consider the geodesic in $\theta_0^v$ connecting these points, which is a concatenation of saddle connections $\sdl_1 \sdl_2 \cdots \sdl_n$.  For each $1 \leq i \leq n$, let $w_i^\pm$ be the vertices corresponding to the two strips $A_i^\pm$ that intersect in the saddle connection $\sdl_i$. We can form an edge path $\gamma(w,w')$ in $\bXi^v$, containing $w,w'$, and the $w_i^+$ as vertices, since $A_i^+ \cap A_{i+1}^+ \neq \emptyset$. 
Observe that any path from $x$ to $x'$ must pass through the union $A_i^+ \cup A_i^-$, for each $i$, since $x$ and $x'$ lie in the closures of distinct components of $\bTheta_0^v \smallsetminus (A_i^+ \cup A_i^-)$.

Now let $w = w_0,w_1,\ldots,w_k = w'$ be the vertices of an edge path connecting $w$ to $w'$ in $\bXi^v$.  For any points in the strips corresponding to $w$ and $w'$, respectively, it is easy to construct a path in $\bTheta_0^v$ between these points that decomposes as a concatenation $\nu_1 \nu_2 \cdots \nu_k$ so that $\nu_j$ is contained entirely in the strip corresponding to $w_j$.  From the previous paragraph, this path must pass through $A_i^+ \cup A_i^-$, for each $i =1,\ldots,n$.  It follows that for each $1 \leq i \leq n$, the edge path must meet the union of the stars $star(w_i^+) \cup star(w_i^-)$.  Since these stars intersect, their union has diameter at most $3$, and we are done.

We now show that $\bXi^v$ contains a quasi-geodesic line. Consider strips $A_i$ of $\bTheta_0^v$, for $i\in\mathbb Z$, such that for all $i$ we have
\begin{itemize}
\item $A_i$ and $A_{i+1}$ share a saddle connection;
\item $A_{i-1}$ and $A_{i+2}$ lie on distinct components of the complement of the interior of $A_i\cup A_{i+1}$ in $\bTheta_0^v$.
\end{itemize}

The $A_i$ give a bi-infinite path in $\bXi^v$, and we now show that this path is a quasi-geodesic. Fix integers $m,n$ and consider a geodesic $\gamma$ in $\bXi^v$ from $A_m$ to $A_n$ (where we think of the strips themselves as vertices of $\bXi^v$ for convenience). Then for each $m< k<n-1$ we have that $\gamma$ needs to contain a vertex $v(k)$ which, regarded as a strip, intersects $A_k$ or $A_{k+1}$. Indeed, the interior of $A_k\cup A_{k+1}$ separates $A_m$ from $A_n$, and the sequence of vertices of $\gamma$ corresponds to a connected union of strips containing $A_m$ and $A_n$. Moreover, there is no strip intersecting both $A_k$ and $A_{k'}$ if $|k-k'|\geq 3$, and in particular we have $v(k)\neq v(k')$ if $|k-k'|\geq 4$. These observations imply that $\gamma$ contains at least $\lfloor(n-m-2)/4\rfloor$ vertices, so that geodesics connecting $A_m$ to $A_n$ have length comparable to $n-m$, and the $A_i$ form a quasi-geodesic line as required.
\end{proof}

\subsection{Windows and bridges}

Recall that $\Sigma \subset E$ and $\bar\Sigma \subset \bar E$ are the sets of all singular points in all fibers in $E$ and $\bar E$, respectively; see \S\ref{S:E and bar E}.  For $v\in \vtx$, consider the set $\bar{\Sigma}^v$  of points in $\bar{\Sigma}$ that are inside some $v$--spine, as well as those points $\bar{\Sigma}^{\notin v}$ that are outside every $v$--spine:
\[ \bar{\Sigma}^v = \bigcup_{X \in \bar D} (\theta^v_X\cap \Sigma)\qquad\text{and}\qquad \bar{\Sigma}^{\not\in v} = \bar{\Sigma} \setminus \bar{\Sigma}^v.\]

For each $Y\in \bar D$ we now define a \emph{window map} $\Pi^v_Y \colon \bar{\Sigma} \to  \pow(\partial \bTheta_Y^v)$ from cone points to the set of subsets of the boundary $\partial \bTheta_Y^v \subset \bTheta^v_Y$. There are two cases. Firstly,
\[\text{for $x\in \bar{\Sigma}^{\notin v}$,}\quad \Pi^v_Y(x) = \big\{ z \in \partial \bTheta^v_Y \mid [f_Y(x),z] \cap \bTheta^v_Y = \{z\} \big\}.\]
In words, $\Pi^v_Y(x)$ is the union of entrance points in $\bTheta^v_Y$ of any flat geodesic in $E_Y$ from $f_Y(x)$ to $\bTheta^v_Y$ (basically the closest point projection in $E_Y$), and we call it the \emph{window} for $x$ in $\bTheta^v_Y$.
Observe that for any $X,Y\in \bar D$ we have $f_{X,Y}(\Pi_Y^v(x)) = \Pi_X^v(x)$.  
The second case, that of $\bar{\Sigma}^{v}$, is handled slightly differently:
\begin{align*}
\text{for $x\in \bar{\Sigma}^v$,}\quad \Pi^v_Y(x) &= f_{Y,X}(\Pi^v_X(x)), \quad\text{where}
\quad X = c_v(\pi(x))\in \partial B_v \\
\text{and}\quad \Pi^v_X(x) &= \big\{z \in \partial \bTheta^v_X \mid z  \text{ is a closest cone point to } f_X(x)\in \theta^v_X\big\}.
\end{align*}
Thus affine invariance $\Pi^v_Z(x) = f_{Z,Y}(\Pi^v_Y(x))$ is built directly into the definition.

Now for any $Y \in \bar D$ and $x,y \in \bar{\Sigma}$, we define
\[ d_Y^v(x,y) = \diam(\Pi^v_Y(x) \cup \Pi^v_Y(y)),\]
where the distance is computed in the path metric on $E_Y$ (or equivalently on $\bTheta^v_Y$).
Finally, we extend window maps to arbitrary subsets by declaring
\[\text{for $U\subset \bar E$,}\quad \Pi^v_Y(U) = \Pi^v_Y(U\cap \bar{\Sigma}) = \bigcup_{x\in \bar{\Sigma}\cap U} \Pi^v_Y(x)\]
which has the same effect as defining $\Pi^v_Y(x) = \emptyset \subset \partial \bTheta^v_Y$ for $x\notin \bar{\Sigma}$.

\begin{lemma} \label{lem:different spine}
If $x,y \in \bar{\Sigma}^{\not \in v}$ satisfy $f_0(x) = f_0(y)$, then $\Pi^v_Y(x) = \Pi^v_Y(y)$ for all $Y\in \bar D$.
\end{lemma}
Recall that $f_0 = f_{X_0}$.
\begin{proof}
This is immediate since $f_0(x) = f_0(y)$ if and only if $f_Y(x) = f_Y(y)$ for all $Y \in \bar D$, and $\Pi^v_Y(x)$ is defined just in terms of $f_Y(x) = f_Y(y)$.
\end{proof}

The following gives a counterpoint to Lemma~\ref{lem:different spine} for points in $\bar{\Sigma}^{v}$.

\begin{lemma} \label{lem:combo Lipschitz}
 There exists $K_2>0$ such that for any $v \in \vtx$ the following holds:  If $x, y \in \bar{\Sigma}^v$ satisfy $f_0(x) = f_0(y)$ and either
 \begin{enumerate}
 \item $x$ and $y$ are connected by a horizontal geodesic of length  $\leq 1$, or
 \item $x$ and $y$ are contained in $\partial \mathcal{B}_{w}$ for some $w \in \vtx$ with $\alpha(w) \neq \alpha(v)$,
 \end{enumerate}
 then $d_X^v(x,y) \leq K_2$, where $X = c_v(\pi(x))$.
 \end{lemma}
 \begin{proof} 
Set $X = c_v(\pi(x))$ and $Y = c_v(\pi(y))$. Since $c_v\colon D \to \partial B_v$ is $1$--Lipschitz and $\diam(c_v(\partial B_w))$ is uniformly bounded for all such $w$, either condition (1) or (2) gives a uniform bound $K>0$ on the distance between $X$ and $Y$. Hence $f_{X,Y}$ is $e^{K}$--biLipschitz. The distance between $f_Y(y)\in \theta^v_Y$ and its closest cone points $\Pi^v_Y(y)$ in $\partial \bTheta^v_Y$ is also uniformly bounded by $2M$, by Lemma~\ref{L:strip-and-saddle-bound}. The same holds for the distance between $f_X(x)$ and $\Pi^v_X(x)$. It follows that $\Pi^v_X(y) = f_{X,Y}(\Pi^v_Y(y))$ lies within distance $2e^KM$ of $f_X(x) = f_{X,Y}(f_Y(y))$ and hence within distance $2e^KM  + 2M$ of $\Pi^v_X(x)$.
 \end{proof}

The next lemma explains that the image of $\Pi^v_Y$ is   not so far from being a point.
 
\begin{lemma}[Window lemma]\label{lem:window}
For any $v \in \vtx$, $Y\in \bar D$, and $x \in \bar{\Sigma}^{\not\in v}$, the window $\Pi^v_Y(x) \subset \partial \bTheta^v_Y$ is either a cone point or a single saddle connection.
\end{lemma}

\begin{proof} 
Recall that each cone point in the flat surface $E_Y$ has total angle at least $3\pi$, and that $E_Y$ is a unique geodesic space in which a concatenation of saddle connections is geodesic if and only if successive saddle connections subtend an angle of at least $\pi$ on each side.

If $f_Y(x)\in \partial \bTheta^v_Y$ then clearly $\Pi^v_Y(x)$ is the cone point $f_Y(x)$ itself. So suppose $f_Y(x)\notin \bTheta^v_Y$ and let $\ell$ be the component of $\partial \bTheta^v_Y$ separating $f_Y(x)$ from $\theta^v_Y$ in $E_Y$, so that $\Pi^v_Y(x) \subset \ell$. Take any flat geodesic $[f_Y(x), z]$ in $E_Y$ from $f_Y(x)$ to a cone point $z\in \ell$.  The geodesic $[f_Y(x), z]$ is a concatenation of saddle connections and first meets $\ell$ in some cone point $p$. Since the total cone angle at $p$ is at least $3\pi$ and the angle at $p$ along the side of $\ell$ containing $\bTheta^v$ is exactly $\pi$, the last saddle connection $\delta$ in the geodesic $[f_Y(x), p] \subset [f_Y(x),z]$ must make an angle of at least  $\pi$ with one of the two halves of $\ell$ determined by $p$. It follows that concatenating $[f_Y(x),p]$ with that half of $\ell$ gives an infinite geodesic ray in $E_Y$. 
Hence, by uniqueness of geodesics, the geodesic from $f_Y(x)$ to any cone point on that side of $p$ evidently passes through $p$.

\begin{figure}
\label{fig:windows}
\begin{tikzpicture}[scale=1.5,inner sep=0pt, dot/.style={fill=black,circle,minimum size=4pt}]
\pgfmathsetmacro\sw{.5};
\pgfmathsetmacro\sh{4};
\coordinate (origin) at (0,-2);
\draw [fill,color=blue,opacity=.3] (origin) rectangle node[opacity=1,color=black,rotate=90] {\bf The strip between $\theta^v_Y$ and $f_Y(x)$} ++(\sw,\sh);
\draw[thick] (origin) -- ++(0,\sh);
\draw[thick] (origin) ++(\sw,0) -- ++(0,\sh);
\foreach \x in {.15, .3, .65, .9}
{
  \path (origin) ++(0,\x*\sh) node[dot] {};
  \draw[thick] (origin) ++(-1.2*\sw,\x*\sh) node[dot] {} -- ++(1.2*\sw,0);
  \draw[thick] (origin) ++(-1.2*\sw,\x*\sh) node[dot] {} -- ++(140:.8*\sw);
  \draw[thick] (origin) ++(-1.2*\sw,\x*\sh) node[dot] {} -- ++(235:.6*\sw);
}
\path (origin) ++(2.6*\sw,.6*\sh) coordinate (pt);
\path (pt) node[dot] {};
\draw[ultra thick,red] (origin) ++(\sw,.5*\sh) -- node (spot) {} ++(0,.2*\sh);
\path (origin) ++(-2*\sw,.5*\sh) node {The $v$ spine, $\theta^v_Y$};
\draw[thick,red,<-] (spot) -- ++(20:2.5*\sw) node[above right,color=black] {The window $\Pi^v_Y(x)$};
\draw[thick] (origin) ++(\sw,.5*\sh) node[below right] {\hspace{5pt}$\ge \pi$} -- (pt);
\draw[thick] (origin) ++(\sw,.7*\sh) node[above right] {\hspace{5pt}$\ge \pi$} -- (pt);
\foreach \x in {.1, .2, .5, .7, .85, 1}
{
  \path (origin) ++(\sw,\x*\sh) node[dot] {};
}
\draw[thick] (pt) -- ++(15:.8*\sw) node[dot] {}
                        -- ++(-12:\sw) node[dot] {}
                        -- ++(-45:\sw*1.2) node[dot] {}
                        -- ++(30:\sw) node[dot] {} node[below right] {$f_Y(x)$};
\end{tikzpicture}
\caption{The window (shown in red) for $x$ in the thickend spine $\bTheta^v_Y$} \label{F:window}
\end{figure}

If both angles between $\delta$ and $\ell$ at $p$ are at least $\pi$, then any geodesic from $f_Y(x)$ to $\ell$ passes through $p$. Hence $p$ is the unique point in $\partial \bTheta^v_Y$ closest to $f_Y(x)$ and $\Pi^v_Y(x) = \{p\}$ is a cone point as required.  Otherwise, consider the flat geodesic from $f_Y(x)$ to the adjacent cone point $p'$ on the other side of $p$ along $\ell$. The last saddle connection of this geodesic must also make an angle with $\ell$ of at least $\pi$ on one side.  This cannot be the side containing $p$, or else the geodesic from $f_Y(x)$ to $p$ would pass through $p'$ contradicting our choice of $p$. Hence any geodesic from $f_Y(x)$ to a cone point on the opposite side of $p'$ must pass through $p'$. Therefore $\Pi^v_Y(x)$ is the saddle connection between $p$ and $p'$, and we are done.  See Figure~\ref{F:window}.
\end{proof}

The following lemma gives us partial control over the window for points in adjacent vertex spaces in the same Bass--Serre tree.

\begin{lemma}[Bridge lemma] \label{lem:bridge}
For any $v, w \in \vtx$ with $d_{\tree}(v,w) = 1$, any $Y \in \bar D$, and any component $U \subset E_Y \smallsetminus \theta_Y^w$ not containing $\theta^v_Y$, there exists a (possibly degenerate) saddle connection $\delta_U \subset \partial \bTheta^v_Y$ with the following property: Every $x \in \bar{\Sigma}$ with $f_Y(x) \in \overline{U}$ satisfies $\Pi^v_Y(x) \subset \delta_U$.
 \end{lemma}

We call $\delta_U$ the {\em bridge} for $U$ in $E_Y$.  It is clear from the construction in the proof below that $f_{Z,Y}(\delta_U)$ is the bridge for $f_{Z,Y}(U)$, for any $Z \in \bar D$.

\begin{proof}
Let $U$ be as in the statement and $W$ be the component of $E_Y \smallsetminus \theta^w_Y$ containing $\theta^v_Y$.  Let $\gamma_U = \overline{U} \cap \theta^w_Y$ and $\gamma_W = \overline{W} \cap \theta^w_Y \subset \partial \bTheta^v_Y$, which are both bi-infinite flat geodesics in $\theta^w_Y$.

If $\gamma_U \cap \gamma_W = \emptyset$, then there is a unique geodesic between them in $\theta^w_Y$, and we take $\delta_U$ to be the endpoint of this geodesic which lies along $\gamma_W$.
On the other hand, if $\gamma_U \cap \gamma_W \neq \emptyset$, then their intersection is contained in the boundary of a strip along $\theta^v_Y$ and another along $\theta^w_Y$.  Two distinct strips in the same direction that intersect do so in either a single point or a single saddle connection, and hence $\gamma_U \cap \gamma_W$ is a point or single saddle connection, and we call $\delta_U$. See Figure~\ref{F:bridges}.

\begin{figure}[htb]
\label{fig:bridges}
\begin{tikzpicture}[scale=1.2,inner sep=0pt, dot/.style={fill=black,circle,minimum size=3pt}]
\coordinate (origin) at (0,0);
\pgfmathsetmacro\sw{.5};
\pgfmathsetmacro\sh{3};
\fill[color=blue,opacity=.3] (origin) rectangle ++(\sw,\sh);

\path (origin) ++(0,.5*\sh) coordinate (v) node[dot]{};
\draw[thick] (v) -- ++(-90:.4*\sh) coordinate (v1) node[dot]{};
\draw[thick] (v) -- ++(180:.8) coordinate (v2) node[dot]{};
\draw[thick] (v) -- ++(90:.35*\sh) coordinate (v3) node[dot]{};
\draw[thick] (v1) -- ++(-90:.1*\sh) coordinate (v11);
\draw[thick] (v1) -- ++(210:.5);
\draw[thick] (v3) -- ++(90:.15*\sh);
\draw[thick] (v3) -- ++(180:.5) coordinate (v31) node[dot]{};
\draw[thick] (v3) -- ++(180:1) coordinate (v32) node[dot]{};
\draw[thick] (v3) -- ++(180:1.5);
\draw[thick] (v31) -- ++(135:.6);
\draw[thick] (v32) -- ++(230:.6);
\draw[thick] (v2) -- ++(180:.5) coordinate (v21) node[dot]{};
\draw[thick] (v2) -- ++(180:1);
\draw[thick] (v21) -- ++(140:.5);
\draw[thick] (v2) -- ++(230:.5);

\path (origin) ++(\sw,.7*\sh) coordinate (w) node[dot]{};
\draw[thick] (w) -- ++(-90:.3*\sh) coordinate (w1) node[dot]{};
\draw[thick] (w) -- ++(-90:.5*\sh) coordinate (w2) node[dot]{};
\draw[thick] (w) -- ++(-90:.7*\sh);
\draw[thick] (w2) -- ++(-10:.7) coordinate (w21) node[dot]{};
\draw[thick] (w21) -- ++(-40:.7) coordinate (w211);
\draw[thick] (w1) -- ++(-5:1) coordinate (w11) node[dot]{};
\draw[thick] (w11) -- ++(-20:.8) coordinate (w111);
\draw[thick] (w11) -- ++(30:.8);
\draw[thick] (w) -- ++(90:.3*\sh);
\draw[thick] (w) -- ++(10:1.1) coordinate (w3) node[dot]{};
\draw[thick] (w3) -- ++(-30:1.1) coordinate (w31);
\draw[thick] (w3) -- ++(80:.7) coordinate (w32);

\filldraw[opacity=.2,color=teal] (w3) -- (w31) to[out=10,in=20] (w32) -- (w3);
\path (w21) ++(15:.6) node {$U'$};
\filldraw[opacity=.2,color=green] (w1) -- (w2) -- (w21) -- (w211) to[out=0,in=280] (w111) -- (w11) -- (w1);
\path (w3) ++(20:.5) node {$U$};

\draw[color=red,ultra thick] (w1) -- node[label={[label distance=1]0:{$\delta_{U'}$}}] {} (w2);
\filldraw[color=red] (w1) circle[radius=.06];
\filldraw[color=red] (w2) circle[radius=.06];
\filldraw[color=red] (w) circle[radius=.06] node[label={[label distance=2.8]65:{$\delta_{U}$}}] {};

\path (w1) ++(35:.8) node {$\theta_Y^{w}$};
\path (v) ++(235:1) node {$\theta^v_Y$};
\end{tikzpicture}

\caption{Examples of bridges (in red), and the proof of Lemma \ref{lem:bridge}.} \label{F:bridges}
\end{figure}

Now consider any point $x \in \bar{\Sigma}$ with $f_Y(x) \in \overline{U}$. Observe that $x\in \bar{\Sigma}^{\notin v}$ so that $\Pi^v_Y(x)$ falls under the first case of the window definition.
Further, any flat geodesic from $f_Y(x)$ to $\bTheta^v_Y$ must pass through both $\gamma_U$ and $\gamma_W$, and hence must pass through $\delta_U \subset \partial \bTheta^v_Y$.  It follows that $\Pi^v_Y(x) \subset \delta_U$, as required.  
\end{proof}

The following is an easy consequence of the previous lemma.
  
\begin{corollary}\label{cor:bridge same}
For any $v, w \in \vtx$ with $2 \leq d_{\tree}(v,w) < \infty$ and $Y \in \bar D$, there exists a (possibly degenerate) saddle connection $\delta^v_Y(w) \subset \partial \bTheta^v_Y$ so that if $x \in \bar{\Sigma}$ has $f_Y(x) \in \bTheta^w_Y$, then $\Pi^v_Y(x) \subset \delta^v_Y(w)$. In particular, $\Pi^v_Y(\bTheta^w)\subset \delta^v_Y(w)$.
\end{corollary}

\begin{proof}
There exists $u \in T_{\alpha(v)}^{(0)}$ between $v$ and $w$ with $d_{\tree}(u,v) = 1$, and a component $U  \subset E_Y \setminus \theta^u_Y$ whose closure contains $\bTheta^w_Y$.  Setting $\delta^v_Y(w) = \delta_U$ and applying Lemma \ref{lem:bridge} completes the proof.
\end{proof}

The next corollary is similar. 

\begin{corollary}\label{cor:bridge diff}
If $v, w \in \vtx$ satisfy $d_{\tree}(v,w) = \infty$ and $\theta^w_0 \cap \theta^v_0 = \emptyset$, then for each $Y\in \bar D$ there is a connected union $\delta^v_Y(w) \subset \partial \bTheta^v_Y$ of at most two saddle connections such that $\Pi^v_Y(x) \subset \delta^v_Y(w)$ for all $x\in \bar{\Sigma}^w$. In particular $\Pi^v_Y(\theta^w)\subset \delta^v_Y(w)$.
\end{corollary}

\begin{proof}
From the hypotheses, $\theta^w_Y$ is contained in some component $W \subset E_Y \smallsetminus \theta^v_Y$.   Let $u \in \vtx$ be such that $d_{\tree}(v,u) = 1$ and $\theta^u_Y \subset W$.  Since $\alpha(u) = \alpha(v) \neq \alpha(w)$,  $\theta^u_Y$ and $\theta^w_Y$ can intersect in at most one point. If $\theta^u_Y\cap \theta^w_Y= \emptyset$, then $\theta^w_Y$ is contained in a component $U$ of $E_Y\setminus \theta^u_Y$ disjoint from $\theta^v_Y$; thus $\Pi^v_Y(\theta^w_Y)$ is contained in the bridge $\delta_U$ for $U$ by Lemma~\ref{lem:bridge}. 
Otherwise $\theta^u_Y\cap \theta^w_Y \ne \emptyset$, and we claim there is a cone point $p\in \theta^u_Y$ such that $p\in \overline{U}$ for every component $U$ of $E_Y\setminus \theta^u_Y$ that intersects  $\theta^w_Y$. Indeed, if $\theta^u_Y\cap \theta^w_Y$ is a cone point, we take $p$ to be this intersection point, and if not  $\theta^u_Y\cap \theta^w_Y$ is an interior point of a saddle connection of $\theta^u_Y$ and we may take $p$ to be either of its endpoints.
For each component $U \subset E_Y\setminus \theta^u_Y$ that intersects $\theta^w_Y$ we then have $\Pi^v_Y(p) \subset \delta_U$, where $\delta_U$ is the bridge for $U$.  Since $f_Y(\bar{\Sigma}^w)\subset \theta^w_Y$ is contained in the union of the closures of such $U$, it follows that $\Pi^v_Y(\bar{\Sigma}^w)$ is contained in a union of saddle connections along $\partial \bTheta^v_Y$, all of which contain $\Pi^v_Y(p)$, and hence is a connected union of at most two saddle connections.  This completes the proof.
\end{proof}

The final case to consider is that of spines in different directions that intersect:

\begin{lemma}\label{lem:crossing spines}
There exists $K_3>0$ such that if $v, w \in \vtx$ with $d_{\tree}(v,w) = \infty$ and $\theta^w_0 \cap \theta^v_0 \neq \emptyset$, then $\diam (\Pi^v_Y(\theta^w))  < K_3$ for all $Y \in c_v(\partial B_w)$.
 \end{lemma}

\begin{figure}[htb]
\usetikzlibrary{decorations.pathreplacing}
\begin{tikzpicture}[scale=1.8,inner sep=0pt, dot/.style={fill=black,circle,minimum size=3pt},rotate=0]

\pgfmathsetmacro\spokes{5};
\pgfmathsetmacro\spokeslessone{\spokes-1};
\pgfmathsetmacro\angle{360/\spokes};
\pgfmathsetmacro\rad{2};
\pgfmathsetmacro\inrad{.6*\rad};
\pgfmathsetmacro\SpineAngleOffset{\angle/10};
\pgfmathsetmacro\WindowAngleOffset{\angle/4};

\coordinate (origin) at (0,0);
\path (origin) node[dot]{};
\foreach \i in {1,...,\spokes}
{
  \path (origin) ++(\i*\angle:\rad) coordinate (v\i);
  \draw[thick,dashed] (origin) -- (v\i);

  \path (origin) ++(\i*\angle+\SpineAngleOffset:\rad) coordinate (sr\i);
  \path (origin) ++(\i*\angle+\angle-\SpineAngleOffset:\rad) coordinate (sl\i);
  \path (origin) ++(\i*\angle+.5*\angle-\WindowAngleOffset:\inrad) coordinate (mr\i);
  \path (origin) ++(\i*\angle+.5*\angle+\WindowAngleOffset:\inrad) coordinate (ml\i);
  \ifthenelse{\isodd{\i}}
  {
    \pgfmathsetmacro\dist{.75*\rad}
  }{
    \pgfmathsetmacro\dist{.48*\rad}
  };
  \path (origin) ++(\i*\angle+.5*\angle:\dist) coordinate (w\i);
  \path (origin) ++(\i*\angle+.39*\angle:.92*\dist) node {$x_\i$};
  \path (origin) ++(\i*\angle+.33*\angle:\rad) coordinate (b\i);
  \path (origin) ++(\i*\angle+.57*\angle:\rad) node {$W_\i$};
  \path (origin) ++(\i*\angle+.66*\angle:\rad) coordinate (c\i);
  \draw[ultra thick, blue!60!black] (origin) -- (w\i);
  \draw[thick] (b\i) -- (w\i) -- (c\i);
  \ifthenelse{\isodd{\i}}
  {
    \path (mr\i) -- coordinate[midway](m\i) (ml\i);
  }
  {
    \path (w\i) coordinate (m\i) node[dot,red]{};
  }

  \fill[color=blue,opacity=.2] (origin) ++(\i*\angle+\angle:\rad) -- (origin) -- (v\i) -- (sr\i) -- (mr\i) -- (m\i) -- (ml\i) -- (sl\i) -- cycle;
  \draw[thick] (sr\i) -- (mr\i) (ml\i) -- (sl\i);
  \draw[ultra thick, red] (mr\i) -- (m\i) -- (ml\i);
  \filldraw[color=red] (mr\i) circle[radius=.04] (ml\i) circle[radius=.04];
  \ifthenelse{\isodd{\i}}
  {\filldraw[color=blue!60!black] (w\i) circle[radius=0.04];}
  {\filldraw[color=red] (m\i) circle[radius=0.04];}
}

\path (b\spokes) node[above right] {$\theta^w_Y$};
\path (origin) ++(\spokeslessone*\angle+.72*\angle-\angle:.33*\rad) node[color=blue!60!black] {$W_0$};
\draw [decorate,decoration={brace,amplitude=5pt,mirror,raise=4pt},yshift=0pt] (sl\spokeslessone) -- (sr\spokes) node [black,midway,xshift=0.8cm] {\footnotesize $\bTheta^v_Y$};

\end{tikzpicture}
\caption{The proof of Lemma~\ref{lem:crossing spines}: The arcs shown in red are the segments $\delta_i$, which contain $\Pi^v_Y(x)$ for all $x\in \theta^w$ with $f_Y(x)$ in the subgraph $W_i$ of $\theta^w_Y$.}
\label{fig:for_crossing_spines}
\end{figure}

\begin{proof} Let $x_0 = \theta^v_Y \cap \theta^w_Y$ be the unique intersection point of the spines.  Let $W_0$ be the smallest subgraph of $\theta^w_Y$ containing $\theta^w_Y \cap \bTheta^v_Y$ and let $W_1, \dots, W_k$ be the closures of the components of $\theta^w_Y \setminus W_0$, so that $\theta^w_Y = W_0 \cup \dots \cup W_k$. See Figure~\ref{fig:for_crossing_spines}. For $1 \leq i \leq k$, let  $x_i \in W_i$ be the closest (cone) point to $x_0$. Then define $p_i$ to be the intersection of the geodesic $[x_0, x_i]\subset\theta^w_Y$ with $\partial\bTheta^v_Y$ and let $\delta_i\subset \partial \bTheta^v_Y$ be the segment consisting of the ($1$ or $2$) saddle connections along $\partial \bTheta^v_Y$ that contain $p_i$. Define $\delta^v_Y(w) = \delta_1\cup \dots \cup \delta_k$, and note that $\delta^v_X(w) = f_{X,Y}(\delta^v_Y(w))$ and $\theta^w_Y\cap \partial \bTheta^v_Y\subset \delta^v_Y(w)$ by construction.

For any $x\in \bar{\Sigma}$ with $f_Y(x) \in W_i$, where $i =1, \dots, k$, the flat geodesic from $f_Y(x)$ to $x_0$ first intersects $\bTheta^v_Y$ at $p_i$; therefore $p_i\in \Pi^v_Y(x)$ by definition of the window. By Lemma~\ref{lem:window}, it follows that $\Pi^v_Y(x)\subset \delta_i$. The union $\cup_{i=1}^k W_i$ contains every cone point of $\theta^w_Y$ except possibly $x_0$. Thus we have proven $\Pi^v_Y(\theta^w\cap \bar{\Sigma}^{\notin v}) = \Pi^v_Y(\theta^w_Y\cap \bar{\Sigma}^{\notin v}) \subset \delta^v_Y(w)$. 

Let $Z\in D$ be the closest point on $\partial B_v$ to $\partial B_w$, thus $Z\in \partial B_v$ lies on the unique hyperbolic geodesic that intersects $\partial B_v$ and $\partial B_w$ orthogonally. In $E_Z$ the directions $\alpha(v)$ and $\alpha(w)$ are perpendicular. Therefore the cone points of $\partial \bTheta^v_Z$ that are closest to $\theta^v_Z\cap \theta^w_Z$ all lie in $\delta^v_Z(w)$, since they must be endpoints of saddle connections along $\partial \bTheta^v_Z$ that intersect $\theta^w_Z$.
More generally, for any $Y\in c_v(\partial B_w)$, the directions $\alpha(v)$ and $\alpha(w)$ are nearly perpendicular and thus we have
\begin{equation*}
\label{eqn:perpendicular-crossing-spines}
\Pi^v_Y(\theta^w\cap \bar{\Sigma}^{\notin v}) \subset \delta^v_Y(w) \subset B_K(\theta^v_Y\cap \theta^w_Y).
\end{equation*}
for some uniform constant $K > 0$ that depends only on the length of $c_v(\partial B_w)$ and the maximum over $\partial B_v$ of the length/width of any saddle connection/strip in the $\alpha(v)$ direction (Lemma~\ref{L:strip-and-saddle-bound}). 
Now, for any $x\in \theta^w \cap \bar{\Sigma}^v$,  the point $X = c_v(\pi(x))$ lies in $c_v(\partial B_w)$ and we have that $f_X(x) = \theta^v_X\cap \theta^w_X$. The above equation shows there are cone points of $\partial \bTheta^v_X$ within $K$ of $f_X(x)$; hence $\Pi^v_X(x)=\Pi^v_X(f_X(x))$ lies in $B_K(\theta^v_X\cap \theta^w_X)$ by definition. 
Using the fact that the length of $c_v(\partial B_w)$ is uniformly bounded, we see that the map $f_{Y,X}$ is uniformly biLipschitz and therefore that $\Pi^v_Y(x) = f_{Y,X}(\Pi^v_X(x))$ lies within bounded distance of $\theta^v_Y\cap \theta^w_Y = f_{Y,X}(\theta^v_X\cap \theta^w_X)$. Combining this with the above finding that $\Pi^v_Y(\theta^w\cap \bar{\Sigma}^{\notin v})$ lies within bounded distance of $\theta^v_Y\cap \theta^w_Y$, we finally conclude that $\diam(\Pi^v_Y(\theta^w))$ is uniformly bounded.
\end{proof}

\subsection{Projections}
Here we define projections $\Lambda^v \colon \bar{\Sigma} \to \pow(\mathcal K^v)$ and $\xi^v \colon \bar{\Sigma} \to \pow(\bXi^v)$.
 In preparation, we first define $\Pi^v \colon \bar{\Sigma} \to \pow(\partial \bTheta^v)$ by
\[ \Pi^v(x) = \bigcup_{Y \in \partial B_v}\Pi^v_Y(x),\quad\text{for any $v\in \vtx$ and $x\in \bar{\Sigma}$}.\]
In words, $\Pi^v(x)$ consists of the $v$--windows of $x$ in all fibers over $\partial B_v$. As before, we extend to arbitrary subsets $U\subset \bar E$ by setting $\Pi^v(U) = \Pi^v(U\cap \bar{\Sigma})$.

Now, for each $v\in \vtx$ our projections are defined as the compositions
\[ \Lambda^v  = \lambda^v\circ \Pi^v \quad \mbox{ and } \quad \xi^v(x) = i^v\circ \Pi^v.\]
A useful observation is that for any two vertices $v,w \in \vtx$ with $d_\tree(v,w) = 1$ and $X \in \bar D$, we have
\begin{equation} \label{E:an obvious but useful fact}
\xi^v(\theta^w_X) = \xi^v(\theta^w)=w \in \bXi^v.
\end{equation}
Mirroring the notation $d_Y^v(x,y)$ above, for $x,y\in \bar{\Sigma}$ we define
\[d_{\mathcal{K}^v}(x,y) = \diam(\Lambda^v(x)\cup \Lambda^v(y))\quad\text{and}\quad d_{\bXi^v}(x,y) = \diam(\xi^v(x)\cup \xi^v(y)).\]

\begin{lemma} \label{L:bounded is bounded}
There is a function $N'\colon [0,\infty)\to [0,\infty)$ such that for all $v\in \vtx$, $X\in \partial B_v$, and $x,y\in \bar{\Sigma}$ the quantities $d_{\mathcal{K}^v}(x,y)$ and $d_{\bXi^v}(x,y)$ are at most $N'(d_X^v(x,y))$.
\end{lemma}
\begin{proof}
For $Y\in \partial B_v$, let us compare the images of some subset $U \subset \partial \bTheta^v_Y$ and
\[\mathrm{Smear}(U) = \bigcup_{Z\in \partial B_v} f_{Z,Y}(U) \subset \partial \bTheta^v\]
under the maps $\lambda^v,i^v$. 
Since the boundary components of $\bTheta^v$ are preserved by the maps $f_{Z,Y}$, the images $i^v(U)=i^v(\mathrm{Smear}(U))$ are exactly the same. 
Moreover, for each $x\in \partial \bTheta^v_Y$ we have $\lambda^v(\mathrm{Smear}(\{x\})) = \lambda^v(\ell_{x,\alpha(v)})$, and so by Proposition~\ref{prop:KL}(\ref{item:lambda_flowlines_bounded}) this is a set with diameter at most $K_1$.
Therefore $\lambda^v(U)$ and $\lambda^v(\mathrm{Smear}(U))$ have Hausdorff distance at most $2K_1$.

Now, let $x,y\in \bar{\Sigma}$ and $X\in \partial B_v$ be as in the statement. Set $U = \Pi^v_X(x)\cup \Pi^v_X(y)$ so that  $\diam(U) = d^v_X(x,y)$. Since $\Pi^v(x) = \mathrm{Smear}(\Pi^v_X(x))$ and similarly for $\Pi^v(y)$, we see that
\[d_{\mathcal{K}^v}(x,y) = \diam(\lambda^v(\mathrm{Smear}(U))) \quad\text{and}\quad d_{\bXi^v}(x,y) = \diam(i^v(\mathrm{Smear}(U))).\]
Since $\lambda^v$ is coarsely $K_1$--Lipschitz by Proposition~\ref{prop:KL}\ref{item:lambda_lipschitz}, the preceding paragraph and the triangle inequality shows that
\[ d_{\mathcal{K}^v}(x,y) = \diam(\lambda^v(\mathrm{Smear}(U))) \leq \diam(\lambda^v(U)) + 2K_1 \leq (K_1d_X^v(x,y) + K_1) + 2K_1.\]
Similarly, by Corollary~\ref{cor:inclusion_Lipschitz} we have that
\[ \diam_{\bXi^v}(x,y) = \diam(i^v(U))\le N(d_X^v(x,y))+1.\]
Setting $N'(t) = \max\{N(t)+1,K_1t+3K_1\}$ completes the proof.
\end{proof}

 \begin{proposition} \label{prop:R-projections}
 There exists $K_4>0$ so that for any $v \in \vtx$:
 \begin{enumerate}
 \item $\Lambda^v$ and $\xi^v$ are $K_4$--coarsely Lipschitz;
 \item For any $w \in \vtx$, we have
 \begin{enumerate}
 \item $\diam \left(\Lambda^{v}(\bTheta^w)\right) < K_4$, unless $d_{\tree}(v,w) \leq 1$;
 \item $\diam \left(\xi^v(\bTheta^w)\right) < K_4$ unless $w = v$. 
 \end{enumerate}
 \end{enumerate}
 \end{proposition}

 \begin{proof}  For part (1), we first observe that by \cite[Lemma~3.5]{DDLSI}, there exists an $R >0$ so that any pair of points $x,y\in \bar{\Sigma}$ may be connected by a path of length at most $R \bar d(x,y)$ that is a concatenation of at most $R\bar d(x,y)+1$ pieces, each of which is either a saddle connection of length at most $R$ in a vertical vertical fiber, or a horizontal geodesic segment in $\bar E$.
By the triangle inequality, it thus suffices to assume that $x$ and $y$ are the endpoints of either a horizontal geodesic or a vertical saddle connection of length at most $R$. Appealing to Lemma~\ref{L:bounded is bounded}, it further suffices to show that $d_X^v(x,y)$ is linearly bounded by $\bar{d}(x,y)$ for some $X\in \partial B_v$. Lemmas \ref{lem:different spine} and \ref{lem:combo Lipschitz}(1) handle the horizontal segment case, since we are free to subdivide such a path into $\lceil\bar{d}(x,y)\rceil$ segments of length at most $1$.
 
For the vertical segment case we assume $x$ and $y$ lie in the same fiber $E_Y$ and differ by a saddle connection $\delta$ of length at most $R$. Let $\theta^w_Y$, where $w\in \vtx$, be the spine containing $\delta$. The fact that $\delta$ is bounded means that $Y = \pi(x)=\pi(y)$ is bounded distance from the horocycle $\partial B_w$. Let $X = c_w(Y)\in \partial B_w$ and let $\delta' = f_{X,Y}(\delta)$ be the saddle connection in $\theta^w_X$ connecting $x'=f_{X,Y}(x)$ and $y' = f_{X,Y}(y)$. By the triangle inequality and the first part above about bounded length horizontal segments, it suffices to work with the points $x',y'\in \theta^w$.  There are three cases to consider: Firstly, if $v = w$, then $x',y'\in \theta^v_X$ so that $\Pi^v_X(x')$ and $\Pi^v_X(y')$ choose the closest cone points in $\partial \bTheta^v_X$ to $x'$ and $y'$, respectively. Since $x'$ and $y'$ are close, so are $\Pi^v_X(x')$ and $\Pi^v_X(y')$. Secondly, if $d_\tree(v,w) > 1$, then Corollaries \ref{cor:bridge same} and \ref{cor:bridge diff}, and Lemma \ref{lem:crossing spines} give a uniform bound on $d^v_Z(x',y') \le \diam(\Pi^v_Z(\theta^w))$ for any point $Z\in c_v(\partial B_w)$. Finally, if $d_\tree(v,w) = 1$ then $\theta^w_X$ and $\theta^v_X$ are adjacent non-crossing spines in $E_X$. Since $\theta^w_X$ is totally geodesic, if follows that $\Pi^v_X(x')$ and $\Pi^v_X(y')$ are either equal or connected by a single edge of $\partial \bTheta^v_X$. But this saddle connection has uniformly bounded length,  since $X\in \partial B_v$, which completes the proof of (1).

For (2), first recall that strips/saddle connections in the $\alpha(w)$ direction have uniformly bounded width/length over $\partial B_w$ (Lemma~\ref{L:strip-and-saddle-bound}). Therefore $\bTheta^w\cap \bar{\Sigma}$ is contained in a bounded neighborhood of $\theta^w\cap \bar{\Sigma}$. By part (1) it thus suffices to bound $\diam(\Lambda^v(\theta^w))$ and $\diam(\xi^v(\theta^w))$. When $d_{\tree}(v,w) \ge 2$, Corollaries \ref{cor:bridge same} and \ref{cor:bridge diff}, and Lemma \ref{lem:crossing spines}, imply that there exists $X \in \partial B_v$ so that $\Pi^v_X(\theta^w)$ has bounded diameter in $\bTheta^v_X$. Appealing to Lemma~\ref{L:bounded is bounded} now bounds $\diam(\Lambda^v(\theta^w))$ and $\diam(\xi^v(\theta^w))$ in these cases. For the remaining case $d_\tree(v,w)=1$ of 2(b), we note that $\xi^v(\theta^w)$ is a single point by \eqref{E:an obvious but useful fact}, and thus 2(b) follows.
\end{proof}

\section{Hierarchical hyperbolicity of $\Gamma$} \label{S:combinatorial HHS}


\newcommand{\vertexedges}{
\draw[\edgetype] (\xcoordmain,\ycoordmain) -- (\xcoordmain+\vvc,\ycoordmain+\vvb) ;
\draw[\edgetype] (\xcoordmain,\ycoordmain) -- (\xcoordmain+\vva,\ycoordmain);
\draw[\edgetype] (\xcoordmain,\ycoordmain) -- (\xcoordmain + \vvd,\ycoordmain+\vve);
}

\newcommand{\vblownup}[9]{
\def\thecolor{#1}
\def\xcoordmain{#2}
\def\ycoordmain{#3}
\def\vaname{#4}
\def\vbname{#5}
\def\vcname{#6}
\def\vdname{#7}
\def\addedges{#8}
\def\edgetype{#9}
\addedges
\node[draw,fill=\thecolor,circle] at (\xcoordmain,\ycoordmain) {\tiny \vaname};
\node[draw,fill=\thecolor,circle] at   (\xcoordmain+\vvc,\ycoordmain+\vvb) {\tiny \vbname};
\node[draw,fill=\thecolor,circle] at (\xcoordmain+\vva,\ycoordmain) {\tiny \vcname};
\node[draw,fill=\thecolor,circle] at (\xcoordmain + \vvd,\ycoordmain+\vve) {\tiny \vdname};
}

\newcommand{\cone}[5]{
\def\edgetype{#1}
\def\xfrom{#2}
\def\yfrom{#3}
\def\xto{#4}
\def\yto{#5}

\draw[\edgetype] (\xfrom,\yfrom) -- (\xto,\yto);
\draw[\edgetype] (\xfrom,\yfrom) -- (\xto+\vvc,\yto+\vvb);
\draw[\edgetype] (\xfrom,\yfrom) -- (\xto+\vva,\yto);
\draw[\edgetype] (\xfrom,\yfrom) -- (\xto+\vvd,\yto+\vve);
}

\newcommand{\join}[5]{
\def\typeedge{#1}
\def\xxfrom{#2}
\def\yyfrom{#3}
\def\xxto{#4}
\def\yyto{#5}

\cone{\typeedge}{\xxfrom}{\yyfrom}{\xxto}{\yyto}
\cone{\typeedge}{\xxfrom+\vvc}{\yyfrom+\vvb}{\xxto}{\yyto}
\cone{\typeedge}{\xxfrom+\vva}{\yyfrom}{\xxto}{\yyto}
\cone{\typeedge}{\xxfrom\vvd}{\yyfrom+\vve}{\xxto}{\yyto}
}

In this section we complete the proof that $\Gamma$ is hierarchically hyperbolic. We will use a criterion from \cite{comb_HHS}, which we now briefly discuss.
For further information and heuristic discussion of this approach to hierarchical hyperbolicity, we refer the reader to \cite[\S1.5, ``User's guide and a simple example'']{comb_HHS}.

Consider a simplicial complex $\CX$ and a graph $\CW$ whose vertex set is the set of maximal simplices of $\CX$.  
The pair $(\CX,\CW)$ is called a \emph{combinatorial HHS} if it satisfies the requirements listed in Definition~\ref{defn:combinatorial_HHS} below, and \cite[Theorem 1.18]{comb_HHS} guarantees that in this case $\CW$ is an HHS. The main requirement is along the lines of: $\CX$ is hyperbolic, and links of simplices of $\CX$ are also hyperbolic. However, this is rarely the case because co-dimension--1 faces of maximal simplices have discrete links. To rectify this, additional edges (coming from $\CW$) should be added to $\CX$ and its links as detailed in Definition~\ref{defn:X_graph}. 
In our case, after adding these edges, $\CX$ will be quasi-isometric to $\hat E$, and each other link will be quasi-isometric to either a point or to one of the spaces $\mathcal K^v$ or $\bXi^v$  introduced in \S\ref{S:projections}.

There are two natural situations where such pairs arise that the reader might want to keep in mind. First, consider a group $H$ acting on a simplicial complex $\CX$ so that there is one orbit of maximal simplices, and those have trivial stabilizers. In this case, we take $\CW$ to be (a graph isomorphic to) a Cayley graph of $H$. (More generally, if the action is cocompact with finite stabilizers of maximal simplices, then the appropriate $\CW$ is quasi-isometric to a Cayley graph.) For the second situation, $\CX$ is the curve graph of a surface; then maximal simplices are pants decompositions of the surface and $\CW$ can be taken to be the pants graph. We will use this as a working example below, when we get into the details.

Most of the work carried out in \S\ref{S:projections} will be used (as a black-box) to prove that, roughly, links are quasi-isometrically embedded in a space obtained by removing all the ``obvious'' vertices that provide shortcuts between vertices of the link. This can be seen as an analogue of Bowditch's fineness condition in the context of relative hyperbolicity.

This section is organized as follows. In \S\ref{subsec:setup} we list all the relevant definitions and results from \cite{comb_HHS}, and we illustrate them using pants graphs. In \S\ref{subsec:our_comb_HHS} we construct the relevant combinatorial HHS for our purposes. In \S\ref{subsec:links_list} we analyze all the various links and related combinatorial objects; 
we note that most of the work done in \S\ref{S:projections} is used here to prove Lemma \ref{lem:retractions}. 
At that point, essentially only one property of combinatorial HHSs will be left to be checked, and we do so in \S\ref{subsec:final_proof}.

\subsection{Basic definitions}\label{subsec:setup}

We start by recalling some basic combinatorial definitions and constructions. Let $\CX$ be a flag simplicial 
complex.   

\begin{defn}[Join, link, star]\label{defn:join_link_star}
Given disjoint simplices $\Delta,\Delta'$ of $\CX$, the \emph{join} is denoted $\Delta\star\Delta'$ and is the simplex spanned by 
$\Delta^{(0)}\cup\Delta'^{(0)}$, if it exists.  More generally, if $K,L$ are disjoint induced subcomplexes of 
$\CX$ such that every vertex of $K$ is adjacent to every vertex of $L$, then the join
$K\star L$ is the induced subcomplex with vertex set $K^{(0)}\cup L^{(0)}$.  

For each simplex $\Delta$, the \emph{link} $\link(\Delta)$ is the union of 
all simplices $\Delta'$ of $\CX$ such that $\Delta'\cap\Delta=\emptyset$ and $\Delta'\star\Delta$ is a simplex of $\CX$.  
The \emph{star} of $\Delta$ is $Star(\Delta)=\link(\Delta)\star\Delta$, i.e. the union of all simplices of $\CX$ that contain 
$\Delta$.

We emphasize that $\emptyset$ is a simplex of $\CX$, whose link is all of $\CX$ and whose star is all of $\CX$.  
\end{defn}

\begin{defn}[$X$--graph, $W$--augmented dual complex]\label{defn:X_graph}
An \emph{$\CX$--graph} is any graph $\CW$ 
whose vertex set is the set of maximal simplices of $\CX$ (those not contained in any larger simplex).

For a flag complex $\CX$ and an $\CX$--graph $\CW$, the \emph{$\CW$--augmented dual graph} 
$\duaug{\CX}{\CW}$ is the graph defined as follows:
\begin{itemize}
     \item the $0$--skeleton of $\duaug{\CX}{\CW}$ is $\CX^{(0)}$;
     \item if $v,w\in \CX^{(0)}$ are adjacent in $\CX$, then they are adjacent in $\duaug{\CX}{\CW}$; 
	 \item if two vertices in $\CW$ are adjacent, then we consider 
	 $\sigma,\rho$, the associated maximal simplices of $\CX$, and 
	 in $\duaug{\CX}{\CW}$ we connect each vertex of $\sigma$ to each vertex 
	 of $\rho$.
\end{itemize}
We equip $\CW$ with the usual path-metric, in which each edge has unit length, and do the same for $\duaug{\CX}{\CW}$.  Observe that the $1$--skeleton of $X$ is a subgraph $\CX^{(1)} \subset \duaug{\CX}{\CW}$.
\end{defn}

We provide a running example to illustrate the various definitions in a familiar situation.  This example will not be used in the sequel.
\begin{example} \label{Ex:curve/pants graph 1} If $\CX$ is the curve complex of the surface $S$, then an example of the an $\CX$--graph, $\CW$, is the pants graph, since a maximal simplex is precisely a pants decomposition.  The $\CW$--augmented dual graph can be thought of as adding to the curve graph, $\CX^{(0)}$ an edge between any two curves that fill a one-holed torus or four-holed sphere and intersect once or twice, respectively: indeed, these subsurfaces are precisely those where an {\em elementary move} happens as in the definition of adjacency in the pants graph.
\end{example}

\begin{defn}[Equivalent simplices, saturation]\label{defn:simplex_equivalence}
For $\Delta,\Delta'$ simplices of $\CX$, we write $\Delta\sim\Delta'$ to mean
$\link(\Delta)=\link(\Delta')$. We denote by $[\Delta]$ the equivalence class of $\Delta$.  Let $\sat(\Delta)$ denote 
the set of vertices $v\in \CX$ for which there exists a simplex $\Delta'$ of $\CX$ such that $v\in\Delta'$ and 
$\Delta'\sim\Delta$, i.e. $$\sat(\Delta)=\left(\bigcup_{\Delta'\in[\Delta]}\Delta'\right)^{(0)}.$$
We denote by $\mathfrak S$ the set of $\sim$--classes of non-maximal simplices in $\CX$.
\end{defn}

\begin{defn}[Complement, link subgraph]\label{defn:complement}
Let $\CW$ be an $\CX$--graph.  For each simplex $\Delta$ of $\CX$, let
$Y_\Delta$ be the subgraph of $\duaug{\CX}{\CW}$ induced by the set
of vertices $(\duaug{\CX}{\CW})^{(0)}-\sat(\Delta)$.

Let $\mathcal C(\Delta)$ be the full subgraph of $Y_\Delta$ spanned by $\link(\Delta)^{(0)}$.  Note that $\mathcal 
C(\Delta)=\mathcal C(\Delta')$ whenever $\Delta\sim\Delta'$.  (We emphasize 
that we are taking links in $\CX$, not in $\CX^{+\CW}$, and then considering the 
subgraphs of $Y_\Delta$ induced by those links.)
\end{defn}

We now pause and continue with the illustrative example.
\begin{example} \label{Ex:curve/pants graph 2} Let $\CX$ and $\CW$ be as in Example~\ref{Ex:curve/pants graph 1}.  A simplex $\Delta$ is a multicurve which determines two (open) subsurfaces $U = U(\Delta),U' = U'(\Delta) \subset S$, where $U$ is the union of the complementary components of the multicurve that are not a pair of pants, and $U' = S - \overline{U}$.  Note that $\partial U \subset \Delta$ is a submulticurve and that $\Delta - \partial U$ is a pants decomposition of $U'$.  A simplex $\Delta'$ is equivalent to $\Delta$ if it defines the same subsurfaces.  Thus $\sat(\Delta)$ consists of $\partial U(\Delta)$ together with all essential curves in $U'(\Delta)$, while $\mathcal C(\Delta)$ is the join of graphs quasi-isometric to curve graphs of the components of $U(\Delta)$.  For components of $U(\Delta)$ which are one-holed tori or four-holed spheres, the  corresponding subgraphs are isometric to their curve graphs (since the extra edges in $\duaug{\CX}{\CW}$ precisely give edges for these curve graphs).
\end{example}

\begin{defn}[Nesting]\label{defn:nest}
 Let $\CX$ be a simplicial complex.  Let $\Delta,\Delta'$ be non-maximal simplices of $\CX$.  Then we write $[\Delta]\nest[\Delta']$ if $\link(\Delta)\subseteq\link(\Delta')$.
\end{defn}
We note that if $\Delta' \subset \Delta$, then $[\Delta] \nest [\Delta']$.  Also, for Example~\ref{Ex:curve/pants graph 1},\ref{Ex:curve/pants graph 2}, $[\Delta] \nest [\Delta']$ if and only if $U(\Delta) \subset U(\Delta')$.

Finally, we are ready for the main definition:

\begin{defn}[Combinatorial HHS]\label{defn:combinatorial_HHS}
 A \emph{combinatorial HHS} $(\CX, \CW)$ consists of a flag simplicial 
 complex $\CX$ and an $\CX$--graph $\CW$ satisfying the following conditions for some $n \in \mathbb N$ and $\delta \geq 1$:
 \begin{enumerate}
  \item \label{item:chhs_flag}any chain $[\Delta_1]\propnest[\Delta_2]\propnest\dots$ has length at most $n$;
    \item \label{item:chhs_delta} for each non-maximal simplex $\Delta$, the subgraph 
$\mathcal C(\Delta)$ is 
$\delta$--hyperbolic and $(\delta,\delta)$--quasi-isometrically embedded in 
$Y_\Delta$;
\item \label{item:chhs_join}Whenever $\Delta$ and $\Delta'$ are non-maximal simplices for which there exists 
a non-maximal simplex $\Gamma$ such that $[\Gamma]\nest[\Delta]$, $[\Gamma]\nest[\Delta']$, and $\diam(\mathcal C 
(\Gamma))\geq \delta$, then there exists a 
simplex $\Pi$ in the link of $\Delta'$ such that $[\Delta'\star\Pi]\nest [\Delta]$ and all $[\Gamma]$ as above satisfy $[\Gamma]\nest[\Delta'\star\Pi]$;
\item \label{item:C_0=C} if $v,w$ are distinct non-adjacent vertices of $\link(\Delta)$, for some simplex $\Delta$ of $\CX$, contained in 
$\CW$--adjacent maximal simplices, then they are contained in $\CW$--adjacent simplices of the form $\Delta\star\Delta'$.
\end{enumerate}
\end{defn}

We will see below that combinatorial HHSs give HHSs. The reader not interested in the explicit description of the HHS structure can skip the following two definitions.

\begin{defn}[Orthogonality, transversality]\label{defn:orth}
Let $\CX$ be a simplicial complex.  Let $\Delta,\Delta'$ be non-maximal simplices of $\CX$.  Then we write $[\Delta]\orth[\Delta']$ if $\link(\Delta')\subseteq \link(\link(\Delta))$.
If $[\Delta]$ and $[\Delta']$ are not $\orth$--related or $\nest$--related, we write $[\Delta]\transverse[\Delta']$.
\end{defn}

\begin{defn}[Projections]\label{defn:projections}
Let $(\mathcal X, \mathcal W)$ be a combinatorial HHS.  

Fix $[\Delta]\in\mathfrak S$ and define a map $\pi_{[\Delta]}\colon \CW\to \pow({\mathcal C([\Delta])})$ as 
follows.  First let $p\colon Y_\Delta\to \pow({\mathcal C([\Delta])})$ be the coarse closest-point projection, i.e. 
$$p(x)= \big\{y\in\mathcal C([\Delta]) \mid d_{Y_\Delta}(x,y)\leq d_{Y_\Delta}(x,\mathcal C([\Delta]))+1\big\}.$$

Suppose that $w\in \CW^{(0)}$, so $w$ corresponds to a unique simplex $\Delta_w$ of $\CX$.
 Define $$\pi_{[\Delta]}(w)=p(\Delta_w\cap Y_\Delta).$$

We have thus defined  $\pi_{[\Delta]}\colon \CW^{(0)}\to \pow({\mathcal C([\Delta])})$.   If $v,w\in \CW^{(0)}$ are joined by an edge $e$ of $\CW$, 
then $\Delta_v,\Delta_w$ are joined by edges in $\duaug{\CX}{\CW}$, and we let 
$\pi_{[\Delta]}(e)=\pi_{[\Delta]}(v)\cup\pi_{[\Delta]}(w)$.

Now let $[\Delta],[\Delta']\in\mathfrak S$ satisfy $[\Delta]\transverse[\Delta']$ or $[\Delta']\propnest [\Delta]$. 
Let $$\rho^{[\Delta']}_{[\Delta]}=p( 
\sat(\Delta')\cap Y_\Delta).$$

Let $[\Delta]\propnest 
[\Delta']$. Let $\rho^{[\Delta']}_{[\Delta]}\colon \mathcal C([\Delta'])\to 
\mathcal C([\Delta])$ be the restriction of $p$ to $\mathcal C([\Delta'])\cap 
Y_\Delta$, and 
$\emptyset$ otherwise.
\end{defn}

The next theorem from \cite{comb_HHS} provides the criteria we will use to prove that $\Gamma$ is a hierarchically hyperbolic group.

Given a combinatorial HHS $(\CX, \CW)$, we denote $\mathfrak S$ the set as in Definition~\ref{defn:simplex_equivalence}, endowed with nesting and orthogonality 
relations as in Definitions~\ref{defn:nest} and~\ref{defn:orth}. Also, we associated to $\mathfrak S$ the hyperbolic spaces as in 
Definition~\ref{defn:combinatorial_HHS}, and define projections as in Definition~\ref{defn:projections}. 

\begin{theorem}{\cite[Theorem 1.18, Remark 1.19]{comb_HHS}}\label{thm:hhs_links}
Let $(\CX, \CW)$ be a combinatorial HHS. Then $(\CW,\mathfrak S)$ is a hierarchically hyperbolic space.

Moreover, if a group $G$ acts by simplicial automorphisms on $\CX$ with finitely many orbits of links of simplices, and the resulting $G$-action on maximal simplices extends to a metrically proper cobounded action on $\CW$, then $G$ acts metrically properly and coboundedly by HHS automorphisms on $(\CW,\mathfrak S)$, and is therefore a hierarchically hyperbolic group.
\end{theorem}

\subsection{Combinatorial HHS structure}\label{subsec:our_comb_HHS}
 We now define a flag simplicial complex $\CX$. The vertex set is $\CX^{(0)}= \vtx \sqcup \kvtx$, where
  \[ \kvtx = \bigsqcup_{v \in \vtx} \mathcal K^v. \]
Given a vertex $s \in \kvtx$, let $v(s) \in \vtx$ be the unique vertex with $s \in \CK^v$.  We also write $\alpha(s)= \alpha(v(s))$.

There are 3 types of edges (see Figure~\ref{fig:blowup}):
 \begin{enumerate}
 \item $v,w \in \vtx$ are connected by an edge if and only if $d_{\tree}(v,w) = 1$.
 \item $s,t \in \kvtx$ are connected by an edge if and only if $d_{\tree}(v(s),v(t))=1$.
\item  $s \in \kvtx$ and $w \in \vtx$ are connected by an edge if and only if $d_{\tree}(v(s),w) \leq 1$.
 \end{enumerate}
We declare $\CX$ to be the flag simplicial complex with the 1-skeleton defined above.   

\begin{figure}[h]
\begin{tikzpicture}[scale=.7,inner sep=0pt, dot/.style={fill=black,circle,minimum size=3pt}]

\coordinate (origin) at (0,0);

\pgfmathsetmacro\vva{1.5};
\pgfmathsetmacro\vvc{1};
\pgfmathsetmacro\vve{-.55};
\pgfmathsetmacro\vvb{.55};
\pgfmathsetmacro\vvd{1};

\join{thin}{0}{0}{0}{3}
\join{thin}{0}{0}{-2}{-2}
\join{thin}{0}{0}{2}{-2}

\vblownup{lime}{0}{3}{$v_1$}{$t_1$}{$t_2$}{$t_3$}{\vertexedges}{ultra thick}
\vblownup{yellow}{0}{0}{$v_2$}{$s_1$}{$s_2$}{$s_3$}{\vertexedges}{ultra thick}
\vblownup{pink}{-2}{-2}{$v_3$}{\phantom{$t_1$}}{\phantom{$t_2$}}{\phantom{$t_3$}}{\vertexedges}{ultra thick}
\vblownup{cyan}{2}{-2}{$v_4$}{\phantom{$t_1$}}{\phantom{$t_2$}}{\phantom{$t_3$}}{\vertexedges}{ultra thick}

\draw[thick,->>] (3,1.5) .. controls (5,2)  .. (7,1.5);

\draw[ultra thick] (9,0) -- (9,3);
\draw[ultra thick] (9,0) -- (11,-2);
\draw[ultra thick] (9,0) -- (7,-2);

\node[draw,fill=lime,circle] at (9,3) {\small $v_1$};
\node[draw,fill=yellow,circle] at (9,0) {\small $v_2$};
\node[draw,fill=pink,circle] at (7,-2) {\small $v_3$};
\node[draw,fill=cyan,circle] at (11,-2) {\small $v_4$};


\node at (-1.25,0) {$v(s_i) = $};
\node at (-1.25,3) {$v(t_i)=$};

\node at (5,2.4) {$Z$};



\end{tikzpicture}
\caption{The simplicial map $Z$ restricted to a part of $\CX$ (on the left) to a part of the union of trees $T_\alpha$ (on the right). 
Vertices in in $\CX$ are colored the same as their image vertices in $T_\alpha$.
}\label{fig:blowup}

\end{figure}

The map $\kvtx \to \vtx$ given by $s \mapsto v(s)$ and the identity $\vtx \to \vtx$ extends to a surjective simplicial map
 \[ Z \colon \CX \to \bigsqcup_{\alpha \in \CP} T_\alpha. \]
We note that we may view the union $\bigsqcup T_\alpha$ on the right as a subgraph of $\CX^{(1)}$, making $Z$ a retraction.

For any vertex $v$ in any tree $T_\alpha$, $Z^{-1}(v)$ is the join of $\{v\}$ and the set $\mathcal K^v$:
\begin{equation} \label{E: Z preimage of vertex} Z^{-1}(v) = \{v\} \star \mathcal K^v. \end{equation}
For any pair of adjacent vertices $v,w \in T_\alpha$ (so $d_{\tree}(v,w) = 1$), the preimage of the edge $[v,w] \subset T_\alpha$ is also a join:
\begin{equation} \label{E: Z preimage of edge} Z^{-1}([v,w]) = Z^{-1}(v) \star Z^{-1}(w) = (\{v\} \star \mathcal K^v) \star (\{w\} \star \mathcal K^w).\end{equation}

 \begin{lemma}\label{lem:max_simplex}
  The maximal simplices of $\CX$ are exactly the $3$--simplices with vertex set $\{s,v(s),t,v(t)\}$ where $s,t \in \kvtx$ and $d_{\tree}(v(s),v(t))=1$. In this case, we say that $(s,t)$ defines a maximal simplex, denoted $\sigma(s,t)$.
 \end{lemma}

 \begin{proof}
 Because the map $Z$ is simplicial, any simplex of $\CX$ is contained in $Z^{-1}(v)$ or $Z^{-1}([v,w])$ for some vertex $v$ in some $T_\alpha$ or some edge $[v,w]$ in some $T_\alpha$.  The lemma thus follows from \eqref{E: Z preimage of vertex} and \eqref{E: Z preimage of edge}.
 \end{proof}
 
Given a vertex $s \in \kvtx$, recall from Lemma~\ref{L:KL-like} that $M(s)=(\lambda^{v(s)})^{-1}(N_{K_1}(s)) \subset \bTheta^{v(s)}$, for $K_1$ as in Proposition \ref{prop:KL} (and Lemma~\ref{L:KL-like}).  Given a pair of vertices $(s,t)$ in $\kvtx$ that define a maximal simplex $\sigma(s,t)$, we will write
$M(s,t) = M(s) \cap M(t)$.
\begin{lemma}\label{lem:dense}
There exists $R>0$ with the following properties.
\begin{enumerate}
 \item \label{item:diam intersect} For any pair of adjacent vertices $s,t\in \kvtx$ (i.e., defining a maximal simplex $\sigma(s,t)$), $M(s,t)$ is a non-empty subset of diameter at most $R$.
 \item \label{item:dense bTheta} Given $v\in \vtx$, we have
 $$\bTheta^v=\bigcup_{v(s)=v,d_\tree(v(s),v(t))=1} M(s,t).$$
 \item \label{item:dense in KL} Fixing $s$, we have
 $$M(s)=\bigcup_{d_\tree(v(s),v(t))=1}M(s,t).$$
  \item \label{item:dense bar E} The collection of all $M(s,t)$ is $R$--dense in $\bar{E}$.
\end{enumerate}
\end{lemma}

\begin{proof}
 Item \eqref{item:diam intersect} follows from Proposition \ref{prop:KL}\eqref{item:KL_qi}. More precisely, the fact that $M(s,t)$ is non-empty follows from $K_1$-coarse-surjectivity of $\lambda^{v(s)}\times\lambda^{v(t)}$, while boundedness follows from the fact that said map is a quasi-isometry.
 
 In order to show item \eqref{item:dense bTheta}, notice that
 \[ \bTheta^v=\bigcup_{d_{\tree}(v,w)=1} \bTheta^v\cap\bTheta^w.\]
That is, every point of $\bTheta^v$ is also in $\bTheta^w$ for some $w$ adjacent to $v$. In view of this, we conclude by noticing that if $x\in \bTheta^v\cap\bTheta^w$, then $x\in M(\lambda^v(x), \lambda^w(x))$. Item \eqref{item:dense in KL} follows similarly.
 
 Finally, item \eqref{item:dense bar E} follows from item \eqref{item:dense bTheta} and the fact that the collection of all $\bTheta^v$ is coarsely dense in $\bar E$.
\end{proof}

Next we define a graph $\CW$ whose vertex set is the set of maximal simplices of $\CX$. We would like to just connect maximal simplices when the corresponding subsets $M(s,t)$ are close in $\bar E$ (first bullet below); however, in order to arrange item \eqref{item:C_0=C} of the definition of combinatorial HHS (and only for that reason) we need different closeness constants for different situations.  We fix $R$ as in Lemma \ref{lem:dense}, and moreover we require $R>K_1^2+K_1$, for $K_1$ as in Proposition~\ref{prop:KL} and Lemma~\ref{L:KL-like}.

Given maximal simplices $\sigma(s_1,t_1)$ and $\sigma(s_2,t_2)$, we declare them to be connected by an edge in $\CW$ if one of the following holds:
\begin{equation}\label{Eq:Def W-edges}
\begin{array}{l} \bullet \, \,  \bar d(M(s_1,t_1),M(s_2,t_2))\leq 10 R \hspace{5cm} \\
 \bullet \,\, s_1=s_2 \mbox{ and } \bar d(M(t_1),M(t_2))\leq 10 R \end{array}
 \end{equation}
  Here the the $\bar d$--distances are the {\em infimal} distances between the sets in $\bar E$ (as opposed to the diameter of the union).
  Note that since $M(s,t)=M(t,s)$, the second case also implicitly describes a ``symmetric case'' with $s_i$ and $t_i$ interchanged.

The following is immediate from Lemma \ref{lem:retract_on_the_middle}, setting $R'=\max\{10R,N_{\mathcal K}(10R)\}$.

\begin{lemma}\label{lem:connect_not_far}
 There exists $R'\geq 10R$ so then the following holds. If $s,t_1,t_2\in \kvtx$ are vertices with $s$ connected to both $t_i$ in $\CX$ and $\bar d(M(t_1),M(t_2))\leq 10 R$, then $\bar d(M(s,t_1),M(s,t_2))\leq R'$. In particular, whenever $\sigma(s_1,t_1)$ and $\sigma(s_2,t_2)$ are connected in $\CW$, we have $\bar d(M(s_1,t_1),M(s_2,t_2))\leq R'$.
\end{lemma}

\begin{lemma}\label{lem:W=barE}
 $\CW$ is quasi-isometric to $\bar E$, by mapping each $\sigma(s,t)$ to (any point in) $M(s,t)$. Moreover, the extension group $\Gamma$ acts by simplicial automorphisms on $\mathcal X$, induced by the existing action on $\mathcal V\subseteq \mathcal X^{(0)}$ and the action on $\mathcal K\subseteq \mathcal X^{(0)}$ as in Proposition \ref{prop:KL}\eqref{item:equivariance}. The resulting action on maximal simplices extends to a metrically proper cobounded action on $\CW$.
\end{lemma}

\begin{proof}
 In view of Lemma \ref{lem:connect_not_far}, the first part follows by combining Lemma \ref{lem:dense}(\ref{item:dense bar E}) and Proposition \ref{P:graph approximation} (applied to any choice of a point in each $M(s,t)$). 
 
 It is immediate to check that the $\Gamma$-action defined on the $0$-skeleton of $\CX$ extends to an action on $\CX$.  That the resulting action on maximal simplices of $\CX$ (that is, the $0$-skeleton of $\CW$) extends to an action on $\CW$ follows from the equivariance property in Proposition~\ref{prop:KL}(\ref{item:equivariance}) and the definitions of the sets $M(s)$ and $M(s,t)$.
 
 Moreover, the quasi-isometry $\CW\to\bar E$ described in the statement is $\Gamma$-equivariant, so that the action of $\Gamma$ on $\CW$ is metrically proper and cobounded since the action of $\Gamma$ on $\bar E$ has these properties.
 \end{proof}

The goal for the remainder of this section is to prove the following.

\begin{theorem}\label{thm:E_is_chhs}
The pair $(\CX,\CW)$ is a combinatorial HHS. Moreover, there is an action of $\Gamma$ on $\CX$ satisfying the properties stated in Theorem \ref{thm:hhs_links}.
\end{theorem}

\subsection{Simplices, links, and saturations}
\label{subsec:links_list}
Before giving the proof of Theorem~\ref{thm:E_is_chhs}, we begin by describing explicitly the kinds of simplices of $\CX$ that there are, explain what their links and saturations are, and observe some useful properties.

\begin{lemma} [Empty simplex] \label{L: empty simplex case} 
For the empty simplex, $\mathcal C(\emptyset)=\CX^{+\CW}$ is quasi-isometric to $\hat{E}$.
\end{lemma}
\begin{proof} We define a map $Z' \colon \CX^{+\CW} \to \hat E$ that extends the (restricted) simplicial map $Z \colon \CX^{(1)} \to \bigsqcup T_\alpha$ already constructed above.   To do that, we must extend over each edge $e=[x,y]$ of $\CX^{+\CW}$ coming from the edge of $\CW$ connecting $\sigma(s_1,t_1)$ and $\sigma(s_2,t_2)$.  Since $Z(x),Z(y) \in \vtx$, and $\bar d(M(s_1,t_1),M(s_2,t_2))\leq R'$ (for $R'$ as in Lemma \ref{lem:connect_not_far}), we see that $d_{\hat E}(v,w)\leq R'$.  We can then define $Z'$ on $e$ to be a constant speed parameterization of a uniformly bounded length path from $Z(x)$ to $Z(y)$.  It follows that $Z'$ is Lipschitz.

The union of the trees $\bigsqcup T_\alpha$ is $R_0$--dense for some $R_0 >0$ by \cite[Lemma~3.6]{DDLSI}, so it suffices to find a one-sided inverse to $Z'$, from $\bigsqcup T_\alpha$ to $\CX^{+\CW}$, and show that with respect to the subspace metric from $\hat E$, it is coarsely Lipschitz.  As already noted, $Z$ restricts to a retraction of $\CX^{(1)}$ onto $\bigsqcup T_\alpha \subset \CX^{(1)} \subset \CX^{+\CW}$, which is thus the required one-sided inverse.  All that remains is to show that it is coarsely Lipschitz.

According to \cite[Lemma~3.8]{DDLSI}, any $v\in T_\alpha,w\in T_\beta$ are connected by a {\em combinatorial path} of length comparable to $\hat{d}(v,w)$.  Such a path is the concatenation of {\em horizontal jumps}, each of which is the $\bar P$--image in $\hat E$ of a geodesic in $\bar D_z$, for some $z \in \bar \Sigma$, that connects two components of $\partial \bar D_z$ and whose interior is disjoint from $\partial \bar D_z$.  From that same lemma, we may assume each horizontal jump has length uniformly bounded above and below, and thus has total number of jumps bounded in terms of $\hat{d}(v,w)$.
Therefore, we can reduce to the case that $v,w$ are joined by a single horizontal jump of bounded length. Such a horizontal jump can also be regarded as a path in $\bar E$ connecting $\bTheta^v$ to $\bTheta^w$. Hence, in view of Lemma \ref{lem:dense}\eqref{item:dense bTheta}, there are $M(s_1,t_1)\subseteq \bTheta^v$ and $M(s_2,t_2)\subseteq \bTheta^w$ within uniformly bounded distance of each other in $\bar E$. Lemma \ref{lem:W=barE} implies that there exists a path in $\CW$ of uniformly bounded length from $\sigma(s_1,t_1)$ to $\sigma(s_2,t_2)$, which can be easily turned into a path of uniformly bounded length from $v$ to $w$ in $\CX^{+\CW}$, as required.
\end{proof}

There is an important type of $1$--dimensional simplex, which we call a {\em $\bXi$--type simplex}, due to the following lemma.  See Figure~\ref{F:link-sat-Xi}.  Given $w \in \vtx$, set
\begin{equation} \label{E:Xi link} \link_{\bXi}(w) = \bigcup_{d_{\tree}(u,w) = 1} Z^{-1}(u) = \bigcup_{d_{\tree}(u,w) = 1} \{u\} \star \CK^u.
\end{equation}

\begin{lemma} [$\bXi$--type simplex]  \label{L: Xi simplex}
Let $\Delta$ be a $1$--simplex of $\CX$ with vertices $s,v(s)$, for $s \in \kvtx$. Then
\[ \link(\Delta) = \link_{\bXi}(v(s)) \quad \mbox{ and } \quad \Sat(\Delta)=\{v(s)\}\cup \CK^{v(s)}.\]
Moreover, $\mathcal C(\Delta)$ is quasi-isometric to $\bXi^{v(s)}$, via a quasi-isometry which is the identity on $\vtx \cap \mathcal C(\Delta)$ and maps $t$ to $v(t)$ for $t \in \kvtx \cap \C(\Delta)$.
\end{lemma}

\begin{figure}[h]
\begin{tikzpicture}[scale=.7,inner sep=0pt, dot/.style={fill=black,circle,minimum size=3pt}]

\coordinate (origin) at (0,0);


\pgfmathsetmacro\vva{1.5};
\pgfmathsetmacro\vvc{1};
\pgfmathsetmacro\vve{-.55};
\pgfmathsetmacro\vvb{.55};
\pgfmathsetmacro\vvd{1};
\pgfmathsetmacro\yup{3};
\pgfmathsetmacro\ydown{-2.5};
\pgfmathsetmacro\xdown{1.5};

\join{thin}{0}{0}{0}{\yup}
\join{thin}{0}{0}{-\xdown}{\ydown}
\join{thin}{0}{0}{\xdown}{\ydown}

\vblownup{lime}{0}{\yup}{\phantom{$v_1$}}{\phantom{$t_1$}}{\phantom{$t_2$}}{\phantom{$t_3$}}{\vertexedges}{ultra thick}
\vblownup{yellow}{0}{0}{$\, \, v \, \, $}{\phantom{$s_1$}}{$\, \, s \,\, $}{\phantom{$s_3$}}{\vertexedges}{ultra thick}
\vblownup{pink}{-\xdown}{\ydown}{\phantom{$v_3$}}{\phantom{$t_1$}}{\phantom{$t_2$}}{\phantom{$t_3$}}{\vertexedges}{ultra thick}
\vblownup{cyan}{\xdown}{\ydown}{\phantom{$v_4$}}{\phantom{$t_1$}}{\phantom{$t_2$}}{\phantom{$t_3$}}{\vertexedges}{ultra thick}


\begin{scope}[shift={(6,0)}]

\join{dotted}{0}{0}{0}{\yup}
\join{dotted}{0}{0}{-\xdown}{\ydown}
\join{dotted}{0}{0}{\xdown}{\ydown}

\vblownup{lime}{0}{\yup}{\phantom{.}}{\phantom{.}}{\phantom{.}}{\phantom{.}}{\vertexedges}{dotted}
\vblownup{yellow}{0}{0}{$\, \, v \, \, $}{$s_1$}{$s_2$}{$s_3$}{\vertexedges}{dotted}
\vblownup{pink}{-\xdown}{\ydown}{\phantom{.}}{\phantom{.}}{\phantom{.}}{\phantom{.}}{\vertexedges}{dotted}
\vblownup{cyan}{\xdown}{\ydown}{\phantom{.}}{\phantom{.}}{\phantom{.}}{\phantom{.}}{\vertexedges}{dotted}

\node at (1,-4) {$\sat(\Delta)$};

\end{scope}

\begin{scope}[shift={(-6,0)}]

\join{dotted}{0}{0}{0}{\yup}
\join{dotted}{0}{0}{-\xdown}{\ydown}
\join{dotted}{0}{0}{\xdown}{\ydown}

\vblownup{lime}{0}{\yup}{\phantom{$v_1$}}{\phantom{$t_1$}}{\phantom{$t_2$}}{\phantom{$t_3$}}{\vertexedges}{ultra thick}
\vblownup{yellow}{0}{0}{\phantom{.}}{\phantom{.}}{\phantom{.}}{\phantom{.}}{\vertexedges}{dotted}
\vblownup{pink}{-\xdown}{\ydown}{\phantom{$v_3$}}{\phantom{$t_1$}}{\phantom{$t_2$}}{\phantom{$t_3$}}{\vertexedges}{ultra thick}
\vblownup{cyan}{\xdown}{\ydown}{\phantom{$v_4$}}{\phantom{$t_1$}}{\phantom{$t_2$}}{\phantom{$t_3$}}{\vertexedges}{ultra thick}

\node at (1,-4) {$\link(\Delta) = \link_{\bXi}(v)$};

\end{scope}

\end{tikzpicture}
\caption{$\bXi$--type simplex: Part of the link and saturation of a $1$--simplex $\Delta$ with vertices $s,v=v(s)$.
Again, vertices of a common color correspond to the same element of $\vtx$.
} \label{F:link-sat-Xi}
\end{figure}
\begin{proof} It is clear from the definitions that the link of $\Delta$ is as described. Also, any simplex with vertex set of the form $\{v(s),t\}$ for some $t \in \kvtx^{v(s)}$ has the same link as $\Delta$. Therefore, to prove that the saturation is as described we are left to show if a simplex $\Delta'$ has the same link as $\Delta$, then its vertex set is contained in the set we described. If $w \in \vtx$ is a vertex of $\Delta'$, then $d_{\tree}(v,w)=1$ for all neighbors $v$ of $v(s)$ in $T_{\alpha(v)}$. This implies $w=v(s)$. Similarly, if $t \in \kvtx$ is a vertex of $\Delta'$, then $v(t)=v(s)$, and we are done.

Let us show that the map given in the statement is coarsely Lipschitz. To do so, it suffices to consider $w \neq w'$ with $d_{\tree}(w, v(s))=d_{\tree}(w', v(s))=1$ and connected by an edge in $\CX^{+\CW}$, and show that they are connected by a bounded-length path in $\bXi^v$. We argue below that $\bar d(\bTheta^w\cap \bTheta^{v(s)},\bTheta^{w'}\cap \bTheta^{v(s)})\leq R'$. Once we do that, the existence of the required bounded-length path follows directly from Lemma \ref{lem:parameter_for_Xi}.
  
 Let us now prove the desired inequality. Notice that $w,w'$ cannot be connected by an edge of $\CX$ since they are both distance 1 from $v(s)$ in the tree $T_{\alpha(v)}$. Hence, $w$ and $w'$ are contained respectively in maximal simplices $\sigma(t,u)$ and $\sigma(t',u')$ connected by an edge in $\mathcal W$. Say, up to swapping $t$ with $u$ and/or $t'$ with $u'$, that $v(t)=w$ and $v(t')=w'$. In either case of the definition of the edges of $\CW$ we have $\bar d(M(t),M(t'))\leq 10 R$. Since $M(t,s)\subseteq \bTheta^w\cap \bTheta^{v(s)}$ and $M(t',s)\subseteq \bTheta^{w'}\cap \bTheta^{v(s)}$, using Lemma \ref{lem:connect_not_far} we get
  $$\bar d(\bTheta^w\cap \bTheta^{v(s)},\bTheta^{w'}\cap \bTheta^{v(s)})\leq \bar d(M(t,s),M(t',s))\leq R',$$
  as we wanted.
  
Conversely, if $w,w'$ as above are joined by an edge in $\bXi^{v(s)}$ we will now show that they are also connected by an edge in $\mathcal C(\Delta)$. By definition of $\bXi^{v(s)}$, we have
$\bTheta^w\cap\bTheta^{w'}\neq\emptyset$.
By Lemma~\ref{lem:dense}\eqref{item:dense bTheta}, There exists $t,u,t',u'$ with $v(t) = w$ and $v(t') = w'$, so that $M(t,u) \cap M(t',u') \neq \emptyset$.
In particular, $\bar d(M(t,u),M(t',u'))\leq 10 R$. 
This says that $\sigma(t,u)$ and $\sigma(t',u')$ are connected in $\CW$, and hence that $w$ and $w'$ are connected in $\CX^{+\CW}$, as required.
\end{proof}

There is also an important type of $2$--dimensional simplex, which we call a {\em $\mathcal K$--type simplex}, due to the next lemma. See Figure~\ref{F:link-sat-KL}.

\begin{lemma} [$\mathcal K$--type simplex]  \label{L: KL simplex}
Let $\Delta$ be a $2$--simplex of $\CX$ with vertices $s,v(s),w$, for $s \in \kvtx$ and $w \in \vtx$ with $d_{\tree}(w,v(s))=1$. Then 
\[ \link(\Delta) = \CK^w \quad \mbox{ and } \quad \Sat(\Delta) = \{w\} \cup \link_{\bXi}(w)^{(0)}.\]
Moreover, $\mathcal C(\Delta)$ is quasi-isometric to $\mathcal K^w$, the quasi-isometry being the identity at the level of vertices.
\end{lemma}
\begin{figure}[h]
\begin{tikzpicture}[scale=.7,inner sep=0pt, dot/.style={fill=black,circle,minimum size=3pt}]

\coordinate (origin) at (0,0);


\pgfmathsetmacro\vva{1.5};
\pgfmathsetmacro\vvc{1};
\pgfmathsetmacro\vve{-.55};
\pgfmathsetmacro\vvb{.55};
\pgfmathsetmacro\vvd{1};
\pgfmathsetmacro\yup{3};
\pgfmathsetmacro\ydown{-2.5};
\pgfmathsetmacro\xdown{1.5};

\join{thin}{0}{0}{0}{\yup}
\join{thin}{0}{0}{-\xdown}{\ydown}
\join{thin}{0}{0}{\xdown}{\ydown}

\vblownup{lime}{0}{\yup}{$\, \, v \, \,$}{\phantom{$t_1$}}{$\, \, s \, \, $}{\phantom{$t_3$}}{\vertexedges}{ultra thick}
\vblownup{yellow}{0}{0}{$\, \, w \, \, $}{\phantom{$s_1$}}{\phantom{$s_2$}}{\phantom{$s_3$}}{\vertexedges}{ultra thick}
\vblownup{pink}{-\xdown}{\ydown}{\phantom{$v_3$}}{\phantom{$t_1$}}{\phantom{$t_2$}}{\phantom{$t_3$}}{\vertexedges}{ultra thick}
\vblownup{cyan}{\xdown}{\ydown}{\phantom{$v_4$}}{\phantom{$t_1$}}{\phantom{$t_2$}}{\phantom{$t_3$}}{\vertexedges}{ultra thick}


\begin{scope}[shift={(6,0)}]

\join{dotted}{0}{0}{0}{\yup}
\join{dotted}{0}{0}{-\xdown}{\ydown}
\join{dotted}{0}{0}{\xdown}{\ydown}


\vblownup{lime}{0}{\yup}{\phantom{$v_1$}}{\phantom{$t_1$}}{\phantom{$t_2$}}{\phantom{$t_3$}}{\vertexedges}{dotted}
\vblownup{yellow}{0}{0}{\phantom{$\, \, w \, \, $}}{\phantom{.}}{\phantom{.}}{\phantom{.}}{\vertexedges}{dotted}
\vblownup{pink}{-\xdown}{\ydown}{\phantom{$v_1$}}{\phantom{$t_1$}}{\phantom{$t_2$}}{\phantom{$t_3$}}{\vertexedges}{dotted}
\vblownup{cyan}{\xdown}{\ydown}{\phantom{$v_1$}}{\phantom{$t_1$}}{\phantom{$t_2$}}{\phantom{$t_3$}}{\vertexedges}{dotted}

\node at (1,-4) {$\sat(\Delta)$};

\end{scope}

\begin{scope}[shift={(-6,0)}]

\join{dotted}{0}{0}{0}{\yup}
\join{dotted}{0}{0}{-\xdown}{\ydown}
\join{dotted}{0}{0}{\xdown}{\ydown}

\vblownup{lime}{0}{\yup}{\phantom{.}}{\phantom{.}}{\phantom{.}}{\phantom{.}}{\vertexedges}{dotted}
\vblownup{yellow}{0}{0}{\phantom{.}}{\phantom{$s_1$}}{\phantom{$s_2$}}{\phantom{$s_3$}}{\vertexedges}{dotted}
\vblownup{pink}{-\xdown}{\ydown}{\phantom{.}}{\phantom{.}}{\phantom{.}}{\phantom{.}}{\vertexedges}{dotted}
\vblownup{cyan}{\xdown}{\ydown}{\phantom{.}}{\phantom{.}}{\phantom{.}}{\phantom{.}}{\vertexedges}{dotted}

\node at (1,-4) {$\link(\Delta) = \CK^{w}$};

\end{scope}

\end{tikzpicture}
\caption{$\mathcal K$--type simplex: Part of the link and saturation of a $2$--simplex $\Delta$ with vertices $s,v=v(s),w$ (colors correspond to elements of $\vtx$).
}\label{F:link-sat-KL}
\end{figure}
\begin{proof}
It is clear from the definitions that the link of $\Delta$ is as described. Also, any simplex with vertex set of the form $\{t,v(t),w\}$ for some $t \in \kvtx$ with $d_{\tree}(v(t),w)=1$ has the same link as $\Delta$. Therefore, to prove that the saturation is as described we are left to show if a simplex $\Delta'$ has the same link as $\Delta$, then its vertex set is contained in the set we described. Given any vertex $u \in \vtx$ of $\Delta'$, it has to be connected to all $t \in \kvtx$ with $v(t)=w$, implying that either $u=w$ or $d_{\tree}(u,w)=1$, as required for vertices in $\vtx$. Similarly, any vertex $u \in \kvtx$ of $\Delta'$ has to be connected to all $t \in \kvtx$ with $v(t)=w$, implying $d_{\tree}(v(u),w)=1$, and we are done.

 To prove that $\mathcal C(\Delta)$ is naturally quasi-isometric to $\mathcal K^w$, it suffices to show that if $u,t\in\mathcal C(\Delta)$ are connected by an edge in $\CX^{+\CW}$, then they are uniformly close in $\mathcal K^w$ and that, vice versa, if $d_{\mathcal K^w}(u,t)=1$, then $u,t$ are connected by an edge in $\mathcal C(\Delta)$.
 
 First, if $u,t\in\mathcal C(\Delta)$ are connected by an edge in $\CX^{+\CW}$, then there exist $u',t'$ so that $\bar d(M(u,t'),M(u',t))\leq R'$ (see Lemma \ref{lem:connect_not_far}). In particular, $\bar d(M(u),M(t))\leq R'$, which in turn gives a uniform bound on the distance in the path metric of $\bTheta^w$ between $M(u)$ and $M(t)$ because the metrics $\bar{d}$ and the path metric on $\theta^w$ are coarsely equivalent. By Proposition \ref{prop:KL}(\ref{item:lambda_lipschitz}) we must also have a uniform bound on $d_{\mathcal K^w}(u,t)$.
 
 Now suppose that $d_{\mathcal K^w}(u,t)=1$.  We can then deduce from Proposition \ref{prop:KL}(\ref{item:KL_qi}) that $\bar d(M(s',u),M(s',t))\leq K_1^2+K_1\leq 10R$ for any fixed $s'\in \mathcal K^{v(s)}$ (the proposition yields the analogous upper bound in the path metric of $\bTheta^w$, which is a stronger statement). Therefore $u,t$ are connected by an edge in $\CX^{+\CW}$, whence in $\mathcal C(\Delta)$, as required.
\end{proof}

The remaining simplices are not particularly interesting as their links are joins (or points), and hence have diameter at most $2$, but we will still need to verify properties for them.  We define the {\em type} of a simplex in $\CX$ to be the graph isomorphism type of its link.  We describe these simplices with finite diameter links in the next lemma.
Recall the definition of $\link_{\bXi}(w)$ in \eqref{E:Xi link}.

\begin{lemma} \label{L:remaining simplices} The following is a list of all types of non-empty, nonmaximal, simplices $\Delta$ of $\CX$ that are not of $\bXi$--type or $\mathcal K$--type, together with their links.  In each case, the link is a nontrivial join (or a point), and $\mathcal C(\Delta)$ has diameter at most $2$.  

In the table below the simplices $\Delta$ have vertices $u,w \in \vtx$  with $d_{\tree}(u,w) = 1$ and $s,t \in \kvtx$ with $d_{\tree}(v(s),v(t)) = 1$ and $d_{\tree}(v(s),u) =1$.\\
\begin{center}
\begin{tabular}{|c|c|c|}
\hline
$\Delta$ & $\link(\Delta)$ \\
\hline
$\phantom{\frac{\int_1^2}{\int_1^2}}\{u\}\phantom{\frac{\int_1^2}{\int_1^2}}$ & \quad $\CK^u \star \link_\bXi(u)$ \quad \\
\hline
$\phantom{\frac{\int_1^2}{\int_1^2}}\{s\}\phantom{\frac{\int_1^2}{\int_1^2}}$ & \quad $\{v(s)\} \star\link_\bXi(v(s))$ \quad \\
\hline
$\phantom{\frac{\int_1^2}{\int_1^2}}\{u,w\}\phantom{\frac{\int_1^2}{\int_1^2}}$ & $\CK^u \star \CK^w$ \\
\hline
$\phantom{\frac{\int_1^2}{\int_1^2}}\{s,t\}\phantom{\frac{\int_1^2}{\int_1^2}}$ & $\{ v(s),v(t) \}$ \\
\hline
$\phantom{\frac{\int_1^2}{\int_1^2}}\{s,u\}\phantom{\frac{\int_1^2}{\int_1^2}}$  & $\{v(s)\} \star \CK^u$ \\
\hline
$\phantom{\frac{\int_1^2}{\int_1^2}}\{s,t,v(s)\}\phantom{\frac{\int_1^2}{\int_1^2}}$ & $\{v(t)\}$ \\
\hline
\end{tabular}
\end{center}

\medskip

\end{lemma}
\begin{proof} This is straightforward given the definition of $\CX$ and we leave its verification to the reader.  Referring to Figure~\ref{fig:blowup}, and comparing with Figures~\ref{F:link-sat-Xi} and \ref{F:link-sat-KL}, may be helpful.
\end{proof}
The next lemma collects a few additional properties we will need.  There are $9$ types of nonempty simplices: maximal, $\bXi$--type, $\CK$--type, and the six types listed in Lemma~\ref{L:remaining simplices}.

\begin{lemma}  \label{L: nice join properties }
The following hold in $\CX$.
\begin{enumerate}[$(a)$]
\item The link of a simplex with a given type cannot be strictly contained in the link of a simplex with the same type.
   \item For all non-maximal simplices $\Delta$ and $\Delta'$ so that there is a simplex $\Omega$ with $\link(\Omega)\subseteq \link(\Delta')\cap \link(\Delta)$ and $diam(\mathcal C (\Omega))>3$, there exists a simplex $\Pi$ in the link of $\Delta'$ with $\link(\Delta'*\Pi)\subseteq \link(\Delta)$ so that for any $\Omega$ as above we have $\link(\Omega)\subseteq \link(\Delta'*\Pi)$.
 \end{enumerate}
\end{lemma}
\begin{proof} Part $(a)$ follows directly from the descriptions of the simplices given in Lemmas~\ref{L: Xi simplex}, \ref{L: KL simplex}, and \ref{L:remaining simplices}, and we leave this to the reader.

Before we prove $(b)$, we suppose $\link(\Delta') \cap \link(\Delta) \neq \emptyset$, and make a few observations.  First, $\Delta'$ and $\Delta$ must project by $Z$ to the same tree: $Z(\Delta'),Z(\Delta) \subset T_\alpha$ for some $\alpha \in \mathcal P$.  Next, note that $Z(\link(\Delta') \cap \link(\Delta))$ is contained in the intersection of the stars in $T_\alpha$ of $Z(\Delta')$ and $Z(\Delta)$.  Moreover, (as in any tree) the intersection of these two stars is contained in a single edge, or else $Z(\Delta') = Z(\Delta) = \{w\} \in T_\alpha^{(0)} \subset \vtx$ is a single point.  In this latter case, by \eqref{E: Z preimage of vertex}, we have
\[ \Delta,\Delta' \subset \{w\} \star \CK^w.\]

Next, note that for any $\CK$--type simplex $\Omega = \{s,v(s),w\}$,  $\link(\Omega) = \CK^w$ by Lemma~\ref{L: KL simplex}, and if $\link(\Omega) \subset \link(\Delta') \cap \link(\Delta)$, then $w$ is in the intersection of the stars of $Z(\Delta')$ and $Z(\Delta)$.  For a $\bXi$--type simplex $\Omega = \{s,v(s)\}$, $\link(\Omega) = \link_{\bXi}(v(s))$ by Lemma~\ref{L: Xi simplex}, and together with Lemma~\ref{L:remaining simplices} and the previous paragraph, we see that $\link(\Omega) \subset \link(\Delta') \cap \link(\Delta)$ if and only if $Z(\Delta') = Z(\Delta) = \{v(s)\}$ in $T_\alpha$. 

With these observations in hand, we proceed to the proof of $(b)$, which divides into two cases.

\smallskip

\noindent {\bf Case 1.} There is a $\bXi$--type simplex $\Omega = \{s,v(s)\}$ with $\link(\Omega) \subset \link(\Delta') \cap \link(\Delta)$.

\smallskip

In this situation, $Z(\Delta') = Z(\Delta) = \{v(s)\}$, and thus $\Delta,\Delta' \subset \{v(s)\} \star \CK^{v(s)}$. 
From Lemmas~\ref{L: Xi simplex} and \ref{L:remaining simplices}, we see that $\link(\Delta') \cap \link(\Delta)$ must be equal to one of $\link(\Omega)$, $\link(\{s\})$, or $\link(\{v(s)\})$.  Inspection of these links shows that $\link(\Omega)$ is the only link of a $\bXi$--type simplex contained in it.  First suppose that $\link(\Delta') \cap \link(\Delta)$ has the form $\link(\Omega)$ or $\link(\{s\})$.  In this situation, we easily find $\Pi \subset \link(\Delta')$ so that $\link(\Omega) = \link(\Delta' \star \Pi)$.  Furthermore, for any $\CK$--type simplex link $\CK^w$ in the intersection, we must have $\CK^w \subset \link(\Omega) = \link(\Delta' \star \Pi)$.  Therefore, the link of any $\bXi$--or $\CK$--type simplex contained in $\link(\Delta') \cap \link(\Delta)$ must be contained in $\link(\Omega) = \link(\Delta' \star \Pi)$, as required.  Now suppose instead that $\link(\Delta') \cap \link(\Delta) = \link(\{v(s)\})$.  By Lemma~\ref{L:remaining simplices}, we see that $\Delta = \Delta' = \{v(s)\}$.  In this case, setting $\Pi = \emptyset$ trivially completes the proof since then $\link(\Delta') \cap \link(\Delta) = \link(\Delta') = \link(\Delta'\star \Pi)$.

\smallskip

\noindent {\bf Case 2.} No link of a $\bXi$--type simplex is contained in $\link(\Delta') \cap \link(\Delta)$.

\smallskip

From the observations above, $Z(\Delta)$ and $Z(\Delta')$ do not consist of the same single point, and hence the stars of $Z(\Delta)$ and $Z(\Delta')$ intersect in either a point or an edge in $T_\alpha$.  Since $Z(\link(\Delta') \cap \link(\Delta))$ is contained in the intersection of these stars, there are at most two $\CK$--type simplices whose links are contained in $\link(\Delta') \cap \link(\Delta)$.  If there are two $\CK$--type simplices $\Omega,\Omega'$ with $\CK^w = \link(\Omega)$ and $\CK^u = \link(\Omega')$ contained in $\link(\Delta') \cap \link(\Delta)$, then observe that
\[ \CK^u,\CK^w \subset \link(\{u,w\}) = \CK^u \star \CK^w \subset \link(\Delta') \cap \link(\Delta).\] 
By inspection of the possible links in Lemmas~\ref{L: Xi simplex}, \ref{L: KL simplex}, and \ref{L:remaining simplices}, it must be that $\Delta'$ is either $\{u\}$, $\{w\}$, or $\{u,w\}$, and so setting $\Pi$ to be $\{w\}$, $\{u\}$, or $\emptyset$, respectively, we are done.  On the other hand, if there is exactly one $\CK$--type simplex $\Omega$ with $\CK^w = \link(\Omega) \subset \link(\Delta') \cap \link(\Delta)$, then again inspecting all possible situations, we can find $\Pi \subset \link(\Delta')$ with $\CK^w =\link(\Omega) = \link(\Delta' \star \Pi)$, and again we are done with this case.  This completes the proof.
\end{proof}

\begin{lemma}\label{lem:retractions}
 There exists $L\geq 1$ so that for every non-maximal simplex, there is an $(L,L)$--coarsely Lipschitz retraction $r_\Delta\colon Y_\Delta\to \pow({\mathcal C(\Delta)})$. In particular, $\mathcal C(\Delta)$ is uniformly quasi-isometrically embedded in $Y_\Delta$.
\end{lemma}

\begin{proof} By Lemma~\ref{L:remaining simplices}, we only have to consider simplices of $\bXi$--and $\mathcal K$--type. 
 
Consider $\Delta = \{s,v\}$ with $v = v(s)$ of $\bXi$--type first.  Recall from Lemma~\ref{L: Xi simplex} that $\bXi^v$ naturally includes into $\C(\Delta)$ by a quasi-isometry.  Here we will use make use of the map $\xi^v$, whose relevant properties for our current purpose are stated in Proposition \ref{prop:R-projections}. For a vertex $u \in \vtx \smallsetminus \Sat(\Delta)$ (so, $u \neq v$) we define $r_\Delta(u)=\xi^v(\bTheta^{u})$.  For $t \in \kvtx \smallsetminus \Sat(\Delta)$ we define $r_\Delta(t)=\xi^v(\bTheta^{v(t)})$. Notice that the sets $r_\Delta(u)$ are uniformly bounded by Proposition \ref{prop:R-projections} (and Lemma~\ref{L: Xi simplex}).  Also, $r_\Delta$ is coarsely the identity on the vertices of $\link(\Delta)$ in $\vtx$ by Equation~\eqref{E:an obvious but useful fact} and Proposition \ref{prop:R-projections}(2b).  To check that $r_\Delta$ is coarsely Lipschitz it suffices to consider $\CX^{+\CW}$--adjacent vertices of $\vtx$. Notice that vertices $w,w' \in \vtx$ that are adjacent in $\CX^{+\CW}$ have corresponding $\bTheta^w,\bTheta^{w'}$ within $10R$ of each other in $\bar E$. Indeed, $\bTheta^v$ and $\bTheta^w$ actually intersect if $w,w'$ are adjacent in $\CX$, and they contain subsets $M(\cdot)$ within $10R$ of each other if $w,w'$ are contained in $\CW$--adjacent maximal simplices (this is true regardless of which case of the definition \eqref{Eq:Def W-edges} for the edges of $\CW$ applies). The fact that $r_\Delta$ is coarsely Lipschitz now follows from, Proposition~\ref{prop:R-projections}(1), which says that $\xi^v$ is coarsely Lipschitz on $\bar E$.
 
Next, consider $\Delta = \{s,v(s),w\}$ of $\mathcal K$--type.  For a vertex $u \in \vtx \smallsetminus \Sat(\Delta)$ (so, $d_{\tree}(u,w) \geq 2$), define $r_\Delta(u)=\Lambda^w(\bTheta^u)$.  For a vertex $t \in \kvtx \smallsetminus \Sat(\Delta)$, define $r_\Delta(t)=\Lambda^w(M(t))$.  Notice that, by definition of $M(t)$, if $t \in \CK^w$, then $r_\Delta(t)$ lies within Hausdorff distance $K_1$ of $t$.  Also, since $d_{\tree}(u,w) \geq 2$, Proposition~\ref{prop:R-projections}(2a) ensures that the diameter of $\Lambda^w(\bTheta^u)$ is bounded. Since $M(t)\subseteq \bTheta^{v(t)}$, we see that all the sets in the image of $r_\Delta$ are bounded, and also we see that in order to prove that $r_\Delta$ is Lipschitz it suffices to consider vertices of $\kvtx$.  But vertices $s,t \in \kvtx$ that are adjacent in $\CX^{+\CW}$ have corresponding $M(s),M(t)$ within $10R$ of each other in $\bar E$, so the conclusion follows from Proposition~\ref{prop:R-projections}(1), which states that $\Lambda^w$ is coarsely Lipschitz on $\bar E$.
\end{proof}
 
\subsection{Final proof}\label{subsec:final_proof} We now have all the tools necessary for the:

\begin{proof}[Proof of Theorem \ref{thm:E_is_chhs}]  We must verify each of the conditions from Definition~\ref{defn:combinatorial_HHS}.

Item \eqref{item:chhs_flag} (bound on length of $\nest$--chains) follows from Lemma~\ref{L: nice join properties }(a), which implies that any chain $\link(\Delta_1)\subsetneq \dots$ has length bounded by the number of possible types, which is 9.
 
 Let us now discuss item \eqref{item:chhs_delta} of the definition. The descriptions of the $\mathcal C(\Delta)$ from Lemmas~\ref{L: empty simplex case}, \ref{L: Xi simplex}, \ref{L: KL simplex}, \ref{L:remaining simplices}  yields that all $\mathcal C(\Delta)$ are hyperbolic, since each of them is either bounded or uniformly quasi-isometric to one of $\hat{E}$ (which is hyperbolic by Theorem \ref{T:quotient metric on E hat}), some $\bXi^v$ (which is hyperbolic by Lemma \ref{lem:Xi_hyp}), or $\mathbb R$. Moreover, any $\mathcal C(\Delta)$ is (uniformly) quasi-isometrically embedded in $Y_\Delta$ by Lemma \ref{lem:retractions}.

Item \eqref{item:chhs_join} of the definition (common nesting) is precisely Lemma \ref{L: nice join properties }(b).

Finally, we show item \eqref{item:C_0=C} of the definition (fullness of links), which we recall for the convenience of the reader:
 \begin{itemize}
  \item   If $v,w$ are distinct non-adjacent vertices of $\link(\Delta)$, for some simplex $\Delta$ of $\CX$, contained in 
$\CW$--adjacent maximal simplices, then they are contained in $\CW$--adjacent simplices of the form $\Delta\star\Delta'$.
 \end{itemize}

It suffices to consider simplices $\Delta$ of $\bXi$--and $\mathcal K$--type.  Indeed, in all other cases (see Lemma~\ref{L:remaining simplices}), the vertices $v,w$ under consideration are contained in the link of a simplex $\Delta'$ containing $\Delta$ where $\Delta'$ is of $\bXi$--or $\mathcal K$--type (as can be seen by enlarging $\Delta$ until its link is no longer a join; $v$ and $w$ are not $\CX$--adjacent so they are contained in the same ``side'' of any join structure). Hence, once we deal with those cases, we know that there are suitable maximal simplices containing the larger simplex, whence $\Delta$.

  Consider first a simplex $\Delta$ of $\mathcal K$--type with vertices $s,v(s),w$. Consider distinct vertices $t_1,t_2$ (necessarily in $\kvtx$) of $\link(\Delta)$, and suppose that there are vertices $s_1,s_2 \in \kvtx$ so that the maximal simplices $\sigma(s_1,t_1)$ and $\sigma(s_2,t_2)$ are connected in $\CW$. There are two possibilities:
 
 \begin{itemize}
  \item $\bar d(M(s_1,t_1),M(s_2,t_2))\leq 10 R$. In this case, we have $\bar d(M(t_1),M(t_2))\leq 10R$.
    In particular, in view of the second bullet in the definition of the edges of $\CW$, we have that $t_1,t_2$ are contained, respectively, in the $\CW$--connected maximal simplices $\Delta*t_1=\sigma(s,t_1)$ and $\Delta* t_2=\sigma(s,t_2)$.
  \item $s_1=s_2$ and $\bar d(M(t_1),M(t_2))\leq 10R$ (notice that $t_1\neq t_2$ so that the ``symmetric'' case cannot occur). Again, we reach the same conclusion as above.
   \end{itemize}

 We can now consider a simplex $\Delta$ of $\bXi$--type with vertices $s,v(s)$. Consider vertices $x_1,x_2$ of $\link(\Delta)$ that are not $\CX$--adjacent but are contained in $\CW$--adjacent maximal simplices. Furthermore, we can assume that $x_1,x_2$ are not in the link of a simplex of $\mathcal K$--type (the case we just dealt with) which contains $\Delta$, since in that case we already know that there are suitable maximal simplices containing the larger simplex, whence $\Delta$. Then, using the structure of $\link(\Delta)$, we see that there are vertices $s_i,t_i \in \kvtx$ so that:
 \begin{itemize}
  \item  $x_i\in \{t_i,v(t_i)\}$,
  \item $t_i$ and $v(t_i)$ all belong to $\link(\Delta)$, and
    \item $\sigma(s_1,t_1)$ and $\sigma(s_2,t_2)$ are connected in $\CW$.
 \end{itemize}

 In turn, the last bullet splits into two cases:
 
 \begin{itemize}
  \item $\bar d(M(s_1,t_1),M(s_2,t_2))\leq 10 R$. In this we have $\bar d(M(t_1),M(t_2))\leq 10 R$, so that $x_1,x_2$ are contained, respectively, in the $\CW$--connected maximal simplices $\Delta*\{t_1, v(t_1)\}=\sigma(s,t_1)$ and $\Delta*\{t_2, v(t_2)\}=\sigma(s,t_2)$, so this case is fine.
  \item $s_1=s_2$ and $\bar d(M(t_1),M(t_2))\leq 10 R$. But again we reach the same conclusion as before.
 \end{itemize}

 We now also have to check the existence of an action of $\Gamma$ with the required properties. The action is constructed in Lemma \ref{lem:W=barE}, where all properties are checked except finiteness of the number of orbits of links of $\CX$. The finitely many possible types of links are listed in Lemmas \ref{L: empty simplex case} -- \ref{L:remaining simplices}, and for each type of simplex there are only finitely many orbits, so we are done.
\end{proof}

\section{Quasi-isometric rigidity} \label{S:QI rigidity}

In this section, using the HHS structure, we prove a strong form of rigidity for the group $\Gamma$ and the model space $\bar E$.  Recall that $\bar E$ is defined via a particular truncation $\bar D$ of the Teichm\"uller disk $D$ obtained by removing $1$--separated horoballs.  We say that such a truncation $\bar E$ is an {\em allowable truncation} of $E$ if $\Gamma$ acts by isometries on it with cocompact quotient.  Write $\Isom(\Omega)$ and $\QI(\Omega)$ for the isometry group and quasi-isometry group, respectively, of a metric space $\Omega$.   For $\bar E$, we write $\Isomfib(\bar E) \le \Isom(\bar E)$ for the subgroup of isometries that map fibers to fibers.

\begin{qi-theorem}
There is an allowable truncation $\bar E$ of $E$ such that the natural homomorphisms $\Isomfib(\bar E) \to \Isom(\bar E) \to \QI(\bar E) \cong \QI(\Gamma)$ are all isomorphisms, and $\Gamma  \le \Isom(\bar E) \cong \QI(\Gamma)$ has finite index.
\end{qi-theorem}

The proof is divided up into several steps which we outline here before getting into the details.  The first step is to use the HHS structure to identify certain quasi-flats in $\bar E$, and prove that they are coarsely preserved by a quasi-isometry.  The maximal quasi-flats are encoded by the strip bundles in $\bar E$, and using the preservation of quasi-flats, we show that a quasi-isometry further preserves strip bundles, and even sends all strip bundles for strips in any fixed direction to strip bundles in some other fixed direction.  From there we deduce that a quasi-isometry sends fibers $E_X$ within a bounded distance of some other fibers $E_Y$, and in fact induces a quasi-isometry between the fibers.  Fixing attention on $E_0$ and further appealing to the structure of strip bundles, we show that a self quasi-isometry of $\bar E$ induces a special type of quasi-isometry from $E_0$ to itself sending strips in a fixed direction within a uniformly bounded distance of strips in some other fixed direction.  This quasi-isometry is promoted to a piecewise affine biLipschitz map from $E_0$ to itself, which we then show is in fact affine.  This produces 
a homomorphism to the full affine group of $E_0$, $\QI(\bar E) \to \Aff(E_0)$.  Given an affine homeomorphism of $E_0$, we construct an explicit fiber preserving isometry associated to it, which via the inclusions $\Isomfib (\bar E) \to \QI(\bar E)$ serves as a one-sided inverse.  Finally, we prove that the homomorphism  $\QI(\bar E) \to \Aff(E_0)$ is injective, hence the homomorphisms $\Isomfib(\bar E) \to \Isom(\bar E) \to \QI(\bar E) \to \Aff(E_0)$ are all isomorphisms.  The fact that $\Gamma$ has finite index in $\Isomfib(\bar E)$, and hence in $\Isom(\bar E)$, is straightforward using the cocompactness of the action of $\Gamma$ and the singular structure.

\subsection{HHS structure and quasi-flats}
\label{sec:HHS-quasi-flats}

Denote by $\mathfrak S_0$ the set $\mathcal V\times \{0,1\}$. We denote the element $(v,0)$ by $v^{qt}$ (for ``quasi-tree'') and $(v,1)$ by $v^{ql}$ (for ``quasi-line'').

We denote by $\mathcal F$ the set of all strip bundles of $\bar E$, that is, subbundles with fiber a strip and base the horocycle corresponding to the direction of the strip. (Roughly, these are the flats of the peripheral graph manifolds.)

\begin{prop}[Properties of the HHS structure]\label{prop:HHS_properties}
 The HHS structure $(\bar E,\mathfrak S)$ on $\bar E$ coming from Theorem \ref{thm:E_is_chhs} has the following properties, for some $K\geq 1$.
 
 \begin{enumerate}
  \item\label{item:asymphoric} The set of non-$\nest$-maximal $Y\in\mathfrak S$ with $diam(\mathcal C(Y))\geq 3$ is in bijection with $\mathfrak S_0$. Under said bijection:
  \item\label{item:quasi-lines-trees} $\mathcal C(v^{qt})$ is $(K,K)$-quasi-isometric to a quasi-tree with at least two points at infinity, and  $\mathcal C(v^{ql})$ is $(K,K)$--quasi-isometric to a line;
  \item\label{item:trees_orth_to_lines} For all $v\in\mathcal V$, we have $v^{qt}\orth v^{ql}$;
  \item\label{item:adacency_orth_relations} For all adjacent $v,w\in\mathcal V$, we have $w^{ql}\orth v^{ql}$ and $w^{ql}\nest v^{qt}$;
  \item\label{item:everyone_else_transverse} All pairs of elements of $\mathfrak S_0$ that do not fall into the aforementioned cases are transverse;
  \item\label{item:flats_are_flats} For each adjacent $v,w\in\mathcal V$ there is $F\in\mathcal F$ so $\pi_{v^{ql}}(F)$ and $\pi_{w^{ql}}(F)$ are $K$--coarsely dense, and $\pi_Y(F)$ has diameter at most $K$ for all $Y\neq v^{ql},w^{ql}$.
 \end{enumerate}
\end{prop}

\begin{proof}
The second paragraph of the proof of Theorem~\ref{thm:E_is_chhs} implies $\C(v^{qt})$ is quasi-isometric to the quasi-tree ${\bf \Xi}^v$ (with at least two points at infinity by Lemma~\ref{lem:Xi_hyp}) and $\C(v^{ql})$ is quasi-isometric to the quasi-line $\mathcal K^v$, and that these are the only non-maximal elements of diameter at least $3$.  This proves (\ref{item:asymphoric}) and (\ref{item:quasi-lines-trees}).
In view of the combinatorial description of orthogonality and nesting from Definition \ref{defn:orth}, properties (\ref{item:trees_orth_to_lines})-(\ref{item:everyone_else_transverse}) boil down to combinatorial properties of $\mathcal X$ that are straightforward to check. For example, regarding property (\ref{item:trees_orth_to_lines}) note that two (equivalence classes of) simplices are orthogonal if their links form a join. The links of the simplices corresponding to $v^{qt}$ and $v^{ql}$ are $\link_{\bXi}(v)$ and $\mathcal K^v$ (see Lemmas \ref{L: Xi simplex} and \ref{L: KL simplex}), which indeed form a join.

 Regarding property (\ref{item:flats_are_flats}), first of all the projections in the HHS structure on $\bar E$ are obtained composing the quasi-isometry $\bar E\to \mathcal W$ from Lemma \ref{lem:W=barE} and the projections defined in Definition \ref{defn:projections} (roughly, those are closest-point projections in the complement of saturations).
 
 The required strip bundle is the intersection $\bTheta^v\cap\bTheta^w$, which under the quasi-isometry of Lemma \ref{lem:W=barE} corresponds to the set of all maximal simplices of $\mathcal W$ of the form $\sigma(s,t)$ for $s\in \mathcal K^v$, $t\in \mathcal K^w$. In view of the description of the $\pi_Y$ from Definition \ref{defn:projections}, the coarse density claim follows since the union of the simplices described above contains the links of the simplices corresponding to $v^{ql}$ and $w^{ql}$, which are $\mathcal K^v$ and $\mathcal K^w$.
 
 Regarding the boundedness claim, it can be checked case-by-case that the set of simplices described above gives a bounded set of $Y_{\Delta}$ for $[\Delta]\neq v^{ql},w^{ql}$ (for example, note that said set is bounded if the saturation of $\Delta$ does not intersect $\mathcal K^v\cup\mathcal K^w$, or if it does not contain $v$ or $w$).
  This implies boundedness of the projections since the projections are coarsely Lipschitz; this follows from Theorem \ref{thm:E_is_chhs} since the projection maps of an HHS are required to be coarsely Lipschitz.
 \end{proof}

From now on we identify $\mathfrak S_0$ with the set of all $Y\in\mathfrak S$ with $\diam(\mathcal C(Y))\geq 3$ as in Proposition~\ref{prop:HHS_properties}.
Notice that the maximal number of pairwise orthogonal elements of $\mathfrak S_0$ is 2. Therefore, a \emph{complete support set} as in \cite[Definition 5.1]{BHS:quasiflats} is just a pair of orthogonal elements of $\mathfrak S_0$. 

Let $\mathcal H$ be the set of pairs $(Y,p)$ where $Y\in\mathfrak S$ and $p\in\partial C(Y)$ with $Y=v^{ql}$ for some $v\in\mathcal V$. We say that two such pairs $(Y,p)$ and $(W,q)$ are orthogonal if $Y$ and $W$ are.
Any element $\sigma=(Y,p)\in\mathcal H$ comes with a quasi-geodesic ray $h_\sigma$ in $\bar E$, as in \cite[Definition 5.3]{BHS:quasiflats}, so that $\pi_Y\circ h_\sigma$ is a quasi-geodesic in $\mathcal C(Y)$ and $\pi_W(h_\sigma)$ is bounded for all $W\neq Y$. 

We recall that given subsets $A$ and $B$ of a metric space $X$, we say that the subset $C$ of $X$ is the coarse intersection of $A$ and $B$ if for every sufficiently large $R$ we have that $N_R(A)\cap N_R(B)$ lies within finite Hausdorff distance of $C$. If the coarse intersection of two subsets exists, then it is well-defined up to finite Hausdorff distance.

\begin{lemma}\label{lem:map_on_hinges}
 Let $\phi\colon\bar E\to\bar E$ be a quasi-isometry. Then there is a bijection $\phi_{\mathcal H}\colon\mathcal H\to\mathcal H$ preserving orthogonality and so that $d_{\mathrm{Haus}}(\phi(h_\sigma),h_{\phi_{\mathcal H}(\sigma)})<\infty$.
\end{lemma}

\begin{proof}
Let $\mathcal H'$ be the set of pairs $(Y,p)$ where $Y\in\mathfrak S$ and $p\in\partial C(Y)$, without the restriction on $Y$.

 The lemma with $\mathcal H'$ replacing $\mathcal H$ would follow directly from \cite[Theorem 5.7]{BHS:quasiflats}, except that the HHS structure on $\bar E$ does not satisfy one of the 3 required assumptions, namely Assumption 2 (while it does satisfy Assumption 1 by parts (\ref{item:asymphoric}) and (\ref{item:quasi-lines-trees}) of Proposition \ref{prop:HHS_properties}, and it also satisfies Assumption 3 since there are no 3 pairwise orthogonal elements of $\mathfrak S_0$, by parts (\ref{item:trees_orth_to_lines})-(\ref{item:everyone_else_transverse}) of Proposition \ref{prop:HHS_properties}).
 
 Inspecting the proof of \cite[Theorem 5.7]{BHS:quasiflats}, we see that Assumption 2 is used in two places.
 
 The first one is to define the map $\phi_{\mathcal H}$ on a certain pair $\sigma=(Y,p)\in\mathcal H'$. The argument applies verbatim if $Y$ satisfies Assumption 2, that is, if and only if $Y$ is the intersection of 2 complete support sets. This is the case if $Y=v^{ql}$ for some $v\in\mathcal H$, that is, if $\sigma\in\mathcal H$. Therefore, one can use that argument to define a map $\phi_{\mathcal H}:\mathcal H\to\mathcal H'$. What is more, the image of $\phi_{\mathcal H}$ needs to be contained in $\mathcal H$. This can be seen from the fact that $h_{\phi_{\mathcal H}(\sigma)}$ for $\sigma\in\mathcal H$ arises as a coarse intersection of \emph{standard orthants}, which are, essentially, products of rays $h_\sigma$, see \cite[Definition 4.1]{BHS:quasiflats} for the precise definition. Notice that \cite[Lemma 4.11]{BHS:quasiflats} says, roughly, that coarse intersections of standard orthants are the expected sub-products. Hence, the failure of Assumption 2 for $Y = v^{qt}$ implies that $h_{(Y,p)}$ cannot be a coarse intersection of standard orthants, and therefore $\phi_{\mathcal H}(\sigma)$ for $\sigma\in\mathcal H$ also needs to lie in $\mathcal H$.

The second place where Assumption 2 is mentioned in \cite[Theorem 5.7]{BHS:quasiflats} is the proof  that $\phi_{\mathcal H}$ preserves orthogonality. There the assumption is used to say that certain quasi-geodesic rays are of the form $h_{\sigma}$. Such quasi-geodesic rays arise as coarse intersections of standard orthants, so, as mentioned above, they need to be of the form $h_\sigma$ for $\sigma\in\mathcal H$, hence Assumption 2 is not actually needed there.

Thus, the arguments in the proof of \cite[Theorem 5.7]{BHS:quasiflats} give the lemma.
\end{proof}

\begin{lemma}\label{lem:flats_to_flats}
For every $K$ there exists $C$ so that the following holds. Let $\phi\colon\bar E\to\bar E$ be a $(K,K)$--quasi-isometry. Then there is a bijection $\phi_{\mathcal F}\colon\mathcal F\to\mathcal F$ so that $d_{\mathrm{Haus}}(\phi(F),\phi_{\mathcal F}(F))\leq C$ for all $F\in\mathcal F$.
\end{lemma}

\begin{proof}
 Let $p^{\pm}$ be the two points at infinity of $\mathcal C(v^{ql})$ for some $v\in\mathcal V$. We claim that there exists $w\in\mathcal V$ so that, for $q^{\pm}$ the points at infinity of $\mathcal C(w^{ql})$, we have $\phi_{\mathcal H}( (v^{ql},p^{\pm}) )=(w^{ql},q^{\pm})$, up to relabeling. We use that $\phi_{\mathcal H}$ preserves orthogonality to show this. Let $u_1,u_2\in\mathcal V$ be distinct and adjacent to $v$, and let $r^\pm_1,r^\pm_2$ be the points at infinity of $\mathcal C(u_1^{ql})$, $\mathcal C(u_2^{ql})$. Then $(v^{ql},p^\pm)$ are the only elements of $\mathcal H$ that are orthogonal to all the $(u_i^{ql},r_i^\pm)$. Since $\phi_{\mathcal H}$ preserves orthogonality, we see that $\phi_{\mathcal H}((v^{ql},p^\pm))$ are both orthogonal to the same 4 distinct elements of $\mathcal H$ with the property that no pair of them is orthogonal. This is easily seen to imply that $\phi_{\mathcal H}((v^{ql},p^\pm))$ must be of the form $(w^{ql},q^{\pm})$, since said $4$ elements need to be associated to at least 2 distinct vertices of $\mathcal V$. This shows the claim.
 
 In view of the claim, we see that \cite[Lemma 5.9]{BHS:quasiflats} applies. (We note that Assumption 2 in said Lemma is only needed to have the map from \cite[Theorem 5.7]{BHS:quasiflats}, but our map from Lemma \ref{lem:map_on_hinges} has the same defining properties, just with a smaller domain and range.)
 
 The \emph{standard flats} in \cite[Lemma 5.9]{BHS:quasiflats} coarsely coincide with the elements of $\mathcal F$ in view of Proposition \ref{prop:HHS_properties}\eqref{item:flats_are_flats} (compare with \cite[Definition 4.1]{BHS:quasiflats}) so the lemma follows.
\end{proof}

Denote by $\mathcal S$ the collection of all strips in $E_0$, and for $A\in\mathcal S$ denote by $\alpha(A)\in\mathcal P$ the direction of $A$. Similarly, for $F\in\mathcal F$ we denote $\alpha(F)\in\mathcal P$ the direction of the strip defining $F$. 

\begin{prop}  \label{prop:strip_bijection}
 Given $K \geq 1$, there exists $C \geq 0$ so that if $\phi \colon \bar E \to \bar E$ is a $(K,K)$--quasi-isometry, then for all $X \in \bar D$, there exists $Y \in \bar D$ so that the Hausdorff distance between $\phi(E_X)$ and $E_Y$ is at most $C$.  In particular, $d_{\mathrm{Haus}}(E_0,\phi(E_0))<\infty$. 
  Moreover, there are bijections
$\phi_{\mathcal P}:\mathcal P\to\mathcal P$
 and 
 $\phi_{\mathcal S}:\mathcal S\to\mathcal S$
 so that:
 \begin{enumerate}
  \item $d_{\mathrm{Haus}}(\phi(\partial \mathcal B_\alpha), \partial \mathcal B_{\phi_{\mathcal P}(\alpha)})\leq C$ for each $\alpha\in\mathcal P$.
  \item $\alpha(\phi_{\mathcal S}(A))=\phi_{\mathcal P}(\alpha(A))$ and $d_{\mathrm{Haus}}(\phi(A), \phi_{\mathcal S}(A))<\infty$ for all $A\in \mathcal{S}$,
 \end{enumerate}
\end{prop}
\begin{proof}
First, note that fibers are quantitatively coarse intersections of the sets $\partial \mathcal B_\alpha$, in the sense that exists a function $f\colon \mathbb R\to \mathbb R$ and $t_0\geq 0$ such that
\begin{itemize}
 \item for any $X\in\bar D$ and any $t\geq t_0$ there are two distinct $\partial \mathcal B_\alpha$ whose $t$--neighborhoods intersect in a set within Hausdorff distance  $f(t)$ of $E_X$;
 \item for $t \ge t_0$, if the $t$--neighborhoods of two distinct $\partial \mathcal B_\alpha$ intersect, then this intersection lies within Hausdorff distance $f(t)$ of a fiber.
\end{itemize}
This follows via the bundle-map $\pi\colon \bar E\to \bar D$ and the corresponding relationship between neighborhoods of distinct horocycles $\partial B_\alpha$  in $\bar D$.

We next make three preliminary observations.  
Firstly, for each $\alpha\in \mathcal P$ the set $\partial \mathcal B_\alpha$ is the union  of all  $F\in \mathcal{F}$ with $\alpha(F) = \alpha$.
Secondly, if $F_1,F_2\in\mathcal F$ have $\alpha(F_1)\neq \alpha(F_2)$ then the coarse intersection of $F_1$ and $F_2$ is bounded. Indeed, $F_i$ is contained in $\partial \mathcal B_{\alpha(F_i)}$, and the coarse intersection of these $\partial \mathcal B_{\alpha(F_i)}$ is some (or really, any) fiber $E_X$.  Since the coarse intersection of $F_i$ with $E_X$ is a strip in the corresponding direction, and strips in different directions have bounded coarse intersection, the claim follows.
Thirdly, observe that $F_1,F_2\in\mathcal F$ have $\alpha(F_1)=\alpha(F_2)$ if and only if there is a chain of elements in $\mathcal F$ from $F_1$ to $F_2$ so that consecutive elements have unbounded coarse intersection. The ``if'' part follows from the previous observation, while the ``only if'' follows from the fact that elements of $\mathcal F$ corresponding to adjacent edges of some $T_\alpha$ have unbounded coarse intersection.

In view of all this and Lemma~\ref{lem:flats_to_flats}, we see that for each $\alpha$ there exists a (necessarily unique) $\phi_{\mathcal P}(\alpha)\in \mathcal{P}$ so that $\phi(\partial \mathcal B_\alpha)$ and $\partial \mathcal B_{\phi_{\mathcal P}(\alpha)}$ have finite Hausdorff distance. In fact, the distance is uniformly bounded by the constant $C$, depending only on $K$, coming from Lemma~\ref{lem:flats_to_flats}.
This is how we define $\phi_{\mathcal P}$. 

Now for any $X\in \bar D$, we may choose $\alpha_1,\alpha_2\in \mathcal P$ so that the fiber $E_X$ has Hausdorff distance at most $f(t_0)$ from the intersection of the $t_0$--neighborhoods of $\partial \mathcal B_{\alpha_i}$. Thus there is some uniform $t'_0 \ge t_0$, again depending only on $K$, so that $\phi(E_X)$ has Hausdorff distance at most $t_0'$ from the intersection of the $t_0'$--neighborhoods of $\partial \mathcal B_{\phi_{\mathcal P}(\alpha_i)}$; further, as mentioned above, this intersection of $t_0'$--neighborhoods has  Hausdorff distance at most $f(t'_0)$ to some fiber $E_Y$, as claimed.

Finally, we define $\phi_{\mathcal S}$ via the bijection $\mathcal F\leftrightarrow \mathcal S$ between strip bundles and strips in $E_0$.  That is, if  $A\in\mathcal S$ corresponds to $F\in\mathcal F$, then $\phi_{\mathcal S}(A)$ is the strip corresponding to $\phi_{\mathcal F}(F)$. 
Since $A$ is the coarse intersection of $F$ with $E_0$, the desired properties for $\phi_{\mathcal S}$ then follow from the facts that $\alpha_{\phi_{\mathcal F}(F)}=\phi_{\mathcal P}(\alpha(F))$ and that $\phi_{\mathcal F}(F)$ lies within finite Hausdorff distance of $\phi(F)$.
\end{proof}

\subsection{From $\QI(\bar E)$ to $\QI(E_0)$}
\label{sec:QI_of_barE_to_QI_of_E_0}

The next step is to construct a homomorphism $\QI(\bar E) \to \QI(E_0)$ by associating a quasi-isometry of $E_0$ to each quasi-isometry of $\bar E$ (see Lemma~\ref{L:QI to QI fiber}).
This step requires some preliminaries which we now explain.

To distinguish between two relevant notions of properness, we will call a map $f\colon X\to Y$ between metric spaces \emph{topologically} proper if it is continuous and preimages of compact sets are compact, and \emph{metrically} proper if there exist diverging functions $\rho_-,\rho_+:\mathbb R_{\geq0}\to \mathbb R_{\geq 0}$ (which we will call properness functions) such that for all $x,y\in X$ we have
\[\rho_{-}(d_X(x,y))\leq d_Y(f(x),f(y))\leq \rho_{+}(d_X(x,y)).\]
(Both types of maps are just referred to as ``proper'' in the appropriate contexts, but neither notion implies the other.)

For $R>0$ and $X\in \bar D$, we endow $N_R(E_X)$ with the restriction of the metric of $\bar E$, while $E_X$ is endowed with its path metric. Then the restriction of $f_X\colon \bar E\to E_X$ to $N_R(E_X)$ is metrically proper. Indeed, $f_X$ is topologically proper and equivariant with respect to a group acting cocompactly. Note that the properness functions here can be taken independently of the fiber $X$ (once we fix $R$) because there is also a cocompact action on $\bar D$. We also note the following lemma.

\begin{lemma}\label{lem:proper_to_qi}
 A metrically proper coarsely surjective map between geodesic metric spaces is a quasi-isometry. Moreover, the quasi-isometry constants depend only on the properness functions and the coarse surjectivity constant.
\end{lemma}

This follows from standard arguments. First, a metrically proper map from a geodesic metric space is coarsely Lipschitz (the proof involves subdividing geodesics into segments of length at most 1, each of which has bounded image). Also, coarse surjectivity allows one to construct a quasi-inverse of the map, which is furthermore metrically proper. As above, the quasi-inverse is coarsely Lipschitz, and we conclude since a coarsely Lipschitz map with a coarsely Lipschitz quasi-inverse is a quasi-isometry.

Given any quasi-isometry $\phi \colon \bar E \to \bar E$ and $X \in \bar D$, define $\nu_\phi^X \colon E_X \to E_X$ to be $\nu_\phi^X = f_X \circ \phi|_{E_X}$.  In the case of the base fiber $X = X_0$ we denote this $\nu_\phi = \nu_\phi^{X_0}$.  When $\phi$ is understood, we also write $\nu^X = \nu_\phi^X$ and $\nu = \nu_\phi$.

\begin{lemma} \label{L:fiber to fiber qi} For any $(K,K)$--quasi-isometry $\phi \colon \bar E \to \bar E$ and $X \in \bar D$, the map $\nu_\phi^X\colon E_X\to E_X$ is a $(K',K')$--quasi-isometry, where $K'$ depends only on $K$ and $d_{\mathrm{Haus}}(E_X,\phi(E_X))$.  
Furthermore, for any $A \in \mathcal S$, $d_{\mathrm{Haus}}(\nu_\phi^X(A),\phi_{\mathcal S}(A))<\infty$. 
\end{lemma}

\begin{proof}
First note that the restriction of $\phi$ to $E_X$ is metrically proper, since the path metric on $E_X$ and the restricted metric from $\bar E$ are coarsely equivalent (that is, the identity on $E_X$ is a metrically proper map between these metric spaces). Next let $R = d_{\mathrm{Haus}}(E_X,\phi(E_X))$, which is finite by Proposition~\ref{prop:strip_bijection}, and note that the restriction $f_X\vert_{N_R(E_X)} \colon N_R(E_X)\to E_X$ is also metrically proper. 
Therefore the composition $\nu_\phi^X = (f_X\vert_{N_R(E_X)})\circ(\phi\vert_{E_X})$ is metrically proper and, moreover, the properness functions depend only on $K,R$ and not on the fiber $E_X$. 

By \cite[Theorem 3.8]{KL:qi_cohopf} and the fact that $E_X$ is uniformly quasi-isometric to $\mathbb H^2$, any metrically proper map of $E_X$ to itself is coarsely surjective and, moreover, the coarse surjectivity constant depends only on the properness functions. Therefore $\nu_\phi^X$ is coarsely surjective and a uniform quasi-isometry by Lemma \ref{lem:proper_to_qi}.

Regarding the claim about $A$, this follows from Proposition \ref{prop:strip_bijection}(\ref{item:quasi-lines-trees}) and the fact that $f_X$ moves each point of $\phi(A)\subseteq N_R(E_X)$ at most $R$ away.
\end{proof}

\begin{lemma} \label{L:QI to QI fiber}
The assignment $\phi \mapsto \nu_\phi$, for any quasi-isometry $\phi \colon \bar E \to \bar E$, gives a well-defined homomorphism $\mathcal A_0\colon \QI(\bar E) \to \QI(E_0)$.
\end{lemma}
\begin{proof} Given any quasi-isometry $\phi \colon \bar E \to \bar E$ and $x \in E_0$ we have
\[ d(\phi(x),\nu_\phi(x)) = d(\phi(x),f_0(\phi(x)) \leq d_{\mathrm{Haus}}(\phi(E_0),E_0) \]
The right hand side is finite by Proposition~\ref{prop:strip_bijection}, so the left hand side is bounded, independent of $x$.
From this, the triangle inequality, and the uniform metric properness of the inclusion of $E_0$ into $\bar E$, it easily follows that if $\phi$ and $\phi'$ are bounded distance, then so are $\nu_{\phi}$ and $\nu_{\phi'}$.  Therefore the assignment $\phi\mapsto \nu_\phi$ descends to a well-defined function $\mathcal A_0 \colon \QI(\bar E) \to \QI(E_0)$.

To see that $\mathcal A_0$ is a homomorphism, suppose $\phi,\phi'$ are $(K,K)$--quasi-isometries of $\bar E$.  Then from the inequality above, for all $x \in E_0$ we have
\[ d(\phi' \circ \phi(x),\phi' \circ \nu_\phi(x)) \leq K d(\phi(x),\nu_\phi(x)) + K \leq K d_{\mathrm{Haus}}(\phi(E_0),E_0) + K.\] 
The left-hand side is thus uniformly bounded, independent of $x$.  From this, the triangle inequality, and Proposition~\ref{prop:strip_bijection}, it follows that $d_{\mathrm{Haus}}(\phi' \circ \phi(E_0),E_0)$ and $d_{\mathrm{Haus}}(\phi'\circ \nu_\phi(E_0),E_0)$ are bounded by some constant $r > 0$.  Then for all $x \in E_0$, 
\[ d(\nu_{\phi' \circ \phi}(x),\nu_{\phi'} \circ \nu_{\phi}(x)) = d(f_0(\phi' \circ \phi(x)),f_0(\phi' \circ \nu_\phi(x)))\leq e^rd(\phi' \circ \phi(x),\phi' \circ \nu_\phi(x)). \]
Combining this with the previous inequality, we see that the quantity on the right, and hence the left, is uniformly bounded above, independent of $x$. Therefore $\nu_{\phi' \circ \phi}$ and $\nu_{\phi'} \circ \nu_{\phi}$ are bounded distance apart and $\mathcal A_0$ is a homomorphism.
\end{proof}

\subsection{From quasi-isometries to affine homeomorphisms}
\label{sec:qis_to_affine}

The flat metric $q$ on $E_0$ determines an associated affine group $\Aff(E_0)$, and we observe that if $\phi \in \Gamma$ is an element of the extension group (which is an isometry of $\bar E$, and so also a quasi-isometry), we have $\nu_\phi \in \Aff(E_0)$.  The next step in the proof of rigidity is the following. 

\begin{proposition} \label{P:qi to affine} For any quasi-isometry $\phi \colon \bar E \to \bar E$, the quasi-isometry $\nu_\phi$ is uniformly close to a unique element $\nu_\phi^a \in \Aff(E_0)$.
\end{proposition}

The proof of the proposition will take place over the remainder of this subsection.  Before getting to the proof, however, we note a useful corollary.  Two quasi-isometries $\phi_1,\phi_2$ that are a bounded distance apart have $\nu_{\phi_1}$ and $\nu_{\phi_2}$ a bounded distance apart, and so by the uniqueness $\nu_{\phi_1}^a = \nu_{\phi_2}^a$.  Thus we have the following.

\begin{corollary}
\label{cor:defn_of_mathcal_A}
The map $[\phi] \mapsto \nu_\phi^a$ defines a homomorphism $\mathcal A \colon \mathrm{QI}(\bar E) \to \Aff(E_0)$. Moreover, the homomorphism $\mathcal A_0 \colon \QI(\bar E) \to \QI(E_0)$ from Lemma~\ref{L:QI to QI fiber} factors as the composition of $\mathcal A$  with the natural inclusion $\Aff(E_0) \to \mathrm{QI}(E_0)$.
\end{corollary}

Fix a triangulation $\mathfrak t$ of $X_0$ so that the vertex set is the set of cone points and all triangles are Euclidean triangles (that is, they are images of triangles by maps that are locally isometric and injective on the interior; see e.g.~\cite[Lemma~3.4]{DDLSI}).  Moreover, we assume that {\em all} saddle connections in some direction $\alpha_0$ appear as edges of the triangulation; see Figure~\ref{F:triangulation to prove affine}.  Lift $\mathfrak t$ to a triangulation $\widetilde{\mathfrak t}$ of $E_0$.  By assumption, all saddle connections in $E_0$ in direction $\alpha_0$ are edges of $\widetilde{\mathfrak t}$, and the complement of the union of this subset is a union of all (interiors of) strips in direction $\alpha_0$.

\begin{center}
\begin{figure}[htb]
\begin{tikzpicture}[scale = .7]
\draw[ultra thick, color=blue] (0,0) -- (4,0) (0,2) -- (4,2) (0,4) -- (2,4);
\draw[ultra thick, color=green!80!black] (0,2) -- (0,0) (0,0) -- (2,2) (2,0) -- (2,2) (2,2) -- (4,0) (4,0) -- (4,2) (0,2) -- (0,4) (0,4) -- (2,2) (2,2) -- (2,4);
\draw[fill] (0,0) circle (.08);
\draw[fill] (4,0) circle (.08);
\draw[fill] (4,2) circle (.08);
\draw[fill] (2,2) circle (.08);
\draw[fill] (2,4) circle (.08);
\draw[fill] (0,4) circle (.08);
\draw[fill] (0,2) circle (.08);
\draw[fill] (2,0) circle (.08);
\node[left] at (0,1) {$1$};
\node[left] at (0,3) {$2$};
\node[right] at (4,1) {$1$};
\node[right] at (2,3) {$2$};
\node[above] at (1,4) {$3$};
\node[below] at (1,0) {$3$};
\node[below] at (3,0) {$4$};
\node[above] at (3,2) {$4$};
\begin{scope}[cm={1,0,0,2,(0,-2)}]
\draw[thick, color=green!80!black] (7.2,1.95) -- (7.5,2.1) -- (8,1.85) -- (8.7,2.2) -- (9.5,1.75) -- (10.5,2.25) -- (11.3,1.8) -- (12,2.15) -- (12.5,1.9) -- (12.8,2.05) (8.7,1.4) -- (9.5,1.35) (8.7,2.85) -- (9.5,2.25) (9.5,2.65) -- (10.5,2.25) (10.5,2.65) -- (11.3,2.6) (9.5,1.75) -- (10.5,1.35) (10.5,1.75) -- (11.3,1.15);
\draw[thick, color=green!80!black] (7.2,2.05) -- (7.2,1.95) (7.5,2.1) -- (7.5,1.9) (8,2.15) -- (8,1.85) (8.7,2.2) -- (8.7,1.8) (9.5,2.65) -- (9.5,1.4) (10.5,2.65) -- (10.5,1.4) (11.3,2.2) -- (11.3,1.8) (12,2.15) -- (12,1.85) (12.5,2.1) -- (12.5,1.9) (12.8,2.05) -- (12.8,1.95);
\draw[thick, color=blue] (7,2) -- (7.2,1.95) -- (7.5,1.9) -- (8,1.85) -- (8.7,1.8) -- (9.5,1.75) -- (10.5,1.75) -- (11.3,1.8) -- (12,1.85) -- (12.5,1.9) -- (12.8,1.95) -- (13,2) (9.5,1.75) -- (8.7,1.4) (10.5,1.75) -- (11.3,1.4) (8.7,1.15) -- (9.5,1.35) -- (10.5,1.35) -- (11.3,1.15);
\draw[fill] (7.2,1.95) circle (.02);
\draw[fill] (7.5,1.9) circle (.03);
\draw[fill] (8,1.85) circle (.04);
\draw[fill] (8.7,1.8) circle (.05);
\draw[fill] (9.5,1.75) circle (.05);
\draw[fill] (10.5,1.75) circle (.05);
\draw[fill] (11.3,1.8) circle (.05);
\draw[fill] (12,1.85) circle (.04);
\draw[fill] (12.5,1.9) circle (.03);
\draw[fill] (12.8,1.95) circle (.02);
\draw[fill] (9.5,1.35) circle (.05);
\draw[fill] (10.5,1.35) circle (.05);
\draw[thick, color=blue] (7,2) -- (7.2,2.05) -- (7.5,2.1) -- (8,2.15) -- (8.7,2.2) -- (9.5,2.25) -- (10.5,2.25) -- (11.3,2.2) -- (12,2.15) -- (12.5,2.1) -- (12.8,2.05) -- (13,2) (9.5,2.25) -- (8.7,2.6)  (10.5,2.25) -- (11.3,2.6) (8.7,2.85) -- (9.5,2.65) -- (10.5,2.65) -- (11.3,2.85);
;
\draw[fill] (7.2,2.05) circle (.02);
\draw[fill] (7.5,2.1) circle (.03);
\draw[fill] (8,2.15) circle (.04);
\draw[fill] (8.7,2.2) circle (.05);
\draw[fill] (9.5,2.25) circle (.05);
\draw[fill] (10.5,2.25) circle (.05);
\draw[fill] (11.3,2.2) circle (.05);
\draw[fill] (12,2.15) circle (.04);
\draw[fill] (12.5,2.1) circle (.03);
\draw[fill] (12.8,2.05) circle (.02);
\draw[fill] (9.5,2.65) circle (.05);
\draw[fill] (10.5,2.65) circle (.05);
\end{scope}
\draw[thin,opacity=.3] (10,2) circle (3);
\end{tikzpicture}
\caption{The surface from Figure \ref{F:twisting and retraction} with a triangulation $\mathfrak t$ on the left.  The edges of $\mathfrak t$ consist of all saddle connections in the horizontal direction (blue) together with some saddle connections (green), each contained in horizontal cylinder.  A piece of the lifted triangulation $\tilde{\mathfrak t}$ in the universal cover.} \label{F:triangulation to prove affine}
\end{figure}
\end{center}

\begin{lemma} \label{L:affine on triangles} Given a quasi-isometry $\phi$, there is a biLipschitz homeomorphism $\nu_\phi^a \colon E_0 \to E_0$ a bounded distance from $\nu_\phi$ so that $\nu_\phi^a$ restricts to an affine map on each triangle of $\widetilde{\mathfrak t}$.  Moreover, if an edge $\delta$ of $\widetilde{\mathfrak t}$ has direction $\alpha \in \mathcal P$, then $\nu_\phi^a(\delta)$ has direction $\phi_{\CP}(\alpha)$.
\end{lemma}
We will later prove that  $\nu_\phi^a$ is in fact {\em globally} affine, justifying the notation.
\begin{proof}
Given $\nu=\nu_\phi \colon E_0 \to E_0$, let $\partial \nu \colon S^1_\infty \to S^1_\infty$ be the restriction of the extension to the Gromov boundary $S^1_\infty$ of $E_0$.  
(Recall that, since the flat metric $(X_0,q)$ on $S$ is biLipschitz to a hyperbolic metric, its universal cover $E_0$ is quasi-isometric to the standard hyperbolic plane and is, in particular, Gromov hyperbolic.)
The space $\mathcal G$ of (unordered) pairs of distinct points in $S^1_\infty$ is precisely the space of endpoint-pairs at infinity of unoriented biinfinite geodesics (up to the equivalence relation of having finite Hausdorff distance).   The map $\partial \nu$ induces a map $\partial \nu_* \colon \mathcal G \to \mathcal G$.

Let $\mathcal G^* \subset \mathcal G$ be the closure of the set of endpoint-pairs at infinity of non-singular geodesics (i.e.~geodesics that miss every cone point).  Observe that all geodesics in a given strip have the same pair of endpoints, and any geodesic with that pair of endpoints is contained in the strip.  Given a strip, we are therefore justified in referring to the {\em pair of endpoints of the strip}.  

It follows from the description of geodesics with endpoints in $\mathcal G^*$ (see \cite[Proposition~2.4]{BankLein}) together with the {\em Veech Dichotomy} (see e.g.~\cite{Masur-Tabachnikov}), that for any $\{\xi,\zeta\} \in \mathcal G^*$, either $\{\xi,\zeta\}$ are the endpoints of a strip, or endpoints of a geodesic meeting at most one cone point.

According to Proposition~\ref{prop:strip_bijection}, for any strip $A \in \mathcal S$, the strip $\phi_{\mathcal S}(A)$ has finite Hausdorff distance to $\phi(A)$, and hence it also has finite Hausdorff distance to $\nu(A)$.  Since $\phi_{\mathcal S}$ is a bijection, this means that the homeomorphism $\partial \nu_*$ sends the dense subset of $\mathcal G^*$ consisting of endpoint of strips onto itself, hence $\nu_*(\mathcal G^*) = \mathcal G^*$.
From this and \cite[Proposition 4.1]{BankLein} (see also \cite[Proposition~11]{DELS}), it follows that there is a bijection
\[ \phi_{\Sigma_0} \colon \Sigma_0 \to \Sigma_0 \] 
from the set of cone points $\Sigma_0$ of $E_0$ to itself with the following property.  If $\gamma \subset E_0$ is a geodesic or strip containing $x \in \Sigma_0$ with endpoints $\{ \xi,\zeta \} \in \mathcal G^*$, then $\partial \nu_*(\{\xi,\zeta\})$ are the endpoints of a geodesic containing $\phi_{\Sigma_0}(x)$.   Given $x \in \Sigma_0$ consider any two geodesics $\gamma_1$ and $\gamma_2$ with endpoints in $\mathcal G^*$ (not necessarily contained in strips) passing through $x$ making an angle at least $\pi/2$ with each other.  We note that $\nu(x)$ is contained in $\nu(\gamma_1)$ and $\nu(\gamma_2)$, and is thus some uniform distance $r > 0$ to both of their geodesic representatives.  Since $\gamma_1$ and $\gamma_2$ meet at angle at least $\pi/2$, the $r$--neighborhoods of the geodesic representatives of $\nu(\gamma_1)$ and $\nu(\gamma_2)$ intersect in a uniformly bounded diameter set, which contains $\phi_{\Sigma_0}(x)$.  Therefore, $\phi_{\Sigma_0}(x)$ is uniformly close to $\nu(x)$, for all $x \in \Sigma_0$.

From the properties of $\phi_{\Sigma_0}$ described above, we see that if $x \in \Sigma_0$ is contained in a strip $A$, then $\phi_{\Sigma_0}(x)$ is contained in the strip $\phi_{\mathcal S}(A)$.
For any saddle connection $\delta$ in some direction $\alpha \in \CP$ between a pair of points $x,y \in \Sigma_0$, there is a unique pair of strips $A,A_0$, also in direction $\alpha$, that contain $\delta$; see Figure~\ref{F:strips, saddle connections, and orders}.  Since $\phi_{\Sigma_0}(x),\phi_{\Sigma_0}(y)$ are contained in $\phi_{\mathcal S}(A)$ and $\phi_{\mathcal S}(A_0)$, it follows that there is a unique saddle connection with endpoints $\phi_{\Sigma_0}(x),\phi_{\Sigma_0}(y)$.  For any strip $A$ the saddle connections whose union makes up one of its boundary components is determined by a collection of strips meeting $A$ in the given saddle connections.  These saddle connections are ordered along this side and thus so are the corresponding strips $\ldots,A_{-1},A_0,A_1,\ldots$.  The endpoints of the strip $A$ and strips $A_n$ appear in a particular order; see Figure~\ref{F:strips, saddle connections, and orders}.   Considering the cyclic ordering of the endpoints of these strips (and those of $A$) on $S^1_\infty$, and the fact that $\partial \nu$ is a homeomorphism,  it follows that  $\phi_{\Sigma_0}$ maps the {\em ordered} set of cone points along each boundary component of the strip $A$ by an order preserving (or reversing) bijection to the ordered set of cone points along the boundary components of $\phi_{\mathcal S}(A)$.

\begin{figure}[htb]
\begin{tikzpicture}[scale = .9]
\filldraw[thin, color=green, opacity=.4] (7,2) -- (7.2,2.1) -- (7.5,2.25) -- (8,2.4) -- (8.7,2.5) -- (9.5,2.6) -- (10.5,2.6) -- (11.3,2.5) -- (12,2.4) -- (12.5,2.25) -- (12.8,2.1) -- (13,2) -- (7,2);
\filldraw[thin, color=blue, opacity=.3] (10.5,2) -- (10.7,1.5) -- (11.15,0) -- (11.3,-.7) -- (10.7,.7) -- (10.4,1.3) -- (9.6,1.3) -- (8.7,-.7) -- (8.85,0) -- (9.3,1.5) -- (9.5,2) -- (10.5,2);
\filldraw[thin, color=red, opacity=.3] (12.5,.32) -- (11.2,1.6) -- (10.9,1.6) -- (11.8,-.4) -- (10.7,1.5) -- (10.5,2) -- (11.3,2) -- (12.5,.32);
\filldraw[thin, color=yellow, opacity=.3] (7.5,.32) -- (8.8,1.6) -- (9.1,1.6) -- (8.2,-.4) -- (9.5,2) -- (8.7,2) -- (7.5,.32);
\draw (7,2) -- (7.2,2.1) -- (7.5,2.25) -- (8,2.4) -- (8.7,2.5) -- (9.5,2.6) -- (10.5,2.6) -- (11.3,2.5) -- (12,2.4) -- (12.5,2.25) -- (12.8,2.1) -- (13,2);
\draw (7,2) -- (13,2);
\draw (9.5,2) -- (9.3,1.5) -- (8.85,0) -- (8.7,-.7);
\draw (10.5,2) -- (10.7,1.5) -- (11.15,0) -- (11.3,-.7);
\draw (8.7,-.7) -- (9.3,.7) -- (9.6,1.3) -- (10.4,1.3) -- (10.7,.7) -- (11.3,-.7);
\draw (9.5,2) -- (8.2,-.4);
\draw (8.7,2) -- (7.5,.32);
\draw (7.5,.32) -- (8.8,1.6) -- (9.1,1.6) -- (8.2,-.4);
\draw (12.5,.32) -- (11.2,1.6) -- (10.9,1.6) -- (11.8,-.4);
\draw (11.8,-.4) -- (10.7,1.5);
\draw (12.5,.32) -- (11.3,2);
\node at (10.15,1.55) {$A_0$};
\node at (10,2.4) {$A$};
\node at (9,1.8) {\small $A_{\, \,1}$};
\draw[thin] (9.03,1.73) -- (9.1,1.73);
\node at (11,1.8) {\small $A_1$};
\node at (9.9,1.85) {\small $\delta$};
\node at (9.5,2.2) {\small $x$};
\node at (10.5,2.2) {\small $y$};
\draw[fill] (9.5,2) circle (.04);
\draw[fill] (10.5,2) circle (.04);
\draw[fill] (8.7,2) circle (.03);
\draw[fill] (11.3,2) circle (.03);
\draw[thin,opacity=.3] (10,2) circle (3);
\end{tikzpicture}
\caption{Strips $A$ and $A_0$ determine the saddle connection $\delta$ connecting points $x,y$ in the universal cover.  The ordered set of saddle connections along one side of the strip $A$ are determined by an ordered set of strips $\ldots,A_{-1},A_0,A_1,\ldots$, and the endpoints of all of these strips have a cyclic ordering around $S^1_\infty$ as indicated.} \label{F:strips, saddle connections, and orders}
\end{figure}

We can now extend the map $\phi_{\Sigma_0}$ to a map $\nu_\phi^a \colon E_0 \to E_0$ using $\widetilde{\mathfrak t}$ as follows.  First, recall that any edge of $\widetilde{\mathfrak t}$ is a saddle connection $\delta$ connecting two points $x,y \in \Sigma_0$.  By the previous paragraph, there is a saddle connection $\delta'$ connecting $\phi_{\Sigma_0}(x)$ and $\phi_{\Sigma_0}(y)$, and we define $\nu_\phi^a$ on $\delta$ so that it maps $\delta$ by an affine map to $\delta'$ extending $\phi_{\Sigma_0}$ on the endpoints.  This defines $\nu_\phi^a$ on the $1$--skeleton, $\widetilde{\mathfrak t}^1$, and since $\phi_{\Sigma_0}$ is a bounded distance from $\nu|_{\Sigma_0}$, it follows that $\nu_\phi^a|_{\widetilde{\mathfrak t}^1}$ is a bounded distance from $\nu|_{\widetilde{\mathfrak t}^1}$.

By our assumptions on $\widetilde{\mathfrak t}$, there is a subset of the edges of $\widetilde{\mathfrak t}$ whose union is precisely the union of boundaries of all strips in direction $\alpha_0$.  The order preserving (or reversing) property described above for the cone points along the boundary of a strip, together with Proposition~\ref{prop:strip_bijection}, implies that for any boundary component of any strip $A$ in direction $\alpha_0$, $\nu_\phi^a$ restricted to its boundary components is a homeomorphism onto the boundary components of $\phi_{\mathcal S}(A)$.  Furthermore, since the sides of any triangle of $\widetilde{\mathfrak t}$ are contained in such a strip $A$, the $\nu_\phi^a$--image of the sides are contained in $\phi_{\mathcal S}(A)$.
We can now extend $\nu_\phi^a$ over the triangles by the unique affine map extending the map on their sides.  

Since disjoint strips map to disjoint strips, the map $\nu_\phi^a$ is a homeomorphism.  By construction, any edge in direction $\alpha$ is sent to an edge in direction $\phi_{\CP}(\alpha)$.  Since $\widetilde{\mathfrak t}$ projects to $\mathfrak t$, there are only finitely many directions that the sides of a triangle can lie in and so finitely many isometry types of triangles.  Each of these finitely many isometry types maps by an affine map to only finitely many types of triangles in the image (because the direction of the images of sides are determined by $\phi_{\CP}$), and therefore these affine maps are uniformly biLipschitz.   Therefore, $\nu_\phi^a$ is biLipschitz, completing the proof.
\end{proof}

To show that $\nu_\phi^a$ is affine, we analyze the effect of using it to conjugate the action of $\pi_1S$ on $E_0$.
\begin{lemma} \label{L:conj isom action} The action of $\pi_1S$ on $E_0$ obtained by conjugating the isometric action by $\nu_\phi^a$ is again an isometric action.
\end{lemma} 

Before proving the lemma, we use it to prove the proposition.

\begin{proof}[Proof of Proposition~\ref{P:qi to affine}] By Lemma~\ref{L:conj isom action}, $\Lambda = \nu_\phi^a \pi_1S (\nu_\phi^a)^{-1}$ acts by isometries, and $\nu_\phi^a$ descends to a homeomorphism $\mu_\phi^a \colon S\to E_0/\Lambda$ and is biLipschitz with respect to descent to $S$ and $E_0/\Lambda$ of $q$.  Since $\nu_\phi^a$ and $\nu$ are a bounded distance, they have the same boundary maps.  Since $\partial (\nu_\phi^a)_* = \partial \nu_*$ maps $\mathcal G^*$ to $\mathcal G^*$, the {\em Current Support Theorem} of \cite{DELS}  (and its proof) implies that the descent of $\mu_\phi^a \colon (S,q) \to (E_0/\Lambda,q)$ is affine. Therefore $\nu_\phi^a$ is an affine map which is a bounded distance from $\nu = \nu_\phi$, as required.

Uniqueness follows from the fact that no two distinct affine maps are a bounded distance apart.
\end{proof}

\begin{proof}[Proof of Lemma~\ref{L:conj isom action}] We need to show that for all $g \in \pi_1S$, the map
\[ \nu_\phi^a \circ g \circ (\nu_\phi^a)^{-1} \colon E_0 \to E_0\]
is an isometry.  For this, fix a triangle $\tau$ of $\widetilde{\mathfrak t}$ and consider the restriction to $\nu_\phi^a(\tau)$.  Let $\alpha_1,\alpha_2,\alpha_3 \in \CP$ be the directions of the sides.  Setting $\alpha_i' = \phi_{\CP}(\alpha_i)$, for $i = 1,2,3$, Lemma~\ref{L:affine on triangles} implies that the directions of the sides of $\nu_\phi^a(\tau)$ are $\alpha_1',\alpha_2',\alpha_3'$.  
The action of $\pi_1S$ on $E_0$ is by isometries, but it also preserves parallelism (i.e.~each element induces the identity on $\mathbb P^1(q)$).  Therefore, for any $g \in \pi_1S$, the directions of the sides of $g(\tau)$ are also $\alpha_1,\alpha_2,\alpha_3$, and by Lemma~\ref{L:affine on triangles} again, it follows that the sides of $\nu_\phi^a(g(\tau))$ are $\alpha_1',\alpha_2',\alpha_3'$.

For any $g \in \pi_1S$, since $\nu_\phi^a$ is affine on $\tau$, the composition $\nu_\phi^a \circ g \circ (\nu_\phi^a)^{-1}$ is also affine on $\nu_\phi^a(\tau)$.  On the other hand, it also preserves the directions of the sides, $\alpha_1',\alpha_2',\alpha_3'$.  Therefore, the restriction of $\nu_\phi^a \circ g \circ (\nu_\phi^a)^{-1}$ is a Euclidean similarity.   Triangles of $\nu_\phi^a(\widetilde{\mathfrak t})$ that share a side are scaled by the same factor by the similarity $\nu_\phi^a \circ g \circ (\nu_\phi^a)^{-1}$ in each triangle (since this is the scaling factor on the shared side).  Therefore, the similarities agree along edges, and hence $\nu_\phi^a\circ  g\circ (h_\phi^a)^{-1}$ defines a {\em global} similarity of $E_0$.

So, the action of $\pi_1S$ on $E_0$ obtained by conjugating by $\nu_\phi^a$ is an action by similarities.  To see that the action is by isometries, suppose that for some element $g \in \pi_1S$, the similarity $g_0 = \nu_\phi^a \circ g \circ (\nu_\phi^a)^{-1}$ scales the metric some number $\lambda \neq 1$.  Taking the inverse if necessary, we can assume $\lambda < 1$.  Fix any $x \in E_0$ and observe that $d_q(g_0(x),g_0^2(x)) = \lambda d_q(x,g_0(x))$, where $d_q$ is the distance function on $E_0$ determined by $q$.  Iterating this, it follows that
\[ d_q(x,g_0^n(x)) \leq \! \sum_{k=1}^n d_q(g_0^{k-1}(x),g_0^k(x)) \! = \!\! \sum_{k=1}^n \lambda^{k-1} d_q(x,g_0(x)) \leq \! d_q(x,g_0(x))\sum_{k=1}^\infty \lambda^{k-1}  .  \]
Since the right-hand side is a convergent geometric series, it follows that $\{g_0^n(x)\}_{n=1}^\infty$ is a Cauchy sequence.  On the other hand, this sequence exits every compact set (since $g_0$ is an infinite order element of $\pi_1S$), and since $q$ is a complete metric on $E_0$, thus we obtain a contradiction.  Therefore, the conjugation action of $\pi_1S$ is by isometries.
\end{proof}

\subsection{Injectivity of $\mathcal A$.} 
\label{sec:injectivity_of_A}
Our next goal is to prove that $\mathcal A \colon \QI(\bar E) \to \Aff(E_0)$, the homomorphism from Corollary~\ref{cor:defn_of_mathcal_A}, is injective.

In preparation, it will be useful to have the following general fact about quasiisometries of hyperbolic spaces, whose proof we sketch for convenience of the reader:

\begin{lemma} \label{lem:hyper fact}
 For each $K,C,\delta$ there exists $R$ so that the following holds. Suppose that $Z$ is $\delta$--hyperbolic and that each $z\in Z$ lies within $\delta$ of all three sides of a nondegenerate ideal geodesic triangle. Let $f\colon Z\to Z$ be a $(K,C)$-quasi-isometry that lies within finite distance of the identity. Then $f$ lies within distance $R$ of the identity.
\end{lemma}

\begin{proof}
Since $f$ is within bounded distance of the identity, its extension $\partial f\colon \partial Z \to \partial Z$ is the identity.  Hence if $z \in Z$ and $\Delta$ is an ideal geodesic triangle as in the statement, then $f(\Delta)$ is a $(K,C)$-quasigeodesic ideal triangle with the same endpoints as $\Delta$.  By the Morse lemma, there is a constant $\kappa = \kappa(K,C,\delta) > 0$ such that $f(z)$ lies within $\kappa$ of the three quasi-geodesic sides of $f(\Delta)$, and these sides in turn lie  within $\kappa$ of the sides of $\Delta$.  
Thus $z$ and $f(z)$ both lie within $2\kappa+\delta$ of all three sides of $\Delta$.
Since the set of points within $2\kappa+\delta$ of all three sides of a nondegenerate geodesic triangle in a $\delta$--hyperbolic space has diameter bounded in terms of $\kappa$ and $\delta$,  we see that $d_Z(z,f(z))$ is bounded solely in terms of $\delta, K,C$, as required.
\end{proof}

With this fact in hand, we can now prove:

\begin{proposition}\label{prop:identity_on_fibers}
Let $\phi\colon \bar E\to \bar E$ be a quasi-isometry with $\nu_\phi^a = id_{E_0}$.
There is a constant $C' = C'(\phi)>0$ such that for all $x \in \bar E$, $d(x,\phi(x)) \leq C'$.  Consequently, $\mathcal A \colon \QI(\bar E) \to \Aff(E_0)$ is injective.
\end{proposition}

\begin{proof}

We first claim that for any $X \in \bar D$, $\phi(E_X)$ lies within the $C''$--neighborhood of $E_X$, where $C'' = C''(\phi)>0$.  To see this, observe that since $\nu_\phi^a$ is the identity and is bounded distance from $\nu_\phi$, it follows that $\phi|_{E_0}$ is within bounded distance of the inclusion of $E_0$ in $\bar E$. Proposition \ref{prop:strip_bijection} then implies that $d_{\mathrm{Haus}}(A,\phi_S(A))<+\infty$ for each strip $A\in\mathcal S$. Since strips that lie within finite Hausdorff distance coincide, we have $\phi_S(A)=A$.  Combining this fact with Proposition \ref{prop:strip_bijection} it follows that  $\phi_{\mathcal P}(\alpha(A))=\alpha(\phi_{\mathcal S}(A))=\alpha(A)$. 
Hence for each $\alpha$, we have that $\phi(\partial \mathcal B_\alpha)$ lies within Hausdorff distance $C$ of $\partial \mathcal B_{\alpha}=\partial \mathcal B_{\phi_{\mathcal P}(\alpha)}$, for $C$ as in Proposition \ref{prop:strip_bijection}.

Now let $X \in \bar D$ be any point and choose distinct $\alpha, \alpha' \in \mathcal P$ so that $X$ is contained in the coarse intersection of $\partial B_{\alpha}$ and $\partial B_{\alpha'}$, implying that $E_X$ lies in the coarse intersection of $\partial \mathcal B_{\alpha}$ and $\partial \mathcal B_{\alpha'}$.  By the coarse preservation of the $\partial \mathcal B_{\alpha}$ in the previous paragraph, the coarse intersection of $\phi(\partial \mathcal B_{\alpha})$ and $\phi(\partial \mathcal B_{\alpha'})$ is within Hausdorff distance $C$ of the coarse intersection of $\partial \mathcal B_{\alpha}$ and $\partial \mathcal B_{\alpha'}$, and hence $E_X$ and $\phi(E_X)$ are within uniform Hausdorff distance.  This proves the claim.

Since $f_X\colon\bar E\to E_X$ is $e^{C''}$-bi-Lipschitz when restricted to fibers in the $C''$-neighborhood of $E_X$, the claim implies that $\nu_\phi^X = f_X\circ \phi\colon E_X \to E_X$ is a quasi-isometry  with constants depending only on $\phi$ and not $X$.  Moreover, since each $E_X$ lies within finite (but not necessarily bounded) Hausdorff distance of $E_0$, the fact that $\phi|_{E_0}$ lies within finite distance of the inclusion $E_0 \hookrightarrow \bar{E}$ implies that $\nu_\phi^X$ lies within finite distance of the identity $E_X \to E_X$.  Since each $E_X$ is uniformly quasiisometric to $\mathbb H^2$, it follows that $\nu_\phi^X\colon E_X \to E_X$ satisfies the assumptions of Lemma \ref{lem:hyper fact}.  We conclude that $\nu_\phi^X$ is within uniformly bounded distance of the identity for each $X\in \bar D$.  Since $d(\nu_\phi^X(x),\phi(x)) \leq C''$, it follows that $d(x,\phi(x))$ is uniformly bounded, independent of $x$. This proves the first statement of the proposition.

If $\mathcal A(\phi)$ is the identity for some $\phi \in \QI(\bar E)$, then by the first part of the proposition, $\phi$ is a bounded distance from the identity.  Therefore, $\phi$ and the identity represent the same class, and $\mathcal A$ is injective.  This completes the proof.
\end{proof}

\subsection{From affine homeomorphisms to isometries}
\label{sec:affine_to_isometry}

Next we will choose a particular allowable truncation and construct a homomorphism $\Aff(E_0) \to \Isom_{\fib}(\bar E)$, that we will eventually show is an isomorphism. We first construct such a homomorphism to the fiber-preserving isometry group {\em of the space $E$}, which avoids the issue of choosing the truncation.

\begin{lemma}
\label{lem:affine_to_isometry_of_E}
For any $\nu \in \Aff(E_0)$, there exists a isometry $\phi = \phi_\nu \in \Isomfib(E)$ such that $f_0 \circ \phi_\nu\vert_{E_0} = \nu$.
Moreover, this assignment $\nu \mapsto \phi_\nu$ defines an injective homomorphism $\Aff(E_0) \to \Isomfib(E)$.
\end{lemma}
\begin{proof}

Recall from \S\ref{S:flat metrics Veech groups} that the projective tangent space at any non-cone point of $E_0$ is denoted $\mathbb P^1(q)$ and is canonically identified with $\partial D$. The derivative of $\nu\colon E_0\to E_0$ (which may reverse orientations) is a well-defined projective transformation $d\nu\in\PGL(\mathbb P^1(q))$ which, using the preferred coordinates on $q = q_0$ with distinguished vertical and horizontal directions, we canonically identify with $\PGL_2(\mathbb R)$. The Teichm\"uller disk $D$ is the orbit of $q$ under the $\SL_2(\mathbb R)$ action and is identified with $\mathbb H^2 = \SL_2(\mathbb R)/SO(2)$ (see e.g.~\cite[\S2.8]{DDLSI}). As $\PGL_2(\mathbb R)$ acts isometrically on $\mathbb H^2$, we thus obtain an isometry $\Phi= d\nu\colon D\to D$ whose induced map $\partial \Phi$ of the circle at infinity $\partial D$ agrees with the derivative $d \nu$ under the canonical identification $\partial D\cong \mathbb P^1(q)$. In particular, setting $X = \Phi(X_0)$, the geodesic ray in $D$ emanating from $X_0$ and asymptotic to $\xi\in \mathbb P^1(q)$ is sent to the geodesic ray emanating from $X$ asymptotic to $d\nu(\xi)$. 

We claim the map $\phi_0 = f_{X, X_0}\circ \nu\colon E_0 \to E_X$ is an isometry of fibers. Indeed, any pair $\xi,\xi^\perp\in \mathbb P^1(q)$ of orthogonal directions on $E_0$ are the endpoints of a geodesic $\rho$ in $D$ containing $X_0$. 
Since $\Phi$ is an isometry with $\partial \Phi = d \nu$, we have that $X = \Phi(X_0)$ lies on the geodesic $\Phi(\rho)$ from $d\nu(\xi^\perp)$ to $d\nu(\xi)$; that is, $d\nu(\xi),d\nu(\xi^\perp)$ are orthogonal on $X$.
But since $\mathbb P^1(q)$ and $\mathbb P^1(q_X)$ are canonically identified by the Teichm\"uller map $f_{X, X_0}$ (see e.g.~\cite[\S2.8]{DDLSI}), this means $\phi_0$ is an affine map whose derivative $d \phi_0 = d\nu$ preserves orthogonality of lines; hence $\phi_0$ is an isometry as claimed. 

Now we define $\phi = \phi_\nu \colon E \to E$ by the formula:
\[ \phi(x) = f_{\Phi(\pi(x)),X_0} \circ \nu \circ f_0(x). \]
In words, this maps the fiber over a point $Y$ to the fiber  over $\Phi(Y)$, and the horizontal disk $D_x$, for $x \in E_0$, to $D_{\nu(x)}$.  The restriction $\phi|_{D_x} \colon D_x \to D_{\nu(x)}$ is an isometry since it covers $\Phi$.  To prove that $\phi$ is an isometry, it therefore suffices to show that $\phi|_{E_Y} \colon E_Y \to E_{\Phi(Y)}$ is an isometry for any $Y \in D$.

Fix any $Y\in D$. For $Y = X_0$, we have already seen that $\phi\vert_{E_0}$ is the isometry $\phi_0\colon E_0\to E_X$. If $Y\ne X$, there exist unique orthogonal directions $\alpha,\alpha^\perp \in \mathbb P(q)$ and $t > 0$, so that $X_0$ and $Y$ both lie on the the geodesic from $\alpha^\perp$ to $\alpha$ in $D$ and  $Y$ lies distance $t$ from $X_0$ in the direction of $\alpha$. This means that $f_{Y,X_0} \colon E_0 \to E_Y$ contracts in direction $\alpha$ by $e^{-t}$ and stretches in direction $\alpha^\perp$ by $e^t$. The image $\Phi(Y)$ lies along the geodesic from $d\phi_0(\alpha^\perp)$ to $d\phi_0(\alpha)$ at distance $t$ from $\Phi(X_0) = X$;  therefore $f_{\Phi(Y),X} \colon E_X \to E_{\Phi(Y)}$ contracts by $e^{-t}$ in direction $d\phi_0(\alpha)$ and stretches by $e^t$ in direction $d \phi_0(\alpha^\perp)$.  The restriction $\phi|_Y \colon Y \to \Phi(Y)$ is given by $f_{\Phi(Y),X} \circ \phi_0 \circ f_{E_0,Y}$. Since $\phi_0$ sends $\alpha\mapsto d\phi_0(\alpha)$ and $\alpha^\perp\mapsto d\phi_0(\alpha^\perp)$,  the description above shows that $\phi|_Y$ is an isometry.  Therefore $\phi$ is an isometry, as required.

To see that $\nu\mapsto \phi_\nu$ is a homomorphism, note that by construction $\Phi_\nu$ is the unique isometry of $D$ for which $\partial \Phi_\nu = d \nu$. Thus the chain rule implies  $\Phi_{\nu\circ g} =\Phi_\nu\circ \Phi_g$ is the unique isometry whose action on $\partial D$ agrees with $d(\nu\circ g) = d\nu \circ dg$. For any $x\in E$ we have $\pi(\phi_g(x))= \Phi_g(\pi(x))$ and hence by construction
\begin{align*}
\phi_{\nu}\circ \phi_g (x)&= f_{\Phi_\nu(\pi(\phi_g(x))),X_0}\circ \nu \circ f_0 \big(f_{\Phi_g(x),X_0}\circ g \circ f_0(x)\big)\\
&= f_{\Phi_\nu(\Phi_g(\pi(x))),X_0}\circ \nu \circ f_{X_0,\Phi_g(x)}\circ f_{\Phi_g(x),X_0}\circ g \circ f_0(x))\\
&= f_{\Phi_{\nu\circ g}(\pi(x)),X_0}\circ (\nu \circ g) \circ f_0(x)
= \phi_{\nu\circ g}(x)
\end{align*}
as needed. Finally, if $\phi_\nu = \mathrm{id}_E$ then clearly $X = \Phi(X_0) = X_0$. Since  $\phi_0 =  \phi_\nu\vert_{E_0}$ by construction, we conclude that 
\[\mathrm{id}_{E_0} = \phi_\nu\vert_{E_0} = \phi_0 = f_{X_0, X_0}\circ \nu = \nu.\]
Hence $\nu$ is the identity affine map, showing that $\nu\mapsto \phi_\nu$ is injective.
\end{proof}

\begin{lemma}
\label{lem:Gamma_finite_index}
The subgroup $\Gamma < \Isom_{\fib}(E)$ has finite index.
\end{lemma}
\begin{proof} 

By \cite[Proposition~5.5]{DDLSI}, $\Isomfib(E)$ acts properly discontinuously on $E$. Therefore $E / \Isom_\fib(E)$ is a topological orbifold with well-defined, positive Riemannian volume. The index of $\Gamma$ in $\Isom_\fib(E)$ is  the degree of the orbifold cover $E/\Gamma\to E/\Isom_\fib(E)$ and equals the ratio of the respective volumes. As $E/\Gamma$ has finite volume, since the quotient $D/G$ has finite area and the fibers $E_X / \pi_1 S$ all have equal, finite area, we conclude that $\Gamma$ indeed has finite index.
\end{proof}

\begin{lemma}
\label{lem:allowable_trunc}
There is an allowable truncation $\bar E$  that is $\Isom_\fib(E)$--invariant and for which restricting to $\bar E$ induces an injection $\Isom_\fib(E) \to \Isom_\fib(\bar E)$.
\end{lemma}
\begin{remark}
Every fiber-preserving isometry of $\bar E$ uniquely extends to one of $E$ (e.g.~by following the proof of Lemma~\ref{lem:affine_to_isometry_of_E}) and thus $\Isom_\fib(E)\to \Isom_\fib(\bar E)$ is in fact an isomorphism.
\end{remark}

\begin{proof}
There is a natural map $\Isom_\fib(E)\to \Isom(D)$ that sends $\Gamma$ onto $G$. Hence, by the previous lemma, the image $G^*$ of $\Isom_\fib(E)$ under this map contains $G$ with finite index. Therefore $G^*$ acts properly discontinuously on $D$ and we may choose a collection $\{B_\alpha\}_{\alpha\in \mathcal{P}}$ of $1$--separated horoballs as in \S\ref{S:flat metrics Veech groups} that is $G^*$--invariant. If $\bar{E}$ denotes the corresponding truncation of $E$, it follows that every element of $\Isom_\fib(E)$ preserves $\bar{E}$. The map $\Isomfib(E)\to \Isomfib(\bar E)$ given by restricting $\phi\mapsto \phi\vert_{\bar E}$ is injective by \cite[Corollary~5.6]{DDLSI} since if $\phi\vert_{\bar E}$ is the identity, then $\phi$ must be the identity on each Teichm\"uller disk $D_x$ and fiber $E_X\subset \bar E$.
\end{proof}

Choosing such an allowable truncation $\bar{E}$, Lemmas \ref{lem:affine_to_isometry_of_E} and \ref{lem:allowable_trunc} now give an injective homomorphism $\Psi\colon \Aff(E_0)\to \Isom_\fib(\bar E)$ given by $\Psi(\nu) = \phi_\nu\vert_{\bar E}$.

\begin{lemma}
\label{lem:A_is_the_inverse_of_Psi}
For any $\nu \in \Aff(E_0)$, $\mathcal A(\Psi(\nu)) = \nu$, where we have identified $\Psi(\nu)$ with its image in $\QI(\bar E)$ from the homomorphism $\Isom_\fib(\bar E) \to \QI(\bar E)$.
\end{lemma}
\begin{proof} 
The construction of $\mathcal A$ in Corollary~\ref{cor:defn_of_mathcal_A}
sends the (quasi-)isometry $\Psi(\nu) = \phi_\nu\vert_{\bar E}\colon\bar{E}\to \bar{E}$ to the the unique affine homeomorphism of $E_0$ that is uniformly close to the map $f_0\circ \phi_\nu\vert_{E_0}\colon E_0\to E_0$. 
But by the construction of $\phi_\nu$ in Lemma~\ref{lem:affine_to_isometry_of_E}, $f_0\circ \phi_\nu\vert_{E_0}$ is affine itself and equal to $\nu$. Thus evidently $\mathcal{A}(\Psi(\nu)) = \nu$ as claimed.
\end{proof}

\begin{lemma}
\label{lem:Psi_is_the_inverse_of_A}
For any $\phi\in \Isom_\fib(\bar E) = \Isom_\fib(E)$, we have $\Psi(\mathcal{A}(\phi)) = \phi$. In particular, the natural maps $\Isomfib(\bar E) \to \Isom(\bar E)\to \QI(\bar E)$ are both injective.
\end{lemma}
\begin{proof}
By construction $\nu = \nu_\phi^a = \mathcal{A}(\phi)$ is the unique affine homeomorphism bounded distance from $f_0\circ \phi\vert_{E_0}$. As this map is itself affine, we have $\nu = f_0\circ \phi\vert_{E_0}$. The isometry $\Phi\colon D\to D$ in the construction of $\Psi(\nu)$ is then just the descent of $\phi$ to $D$. Further, for any $X,Y\in D$ we have $\phi\vert_X\circ f_{X,Y} = f_{\Phi(X),\Phi(Y)}\circ \phi\vert_{E_Y}$, since if $X$ lies at distance $t > 0$ from $Y$ along the geodesic from $\alpha^\perp$ to $\alpha$, then both maps send $(\alpha^\perp,\alpha)\mapsto (\partial \Phi(\alpha^\perp),\partial \Phi(\alpha))$ while contracting the first by $e^{-t}$ and expanding the second by $e^{t}$, hence they are the same affine map $E_Y\to E_{\Phi(X)}$. It follows that the restriction  $\Psi(\nu)\vert_{E_Y}\colon E_Y\to E_{\Phi(Y)}$ is then the composition
\begin{align*}
\Psi(\nu)\vert_{E_Y} &= f_{\Phi(Y), X_0}\circ (f_0\circ \phi\vert_{E_0})\circ f_0\vert_{E_Y}
= f_{\Phi(Y), \Phi(X_0)} \circ \phi\vert_{E_0}\circ f_{X_0,Y}\\
&= f_{\Phi(Y), \Phi(X_0)}\circ f_{\Phi(X_0),\Phi(Y)}\circ \phi\vert_{E_Y}
= \phi\vert_{E_Y}.
\end{align*}
Since this holds for each $Y$, we conclude $\Psi(\nu) = \phi$ as claimed. 

It follows that $\Isomfib(\bar E) \to \QI(\bar E)$ is injective, since if $[\phi]$ is the identity in $\QI(\bar E)$, meaning $\phi$ is finite distance from the identity, then $\mathcal A(\phi) = \mathcal A([\phi])$ and consequently $\phi = \Psi(\mathcal A(\phi))$ are both the identity. Finally, $\Isom(\bar E)\to \QI(\bar E)$ is injective since we have $\Isomfib(\bar E) = \Isom(\bar E)$ by \cite[Corollary~5.4]{DDLSI}.
\end{proof}

\subsection{Rigidity} 
\label{sec:finish_proof_of_rigidity}

We are now ready to complete the proof of Theorem~\ref{thm:qirigid intro}:

\begin{proof}[Proof of Theorem~\ref{thm:qirigid intro}]
By Lemma~\ref{lem:A_is_the_inverse_of_Psi}, the composition
\[\Aff(E_0) \stackrel{\Psi}{\to} \Isom_\fib(\bar E)\to \Isom(\bar E) \to \QI(\bar E)\cong \QI(\Gamma) \stackrel{\mathcal A}{\to} \Aff(E_0)\]
is the identity. Hence the first map $\Psi$ is injective, and the remaining maps are injective by Lemma~\ref{lem:Psi_is_the_inverse_of_A}
and Proposition~\ref{prop:identity_on_fibers} . It follows that each map above is an isomorphism, as claimed. The fact that $\Gamma$ has finite index in $\Isom(\bar E)\cong \QI(\Gamma)$ thus follows from Lemma~\ref{lem:Gamma_finite_index}.
\end{proof}

Standard techniques (see, for example \cite[\S10.4]{Schwartz-qi_classification}) now imply the following:

\begin{corollary} \label{C:qi rigidity classical}
If $H$ is any finitely generated group quasi-isometric to $\Gamma$, then $H$ and $\Gamma$ are weakly commensurable, meaning $H$ has a finite normal subgroup $N$ so that $H/N$ and $\Gamma$ contain finite index subgroups that are isomorphic.
\end{corollary}
This proof requires one more lemma.
\begin{lemma}
\label{lem:qi-uniformly_close_to_id}
 For every $K$ there exists $R'$ such that if $\phi\colon\bar E\to\bar E$ is a $(K,K)$-quasi-isometry that lies within finite distance of the identity, then $\phi$ lies within distance $R'$ of the identity, meaning $d(x,\phi(x))\le R'$ for all $x\in \bar E$.
\end{lemma}
\begin{proof}
First  define a map $\bar \phi\colon\bar D\to \bar D$ by setting $\bar \phi(X)=Y$, where $Y$ is the point provided by Proposition~\ref{prop:strip_bijection} such that $d_{\mathrm{Haus}}(\phi(E_X), E_Y) \le C$. Since $d_{\bar D}(X,Y)=d_{\bar E}(E_X,E_Y)$ for all $X,Y$ in $\bar D$, we see that $\bar \phi$ is a quasi-isometry with constants depending only on $K$. It also lies within finite distance of the identity, as it inherits this property from $\phi$; thus applying Lemma \ref{lem:hyper fact} to $Z = \bar D$ implies that $\bar \phi$ lies within uniformly finite distance of the identity. That is, there exists $R$ depending only on $K$ so that $d_{\mathrm{Haus}}(E_X, \Phi(E_X))\le R$ for all $X\in \bar D$.  Hence, Lemma \ref{L:fiber to fiber qi} implies that for each map $\nu_\phi^X=f_X\circ \phi|_{E_X}$ is a $(K',K')$--quasi-isometry for some $K'$ depending only on $K$.  Again by Lemma \ref{lem:hyper fact}, this time with $Z=E_X$, we see that each $\nu_\phi^X$ moves points uniformly bounded distance, and therefore $\phi\vert_{E_X}$ lies within uniform distance of the inclusion of $E_X$ in $\bar E$. Since this holds for all $X$, we have that $\phi$ lies within uniform distance of the identity, as required.
\end{proof}

\begin{proof}[Proof of Corollary~\ref{C:qi rigidity classical}]
If $H$ is quasi-isometric to $\Gamma$, there is a quasi-isometry $\mu\colon H\to \bar E$  with a quasi-inverse $\mu^{-1}\colon \bar E\to H$. Left multiplication by $h\in H$ gives an isometry $L_h\colon H\to H$. In this way, for each $h\in H$ we obtain a quasi-isometry $\mathcal{B}(h) = \mu\circ L_h \circ \mu^{-1}$ of $\bar E$ with \emph{uniformly bounded constants}. Let us also set $\mathcal{B}'(h)= \Psi(\mathcal{A}(\mathcal{B}(h)))\in \Isom_\fib(\bar E) = \Isom(\bar E)$, which is the unique isometry of $\bar E$ at finite distance from $\mathcal{B}(h)$. 
Since the quasi-isometry constants of $\mathcal{B}(h)$ are uniform, depending only on $\mu$, it follows from Lemma~\ref{lem:qi-uniformly_close_to_id} that there is a constant $R'$ so that $d(\mathcal{B}(h)(x), \mathcal{B}'(h)(x))\le R'$ for all $x\in \bar E$ and $h\in H$.

We now claim the homomorphism $\mathcal{B}'\colon H\to \Isom(\bar E)$ has finite kernel and cokernel. Indeed, if $\mathcal{B}'(h) = \mathrm{Id}_{\bar E}$ 
the above implies $\mathcal{B}(h)$ moves $\mu(e)$ (and in fact all points) distance at most $R'$. But this means $L_h$ moves the identity $e\in H$ uniformly bounded distance, and there are only finitely many such elements of $H$. 
To prove $\mathcal{B}'$ has finite cokernel it suffices,  as in Lemma~\ref{lem:Gamma_finite_index}, to show $\bar E / \mathcal{B}'(H)$ has finite volume or, better yet, finite diameter. For this, given $x,y\in \bar E$ we must find $h$ so that $d(\mathcal{B}'(h)(x),y)$ is uniformly bounded. This is equivalent to bounding $d(\mathcal{B}(h)(x),y) = d(\mu(h\cdot \mu^{-1}(x)),y)$, which is coarsely  $d_H(h\cdot \mu^{-1}(x), \mu^{-1}(y))$. Since $H$ acts transitively on itself, this is clearly possible.

We now see that $H/\ker(\mathcal{B}')$ and $\Gamma$ are both realized as finite index subgroups of $\Isom(\bar E)$ and hence that their intersection has finite index in both.
\end{proof}

\subsection{An alternative proof of quasi-isometric rigidity} \label{S:Mosher proof}

Here we sketch another proof of Theorem~\ref{T:qir simple}, following an approach described by Mosher in \cite{Mosher:Problems}; we refer the reader to that paper for a more detailed discussion.

First, we require some additional definitions.  Consider the maximal orbifold quotient $S \to \mathcal O$ so that $G$ descends to a group $G_{\mathcal O} < \Mod(\mathcal O)$ in the mapping class group of the orbifold $\mathcal O$; that is, $G$ consists of lifts of elements of $G_{\mathcal O}$.  Let $\frak C < \Mod(\mathcal O)$ be the {\em relative commensurator of $G_{\mathcal O}$ in $\Mod(\mathcal O)$}, which consists of the elements $g \in \Mod(\mathcal O)$ so that $g G_{\mathcal O} g^{-1} \cap G_{\mathcal O}$ has finite index in both $g G_{\mathcal O} g^{-1}$ and $G_{\mathcal O}$. (In fact, we must allow for orientation reversing mapping classes, but continue to denote this group $\Mod$ for simplicity.).  Finally, we let $\Gamma_{\frak C}$ denote the $\pi_1^{\rm{orb}}\mathcal O$--extension of $\frak C$.  Since $S \to \mathcal O$ is a finite sheeted cover, $\Gamma_{G_{\mathcal O}}$ contains $\Gamma$ with finite index, and so is quasi-isometric to it, and thus $\QI(\Gamma_{G_{\mathcal O}}) \cong \QI(\Gamma)$.  It is also not hard to see that there is a natural injection from $\Gamma_{\frak C} \to \QI(\Gamma_{G_{\mathcal O}}) \cong \QI(\Gamma)$.

We can now state Mosher's key reduction of quasi-isometric rigidity from \cite{Mosher:Problems}.

\begin{theorem}[Mosher] \label{T:Mosher's QI rigidity reduction} The homomorphism $\Gamma_{\frak C} \to \QI(\Gamma)$ is an isomorphism.
\end{theorem}

Very briefly, the proof of this theorem divides into two key steps.  First, $\Gamma_{\frak C}$ naturally maps not just to $\QI(\Gamma_{G_{\mathcal O}})$, but to the subgroup of coarsely fiber preserving quasi-isometries, $\QI_{\fib}(\Gamma_{G_{\mathcal O}}) < \QI(\Gamma_{G_{\mathcal O}})$.  In \cite{Mosher-fiber}, Mosher proves that this is in fact an isomorphism, $\Gamma_{\frak C} \cong \QI_{\fib}(\Gamma_{G_{\mathcal O}})$, whenever $G$ contains a pseudo-Anosov.  Second, a general result of Farb and Mosher \cite[Theorem 7.7(2)]{FarbMosher-abelian}, proved using coarse algebraic-topology, ensures that $\QI_{\fib}(\Gamma_{G_{\mathcal O}})$ has finite index in $\QI(\Gamma_{G_{\mathcal O}})$ when $G$ is further assumed to be virtually free.

With Theorem~\ref{T:Mosher's QI rigidity reduction} in hand, we see that proving quasi-isometric rigidity of $\Gamma$ reduces to proving that $\Gamma_{G_{\mathcal O}}$ in $\Gamma_{\frak C}$ has finite index (which is precisely what \cite[Problem~5.4]{Mosher:Problems} asks).   Equivalently, this reduces to the following.
\begin{lemma} The subgroup $G_{\mathcal O} < \frak C$ has finite index.
\end{lemma}
\begin{proof} Observe that the defining quadratic differential $q$ for $G$ descends to a quadratic differential $q_{\mathcal O}$ on $\mathcal O$ (with at worst simple poles at the orbifold points).  To see this, note that any pseudo-Anosov element $f \in G$ descends to a pseudo-Anosov element $f_0 \in \mathcal O$ (i.e.~$f$ is a lift of $f_0$).  The stable/unstable foliations for this pseudo-Anosov element $f$ are vertical/horizontal for (an affine deformation of) $q$, and these descend to stable/unstable foliations for $f_0$ which are thus vertical/horizontal for $q_{\mathcal O}$; thus, $q$ is the pull back of $q_{\mathcal O}$.  Since the fixed points of pseudo-Anosov elements of $G_{\mathcal O}$ are dense in $\mathbb P(q_{\mathbb O})$, it follows that $\frak C$ is contained in the stabilizer of $\mathbb P(q_{\mathbb O})$.  By \cite{masur:EG}, for example, $\mathbb P(q_{\mathcal O})$ determines the associated Teichm\"uller disk $\mathbb H_{q_{\mathcal O}}$ and thus the stabilizer of $\mathbb P(q_{\mathcal O})$ is the stabilizer of $\mathbb H_{q_{\mathcal O}}$, and is therefore the maximal Veech group defined by $q_{\mathcal O}$.  Since $G$ is a lattice Veech group, so is $G_{\mathcal O}$.  Therefore, the inclusion of $G_{\mathcal O}$ into this maximal Veech group of $q_{\mathcal O}$ must have finite index. Since $\frak C$ is a subgroup of this Veech group, $G_{\mathcal O} < \frak C$ is finite index as well.
\end{proof}


\begin{remark}Our proof here is similar to the proof for Kleinian groups that the relative commensurator of a (non-lattice, Zariski dense) geometrically finite Kleinian group is equal to the stabilizer of its limit set.  In fact, from \cite[Theorem~2.1]{KL:survey}, the limit set of $G_{\mathcal O}$ in $\PML(\mathcal O)$ is precisely $\mathbb P(q_{\mathcal O})$.
\end{remark}


\bibliographystyle{alpha}
\bibliography{main}

\begin{thebibliography}{BHMS20}

\bibitem[ABD21]{ABD:stable}
Carolyn Abbott, Jason Behrstock, and Matthew~Gentry Durham.
\newblock Largest acylindrical actions and {S}tability in hierarchically
  hyperbolic groups.
\newblock {\em Trans. Amer. Math. Soc. Ser. B}, 8:66--104, 2021.

\bibitem[ABO19]{ABO}
Carolyn Abbott, Sahana~H. Balasubramanya, and Denis Osin.
\newblock Hyperbolic structures on groups.
\newblock {\em Algebr. Geom. Topol.}, 19(4):1747--1835, 2019.

\bibitem[Ahl66]{Ahlfors}
Lars~V. Ahlfors.
\newblock Fundamental polyhedrons and limit point sets of {K}leinian groups.
\newblock {\em Proc. Nat. Acad. Sci. U.S.A.}, 55:251--254, 1966.

\bibitem[BBKL20]{BBKL-undistored_purely_pA}
Mladen Bestvina, Kenneth Bromberg, Autumn~E. Kent, and Christopher~J.
  Leininger.
\newblock Undistorted purely pseudo-{A}nosov groups.
\newblock {\em J. Reine Angew. Math.}, 760:213--227, 2020.

\bibitem[BHMS20]{comb_HHS}
Jason Behrstock, Mark~F. Hagen, A.~Martin, and A.~Sisto.
\newblock A combinatorial take on hierarchical hyperbolicity and applications
  to quotients of mapping class groups.
\newblock preprint, {\tt arXiv:2005.00567}, 2020.

\bibitem[BHS17a]{BHS:asdim}
Jason Behrstock, Mark~F. Hagen, and Alessandro Sisto.
\newblock Asymptotic dimension and small-cancellation for hierarchically
  hyperbolic spaces and groups.
\newblock {\em Proc. Lond. Math. Soc. (3)}, 114(5):890--926, 2017.

\bibitem[BHS17b]{BHS:HHS1}
Jason Behrstock, Mark~F. Hagen, and Alessandro Sisto.
\newblock Hierarchically hyperbolic spaces, {I}: {C}urve complexes for cubical
  groups.
\newblock {\em Geom. Topol.}, 21(3):1731--1804, 2017.

\bibitem[BHS19]{BHS:HHS2}
Jason Behrstock, Mark Hagen, and Alessandro Sisto.
\newblock Hierarchically hyperbolic spaces {II}: {C}ombination theorems and the
  distance formula.
\newblock {\em Pacific J. Math.}, 299(2):257--338, 2019.

\bibitem[BHS21]{BHS:quasiflats}
Jason Behrstock, Mark~F. Hagen, and Alessandro Sisto.
\newblock Quasiflats in hierarchically hyperbolic spaces.
\newblock {\em Duke Math. J.}, 170(5):--, 2021.

\bibitem[BL18]{BankLein}
Anja Bankovic and Christopher~J. Leininger.
\newblock Marked-length-spectral rigidity for flat metrics.
\newblock {\em Trans. Amer. Math. Soc.}, 370(3):1867--1884, 2018.

\bibitem[Bow93]{Bow:GF}
B.~H. Bowditch.
\newblock Geometrical finiteness for hyperbolic groups.
\newblock {\em J. Funct. Anal.}, 113(2):245--317, 1993.

\bibitem[Bow13]{Bow:coarse_median}
Brian~H. Bowditch.
\newblock Coarse median spaces and groups.
\newblock {\em Pacific J. Math.}, 261(1):53--93, 2013.

\bibitem[Bow18]{Bow:rigidity_MCG}
Brian~H. Bowditch.
\newblock Large-scale rigidity properties of the mapping class groups.
\newblock {\em Pacific J. Math.}, 293(1):1--73, 2018.

\bibitem[CLM12]{CLM:raag_in_MCG}
Matt~T. Clay, Christopher~J. Leininger, and Johanna Mangahas.
\newblock The geometry of right-angled {A}rtin subgroups of mapping class
  groups.
\newblock {\em Groups Geom. Dyn.}, 6(2):249--278, 2012.

\bibitem[DDLS21]{DDLSI}
Spencer Dowdall, Matthew~G. Durham, Christopher~J. Leininger, and Allesandro
  Sisto.
\newblock Extensions of {V}eech groups {I}: {A} hyperbolic action.
\newblock {\em arXiv preprint arXiv:2006.16425v2}, 2021.

\bibitem[DELS18]{DELS}
Moon Duchin, Viveka Erlandsson, Christopher~J. Leininger, and Chandrika
  Sadanand.
\newblock You can hear the shape of a billiard table: {S}ymbolic dynamics and
  rigidity for flat surfaces.
\newblock Preprint, to appear in {\em Comment. Math. Helv.}, 2018.

\bibitem[DGO17]{DGO}
F.~Dahmani, V.~Guirardel, and D.~Osin.
\newblock Hyperbolically embedded subgroups and rotating families in groups
  acting on hyperbolic spaces.
\newblock {\em Mem. Amer. Math. Soc.}, 245(1156):v+152, 2017.

\bibitem[DHS17]{DHS}
Matthew~Gentry Durham, Mark~F. Hagen, and Alessandro Sisto.
\newblock Boundaries and automorphisms of hierarchically hyperbolic spaces.
\newblock {\em Geom. Topol.}, 21(6):3659--3758, 2017.

\bibitem[DMS20]{DMS:stabcube}
Matthew~G Durham, Yair~N Minsky, and Alessandro Sisto.
\newblock Stable cubulations, bicombings, and barycenters.
\newblock {\em arXiv preprint arXiv:2009.13647}, 2020.

\bibitem[DT15]{DT15}
Matthew~Gentry Durham and Samuel~J. Taylor.
\newblock Convex cocompactness and stability in mapping class groups.
\newblock {\em Algebr. Geom. Topol.}, 15(5):2839--2859, 2015.

\bibitem[EF97]{EpsteinFujiwara}
David B.~A. Epstein and Koji Fujiwara.
\newblock The second bounded cohomology of word-hyperbolic groups.
\newblock {\em Topology}, 36(6):1275--1289, 1997.

\bibitem[Far98]{Farb:rel}
B.~Farb.
\newblock Relatively hyperbolic groups.
\newblock {\em Geom. Funct. Anal.}, 8(5):810--840, 1998.

\bibitem[FM00]{FarbMosher-abelian}
Benson Farb and Lee Mosher.
\newblock On the asymptotic geometry of abelian-by-cyclic groups.
\newblock {\em Acta Math.}, 184(2):145--202, 2000.

\bibitem[FM02a]{FM:CC}
Benson Farb and Lee Mosher.
\newblock Convex cocompact subgroups of mapping class groups.
\newblock {\em Geom. Topol.}, 6:91--152, 2002.

\bibitem[FM02b]{FM:SbF}
Benson Farb and Lee Mosher.
\newblock The geometry of surface-by-free groups.
\newblock {\em Geometric \& Functional Analysis GAFA}, 12(5):915--963, 2002.

\bibitem[GM91]{masur:EG}
F.P. Gardiner and H.A. Masur.
\newblock Extremal length geometry of {T}eichm\"uller space.
\newblock {\em Complex Variables Theory Appl.}, 16(2-3):209--237, 1991.

\bibitem[Gre66]{Greenberg}
L.~Greenberg.
\newblock Fundamental polyhedra for kleinian groups.
\newblock {\em Ann. of Math. (2)}, 84:433--441, 1966.

\bibitem[Ham]{hamenstadt:WHSG}
Ursula Hamenst{\"a}dt.
\newblock {Word hyperbolic extensions of surface groups}.
\newblock Preprint, \texttt{arXiv:math.GT/0505244}.

\bibitem[HHP20]{HHP:helly}
Thomas Haettel, Nima Hoda, and Harry Petyt.
\newblock Coarse injectivity, hierarchical hyperbolicity, and
  semihyperbolicity.
\newblock {\em arXiv preprint arXiv:2009.14053}, 2020.
\newblock To appear in \emph{Geom. Topol.}

\bibitem[HO13]{HullOsin}
Michael Hull and Denis Osin.
\newblock Induced quasicocycles on groups with hyperbolically embedded
  subgroups.
\newblock {\em Algebr. Geom. Topol.}, 13(5):2635--2665, 2013.

\bibitem[HRSS21]{SquidGames}
Mark~F. Hagen, Jacob Russell, Alessandro Sisto, and Davide Spriano.
\newblock Equivariant hierarchically hyperbolic structures for 3-manifold
  groups via quasimorphisms.
\newblock {\em In preparation}, 2021.

\bibitem[JS79]{JacoShalen}
William~H. Jaco and Peter~B. Shalen.
\newblock Seifert fibered spaces in {$3$}-manifolds.
\newblock {\em Mem. Amer. Math. Soc.}, 21(220):viii+192, 1979.

\bibitem[KL97]{KL:qi_decomposition}
Michael Kapovich and Bernhard Leeb.
\newblock Quasi-isometries preserve the geometric decomposition of {H}aken
  manifolds.
\newblock {\em Invent. Math.}, 128(2):393--416, 1997.

\bibitem[KL07]{KL:survey}
Autumn~E. Kent and Christopher~J. Leininger.
\newblock Subgroups of mapping class groups from the geometrical viewpoint.
\newblock In {\em In the tradition of {A}hlfors-{B}ers. {IV}}, volume 432 of
  {\em Contemp. Math.}, pages 119--141. Amer. Math. Soc., Providence, RI, 2007.

\bibitem[KL08a]{KL:CC}
Autumn~E. Kent and Christopher~J. Leininger.
\newblock Shadows of mapping class groups: capturing convex cocompactness.
\newblock {\em Geom. Funct. Anal.}, 18(4):1270--1325, 2008.

\bibitem[KL08b]{KL:uniform}
Autumn~E. Kent and Christopher~J. Leininger.
\newblock Uniform convergence in the mapping class group.
\newblock {\em Ergodic Theory Dynam. Systems}, 28(4):1177--1195, 2008.

\bibitem[KL12]{KL:qi_cohopf}
Ilya Kapovich and Anton Lukyanenko.
\newblock Quasi-isometric co-{H}opficity of non-uniform lattices in rank-one
  semi-simple {L}ie groups.
\newblock {\em Conform. Geom. Dyn.}, 16:269--282, 2012.

\bibitem[Kob12]{K:raag_in_MCG}
Thomas Koberda.
\newblock Right-angled {A}rtin groups and a generalized isomorphism problem for
  finitely generated subgroups of mapping class groups.
\newblock {\em Geom. Funct. Anal.}, 22(6):1541--1590, 2012.

\bibitem[Loa21]{Loa}
Christopher Loa.
\newblock Free products of abelian groups in mapping class groups.
\newblock {\em arXiv preprint arXiv:2103.05144}, 2021.

\bibitem[LR06]{Leininger-Reid}
C.~J. Leininger and A.~W. Reid.
\newblock A combination theorem for {V}eech subgroups of the mapping class
  group.
\newblock {\em Geom. Funct. Anal.}, 16(2):403--436, 2006.

\bibitem[Man05]{Man:bottleneck}
Jason~Fox Manning.
\newblock Geometry of pseudocharacters.
\newblock {\em Geom. Topol.}, 9:1147--1185, 2005.

\bibitem[Mar74]{Marden}
Albert Marden.
\newblock The geometry of finitely generated kleinian groups.
\newblock {\em Ann. of Math. (2)}, 99:383--462, 1974.

\bibitem[Mas70]{Maskit:bdy}
Bernard Maskit.
\newblock On boundaries of {T}eichm\"{u}ller spaces and on {K}leinian groups.
  {II}.
\newblock {\em Ann. of Math. (2)}, 91:607--639, 1970.

\bibitem[Min96a]{Minsky:product}
Yair~N. Minsky.
\newblock Extremal length estimates and product regions in {T}eichm\"{u}ller
  space.
\newblock {\em Duke Math. J.}, 83(2):249--286, 1996.

\bibitem[Min96b]{Minsky:quasiproj}
Yair~N. Minsky.
\newblock Quasi-projections in {T}eichm\"{u}ller space.
\newblock {\em J. Reine Angew. Math.}, 473:121--136, 1996.

\bibitem[MM99]{MM:CC1}
H.A. Masur and Y.N. Minsky.
\newblock Geometry of the complex of curves. {I}. {H}yperbolicity.
\newblock {\em Invent. Math.}, 138(1):103--149, 1999.

\bibitem[MM00]{MM:CC2}
H.A. Masur and Y.~N. Minsky.
\newblock Geometry of the complex of curves. {II}. {H}ierarchical structure.
\newblock {\em Geom. Funct. Anal.}, 10(4):902--974, 2000.

\bibitem[Mos03]{Mosher-fiber}
Lee Mosher.
\newblock Fiber respecting quasi-isometries of surface group extensions.
\newblock {\em preprint arXiv:math/0308067}, 2003.

\bibitem[Mos06]{Mosher:Problems}
Lee Mosher.
\newblock Problems in the geometry of surface group extensions.
\newblock In {\em Problems on mapping class groups and related topics},
  volume~74 of {\em Proc. Sympos. Pure Math.}, pages 245--256. Amer. Math.
  Soc., Providence, RI, 2006.

\bibitem[MS12]{MjSardar}
Mahan Mj and Pranab Sardar.
\newblock A combination theorem for metric bundles.
\newblock {\em Geom. Funct. Anal.}, 22(6):1636--1707, 2012.

\bibitem[MT02]{Masur-Tabachnikov}
Howard Masur and Serge Tabachnikov.
\newblock Rational billiards and flat structures.
\newblock In {\em Handbook of dynamical systems, {V}ol. 1{A}}, pages
  1015--1089. North-Holland, Amsterdam, 2002.

\bibitem[MW95]{masur:NH}
H.A. Masur and M.~Wolf.
\newblock {T}eichm\"uller space is not {G}romov hyperbolic.
\newblock {\em Ann. Acad. Sci. Fenn. Ser. A I Math.}, 20(2):259--267, 1995.

\bibitem[Raf14]{rafi:HT}
Kasra Rafi.
\newblock Hyperbolicity in {T}eichm\"{u}ller space.
\newblock {\em Geom. Topol.}, 18(5):3025--3053, 2014.

\bibitem[RST18]{RST18}
Jacob Russell, Davide Spriano, and Hung~Cong Tran.
\newblock Convexity in hierarchically hyperbolic spaces.
\newblock {\em arXiv preprint arXiv:1809.09303}, 2018.
\newblock To appear in \emph{Algebr. Geom. Topol.}

\bibitem[Run20]{runnels2020effective}
Ian Runnels.
\newblock Effective generation of right-angled artin groups in mapping class
  groups.
\newblock {\em arXiv preprint arXiv:2004.13585}, 2020.

\bibitem[Rus21]{russell2021extensions}
Jacob Russell.
\newblock Extensions of multicurve stabilizers are hierarchically hyperbolic.
\newblock {\em arXiv preprint arXiv:2107.14116}, 2021.

\bibitem[Sch95]{Schwartz-qi_classification}
Richard~Evan Schwartz.
\newblock The quasi-isometry classification of rank one lattices.
\newblock {\em Inst. Hautes \'{E}tudes Sci. Publ. Math.}, (82):133--168 (1996),
  1995.

\bibitem[Sis16]{S-Morse-hypemb}
Alessandro Sisto.
\newblock Quasi-convexity of hyperbolically embedded subgroups.
\newblock {\em Math. Z.}, 283(3-4):649--658, 2016.

\bibitem[Sis19]{HHS_survey}
Alessandro Sisto.
\newblock What is a hierarchically hyperbolic space?
\newblock In {\em Beyond hyperbolicity}, volume 454 of {\em London Math. Soc.
  Lecture Note Ser.}, pages 117--148. Cambridge Univ. Press, Cambridge, 2019.

\bibitem[Tan19]{tang2019affine}
Robert Tang.
\newblock Affine diffeomorphism groups are undistorted.
\newblock {\em preprint arXiv:1912.06537}, 2019.

\bibitem[Thu86]{thurston:GT}
W.P. Thurston.
\newblock Geometry and topology of $3$--manifolds.
\newblock Princeton University Lecture Notes, online at
  http://www.msri.org/publications/books/gt3m, 1986.

\end{thebibliography}

\end{document}